\documentclass[12pt]{amsart}
\usepackage[a4paper,left=2.25cm,right=2.25cm,top=3cm,bottom=3cm]{geometry}
\usepackage{amssymb}
\usepackage{kbordermatrix}
\usepackage{tikz}
\usetikzlibrary{calc,decorations.markings,positioning}
\usepackage{tikz-cd}
\usepackage{hyperref}

\makeatletter
\newcommand{\gettikzxy}[3]{
  \tikz@scan@one@point\pgfutil@firstofone#1\relax
\pgfmathsetmacro{#2}{\the\pgf@x/\linkpatternunit}
\pgfmathsetmacro{#3}{\the\pgf@y/\linkpatternunit}
}
\tikzset{label anchor/.code={%
    \let\tikz@auto@anchor=\pgfutil@empty
    \def\tikz@anchor{#1}
  },
  label anchor/.default=center
}
\makeatother
\tikzset{arrow/.style={postaction={decorate,thick,decoration={markings,mark = at position #1 with {\arrow{>}}}}},arrow/.default=0.5}
\newdimen\linkpatternunit%
\newcount\linkpatternsize%
\newcount\lpsize
\newif\iflinkpatterninverted
\newif\iflinkpatterntikzstarted%
\newif\iflinkpatternboxed
\newif\iflinkpatternaxis
\newif\iflinkpatternstraightlines
\newif\iflinkpatternnumbered
\newif\iflinkpatternalias
\newif\iflinkpatternnode
\newcount\linkpatternfused
\pgfkeys{/linkpattern/.cd,inverted/.is if=linkpatterninverted,numbered/.is if=linkpatternnumbered,tikzstarted/.is if=linkpatterntikzstarted,straight lines/.is if=linkpatternstraightlines,boxed/.is if=linkpatternboxed,axis/.is if=linkpatternaxis,vertexcolor/.store in=\linkpatternvertexcolor,edgecolor/.store in=\linkpatternedgecolor,boxcolor/.code={\linkpatternboxedtrue\def\linkpatternboxcolor{#1}},tikzoptions/.style={every linkpattern/.append style={#1}},size/.code={\linkpatternsize=#1},numbering0/.code={\def\lpnumbering{#1}\def\linkpatternnumbering{#1}},numbering/.style={numbered,numbering0={#1}},unit/.code={\linkpatternunit=#1},height/.store in=\linkpatternheight,shape/.store in=\linkpatternshape,looseness/.store in=\linkpatternlooseness,squareness/.store in=\linkpatternsquareness,extra space/.store in=\linkpatternextraspace,width/.store in=\linkpatternwidth,alias/.is if=linkpatternalias,pos/.store in=\linkpatternpos,labeloptions/.style={labeloptionslist/.append style={#1}},labeloptionslist/.style={inner sep=2pt,font=\scriptsize,execute at begin node=$,execute at end node=$,label anchor={#1+180}},nodeon/.is if=linkpatternnode,node/.style={nodeon,nodeoptionslist/.append style={#1}},nodeoptionslist/.style={},
pipedream/.style={shape=pipedream,looseness=0,straight lines,numbering0=tangle},
tangle/.style={shape=tangle,numbering0=tangle},
every linkpattern/.style={x=\linkpatternunit,y=\linkpatternunit,baseline=0},
fused/.code={\linkpatternfused=#1},
vertex/.style={circle,thin,draw=black,fill=\linkpatternvertexcolor,inner sep=1.5pt},
edge/.style={very thick,draw=\linkpatternedgecolor}}
\linkpatterninvertedfalse%
\linkpatternnumberedfalse%
\linkpatterntikzstartedfalse%
\linkpatternboxedfalse%
\linkpatternaxistrue%
\linkpatternaliastrue%
\linkpatternstraightlinesfalse%
\def\linkpatternlooseness{0.2}
\def\linkpatternsquareness{0.35}
\def\linkpatternvertexcolor{red}%
\def\linkpatternedgecolor{blue}%
\def\linkpatternboxcolor{none}%
\def\linkpatternheight{0}
\def\linkpatternwidth{0}
\def\linkpatternshape{default}
\def\linkpatternnumbering{default}
\def\linkpatternpos{(0,0)}
\def\linkpatternextraspace{0}
\linkpatternnodefalse
\def\firstchar#1#2\empty{#1}%
\def\linkpatterndo#1#2{
\edef\param{\csname linkpattern#2\endcsname}
\edef\firstcharparam{\expandafter\firstchar\param\empty}
\expandafter\ifcat\firstcharparam a
\expandafter\ifx\csname linkpattern#1\param\endcsname\relax
\csname linkpattern#1unknown\endcsname
\else
\csname linkpattern#1\csname linkpattern#2\endcsname\endcsname
\fi
\else
\csname linkpattern#1unknown\endcsname
\fi
}%
\def\linkpatterncoordtangle{\ifnum\x>\lphalfsize\pgfmathparse{\lpsize+1-\x}\xdef\lpcoordx{\pgfmathresult}\xdef\lpcoordy{\lpheight}\xdef\lpangle{270}\else\xdef\lpcoordx{\x}\xdef\lpcoordy{-\lpheight}\xdef\lpangle{90}\fi}
\def\linkpatterncoordpipedream{\ifnum\x>\lphalfsize\pgfmathparse{\lpsize+1-\x-0.5}\xdef\lpcoordx{\pgfmathresult}\xdef\lpcoordy{0}\xdef\lpangle{270}\else\pgfmathparse{0.5-\x}\xdef\lpcoordy{\pgfmathresult}\xdef\lpcoordx{0}\xdef\lpangle{0}\fi}
\def\linkpatterncoordrectangle{
\ifnum\x>\lptqsize
\pgfmathparse{\lpsize+1-\x-0.5}\xdef\lpcoordx{\pgfmathresult}\xdef\lpcoordy{0}\xdef\lpangle{270}
\else\ifnum\x>\lphalfsize
\pgfmathparse{\x-\lptqsize-0.5}\xdef\lpcoordy{\pgfmathresult}\xdef\lpcoordx{\linkpatternwidth}\xdef\lpangle{180}
\else\ifnum\x>\linkpatternheight
\pgfmathparse{\x-\linkpatternheight-0.5}\xdef\lpcoordx{\pgfmathresult}\xdef\lpcoordy{-\linkpatternheight}\xdef\lpangle{90}
\else
\pgfmathparse{0.5-\x}\xdef\lpcoordy{\pgfmathresult}\xdef\lpcoordx{0}\xdef\lpangle{0}
\fi\fi\fi
}%
\def\linkpatternsetsizeunknown{
\global\lpsize=\linkpatternsize
\if\linkpatternheight0
\xdef\maxsep{0}
\foreach \x/\xx in \mylist%
{%
\edef\tempx{\withoutprime{\x}}
\edef\tempxx{\withoutprime{\xx}}
\pgfmathparse{max(\maxsep,abs(\tempx-\tempxx))}
\xdef\maxsep{\pgfmathresult}
}%
\pgfmathparse{0.25+0.8*\linkpatternsquareness*\maxsep}
\xdef\lpheight{\pgfmathresult}
\else
\xdef\lpheight{\linkpatternheight}
\fi
}
\newcount\tempsize
\def\linkpatternrightmostunknown{
\global\lpsize=0
\global\tempsize=0
\foreach\x/\labx in \linkpatternnumbering
{
\edef\tempx{\withoutprime{\x}}
\ifnum\lpsize<\tempx\global\lpsize=\tempx\fi
\global\advance\tempsize by 1
}
\ifnum\tempsize>\lpsize\global\lpsize=\tempsize\fi
}%
\def\linkpatternrightmostdefault{
\global\lpsize=0
\global\tempsize=0
\foreach \x/\y in \mylist
{
\edef\tempx{\withoutprime{\x}}
\ifnum\lpsize<\tempx\global\lpsize=\tempx\fi
\ifx\x\y
\global\advance\tempsize by 1
\else
\edef\tempy{\withoutprime{\y}}
\ifnum\lpsize<\tempy\global\lpsize=\tempy\fi%
\global\advance\tempsize by 2
\fi
}
\ifnum\tempsize>\lpsize\global\lpsize=\tempsize\fi
}%
\def\linkpatternrightmosttangle{
\global\lpsize=0
\global\tempsize=0
\foreach \x/\y in \mylist
{
\edef\tempx{\withoutprime{\x}}
\ifnum\lpsize<\tempx\global\lpsize=\tempx\fi
\ifx\x\y
\global\advance\tempsize by 1
\else
\edef\tempy{\withoutprime{\y}}
\ifnum\lpsize<\tempy\global\lpsize=\tempy\fi%
\global\advance\tempsize by 2
\fi
}
\global\advance\lpsize by\lpsize
\ifnum\tempsize>\lpsize\global\lpsize=\tempsize\fi
}%

\newcommand\linkpattern[2][]{
{
\pgfkeys{/linkpattern/.cd,#1}
\edef\mylist{#2}
\def\primetest##1'{}%
\def\hasaprime##1{\expandafter\primetest##1''}
\def\internalwithoutprime##1'{##1}%
\def\withoutprime##1{\if\hasaprime##1 %
\expandafter\internalwithoutprime##1\else ##1\fi}%
\iflinkpatternnumbered%
\iflinkpatterninverted
\tikzset{/linkpattern/lbl/.style n args={3}{label={[/linkpattern/labeloptionslist=-##1,##3] ##1:##2}}}
\else
\tikzset{/linkpattern/lbl/.style n args={3}{label={[/linkpattern/labeloptionslist=##1,##3] ##1:##2}}}
\fi
\else%
\tikzset{/linkpattern/lbl/.style={}}
\fi%
\iflinkpatterntikzstarted
\begin{scope}[/linkpattern/every linkpattern,shift=\linkpatternpos]
\else%
\begin{tikzpicture}[/linkpattern/every linkpattern]%
\fi%
\begin{scope}[local bounding box=link pattern box]
\iflinkpatterninverted%
\begin{scope}[yscale=-1]%
\fi%
\linkpatterndo{setsize}{shape}
\ifnum\lpsize=0
\linkpatterndo{rightmost}{numbering}
\fi
\pgfmathtruncatemacro{\lphalfsize}{\lpsize/2}
\linkpatterndo{numbering}{numbering}
\iflinkpatternboxed
\linkpatterndo{drawbox}{shape}
\else
\iflinkpatternaxis
\linkpatterndo{drawaxis}{shape}
\fi
\fi
\foreach\xx/\xlab/\opt in \lpnumbering
{
\ifx\xlab\opt\def\opt{}\fi
\if\hasaprime\xx %
\pgfmathtruncatemacro{\xx}{\lpsize+1-\withoutprime{\xx}}
\fi
\ifnum\linkpatternfused>1
\pgfmathsetmacro{\x}{0.4*(0.5+\linkpatternfused*(0.5+floor((\xx-1)/\linkpatternfused)))+0.6*\xx}
\else
\def\x{\xx}
\fi
\linkpatterndo{coord}{shape}
\iflinkpatternalias\def\xlabb{\xlab}\else\def\xlabb{\xx}\fi
\path (\lpcoordx,\lpcoordy) coordinate[transform shape,/linkpattern/vertex,/linkpattern/lbl={\lpangle+180}{\xlab}{\opt},alias=v\xlabb] (v\xx) ++(\lpangle:\linkpatternunit) coordinate[alias=vv\xlabb] (vv\xx); 
}
\foreach \a/\b/\c in \mylist
{
\if\hasaprime\a %
\pgfmathtruncatemacro{\a}{\lpsize+1-\withoutprime{\a}}
\fi
\draw[/linkpattern/edge]
\ifx\b\c
\else
\c
\fi
\ifx\a\b
(v\a) -- ++(0,\lpheight);
\else
\pgfextra{
\if\hasaprime\b %
\pgfmathtruncatemacro{\b}{\lpsize+1-\withoutprime{\b}}
\fi
\gettikzxy{(v\a)}{\ax}{\ay}
\gettikzxy{(v\b)}{\bx}{\by}
\gettikzxy{(vv\a)}{\axx}{\ayy}
\gettikzxy{(vv\b)}{\bxx}{\byy}
\pgfmathsetmacro{\dist}{sqrt((\ax-\bx)*(\ax-\bx)+(\ay-\by)*(\ay-\by))}
\pgfmathsetmacro{\abx}{(\axx-\ax)*\dist*\linkpatternsquareness+(\bx-\ax)*\linkpatternlooseness)}
\pgfmathsetmacro{\aby}{(\ayy-\ay)*\dist*\linkpatternsquareness+(\by-\ay)*\linkpatternlooseness)}
\pgfmathsetmacro{\bax}{(\bxx-\bx)*\dist*\linkpatternsquareness+(\ax-\bx)*\linkpatternlooseness)}
\pgfmathsetmacro{\bay}{(\byy-\by)*\dist*\linkpatternsquareness+(\ay-\by)*\linkpatternlooseness)}
}
(v\a) 
\iflinkpatternstraightlines
\pgfextra{
\pgfmathsetmacro{\t}{((\ax-\bx)*\bay-(\ay-\by)*\bax)/(\aby*\bax-\abx*\bay)}
\pgfmathsetmacro{\abx}{\t*\abx}
\pgfmathsetmacro{\aby}{\t*\aby}
}
[rounded corners] -- ++(\abx,\aby) -- (v\b);
\else
.. controls ++(\abx,\aby) and ++(\bax,\bay) .. 
\fi
(v\b);
\fi
}
\end{scope}
\iflinkpatternnode
\node[fit=(link pattern box),/linkpattern/nodeoptionslist] {};
\fi
\iflinkpatterninverted
\end{scope}
\fi
\iflinkpatterntikzstarted
\end{scope}
\else%
\end{tikzpicture}%
\fi%
}}%
\newcommand\tanglelinkpattern[3][]{%
{
\pgfkeys{/linkpattern/.cd,#1}
\begin{tikzpicture}[/linkpattern/every linkpattern,baseline=-\linkpatternunit]%
\linkpattern[#1,tikzstarted,numbered=false]{#3}
\pgfmathtruncatemacro{\lptempsize}{2*\linkpatternsize}
\iflinkpatterninverted
\begin{scope}[yshift=0.5*\linkpatternunit]
\else
\begin{scope}[yshift=-0.5*\linkpatternunit]
\fi
\linkpattern[tangle,#1,tikzstarted,size=\lptempsize,
numbering=halftangle,
height=0.5]{#2}
\end{scope}
\end{tikzpicture}%
}}
\newcommand\diag[4][]{%
\pgfkeys{/linkpattern/.cd,#1}
\iflinkpatterntikzstarted\else%
\begin{tikzpicture}[scale=0.5]
\fi%
\iflinkpatterninverted%
\begin{scope}[yscale=-1]%
\fi%
\draw (0,0) grid (#2,#3);
\edef\mylist{#4}
\foreach\y/\x/\z in \mylist
{
\ifx\x\z
\draw[decorate,decoration={zigzag,
amplitude=1pt,segment length=5pt}]
(\x-0.5,#3) -- (\x-0.5,\y-0.5) node[circle,fill=black,inner sep=2pt] {} -- (#2,\y-0.5);
\else
\node at (\x-0.5,\y-0.5) {$\z$};
\fi
}
\iflinkpatterninverted
\end{scope}
\fi
\iflinkpatterntikzstarted\else%
\end{tikzpicture}%
\fi%
}
\makeatletter
\tikzset{circle split part fill/.style  args={#1,#2}{%
 alias=tmp@name,
  postaction={%
    insert path={
     \pgfextra{%
     \pgfpointdiff{\pgfpointanchor{\pgf@node@name}{center}}%
                  {\pgfpointanchor{\pgf@node@name}{east}}%
     \pgfmathsetmacro\insiderad{\pgf@x}
      \fill[#1] (\pgf@node@name.base) ([xshift=-\pgflinewidth]\pgf@node@name.east) arc
                          (0:180:\insiderad-\pgflinewidth)--cycle;
      \fill[#2] (\pgf@node@name.base) ([xshift=\pgflinewidth]\pgf@node@name.west)  arc
                           (180:360:\insiderad-\pgflinewidth)--cycle;                    }}}}}  
 \makeatother
\tikzset{bdot/.style={circle,circle split,draw,circle split part fill={black,white},thin,inner sep=1pt}}%
\tikzset{wdot/.style={circle,circle split,draw,circle split part fill={white,black},thin,inner sep=1pt}}%
\newcommand\circlelinkpattern[2][]{
{
\pgfkeys{/linkpattern/.cd,#1}
\iflinkpatterntikzstarted\else%
\begin{tikzpicture}[/linkpattern/every linkpattern]%
\fi%
\iflinkpatterninverted%
\begin{scope}[yscale=-1]%
\fi%
\global\lpsize=\linkpatternsize
\edef\mylist{#2}
\foreach \x/\y in \mylist
{
\ifnum\x>\lpsize\global\lpsize=\x\fi
\ifnum\y>\lpsize\global\lpsize=\y\fi
}
\iflinkpatternaxis
\draw (0,0) circle (1);
\fi
\foreach\x in {1,...,\lpsize}
{
\pgfmathparse{(0.3*floor((\x-1)/\linkpatternfused)+0.7*((\x-0.5)/\linkpatternfused-0.5))*\linkpatternfused*360/\lpsize}
\coordinate[/linkpattern/vertex] (v\x) at (\pgfmathresult:1);
}
\foreach \x/\y/\z in \mylist
{
\ifx\y\z%
\draw[/linkpattern/edge] (v\x) .. controls ($0.5*(v\x)$) and  ($0.5*(v\y)$) .. (v\y);
\else
\draw[/linkpattern/edge] \z (v\x) .. controls ($0.5*(v\x)$) and  ($0.5*(v\y)$) .. (v\y);
\fi
}
\iflinkpatternnumbered%
\pgfmathparse{\lpsize/\linkpatternfused}
\global\lpsize=\pgfmathresult
\def\linkpatternnumbering{1,...,\lpsize}
\newdimen\angle
\foreach\x/\xx/\opt in \linkpatternnumbering
{
  \pgfmathsetmacro{\angle}{360/\lpsize*(\x-1)}
\ifx\xx\opt%
  \node[outer sep=1pt,anchor=180+\angle] at (\angle:1) {$\scriptstyle\xx$}; 
\else
  \node[outer sep=1pt,anchor=180+\angle,\opt] at (\angle:1) {$\scriptstyle\xx$}; 
\fi
}
\fi%
\iflinkpatterninverted%
\end{scope}
\fi%
\iflinkpatterntikzstarted\else%
\end{tikzpicture}%
\fi%
}}%
\newdimen{\loopcellsize}
\tikzset{bgplaq/.style={draw=black,fill=\linkpatternboxcolor}}
\def\plaqwest{}
\def\plaqeast{}
\def\plaqnorth{}
\def\plaqsouth{}
\pgfkeys{/linkpattern/.cd,west/.store in=\plaqwest,east/.store in=\plaqeast,north/.store in=\plaqnorth,south/.store in=\plaqsouth}
\newcommand\plaq[2][]{
\pgfkeys{/linkpattern/.cd,#1}
\begin{scope}[x=\loopcellsize,y=\loopcellsize]
\draw[bgplaq,use as bounding box] (-0.5,-0.5) rectangle ++(1,1);
\csname plaq#2\endcsname
\end{scope}
}
 \tikzset{loop/.style={matrix,row sep={\loopcellsize,between origins},column sep={\loopcellsize,between origins}}}

\newcommand\MMN{{\mathcal M}_N}
\newcommand\A{\textsc{a}}
\newcommand\B{\textsc{b}}

\newcommand\divides{\,{\mid}\,}
\newcommand\lie[1]{{\mathfrak{#1}}}
\newcommand\lien{{R_N^{\Delta=0}}}

\newcommand\iso{{\ \cong\ }}

\newcommand\calO{{\mathcal O}}
\newtheorem{Theorem}{Theorem} 
\newtheorem{Proposition}{Proposition} 
\newtheorem{Lemma}{Lemma}

\newtheorem{Corollary}{Corollary}
\newtheorem*{Corollary*}{Corollary}
 
\newtheorem*{Theorem*}{Theorem}
\theoremstyle{remark}

\newcommand\onto{\mathop{\twoheadrightarrow}}
\newcommand\into{\operatorname*{\hookrightarrow}}
\newcommand\Ex{\noindent{\em Example. }}
\newcommand\union{\bigcup}

\newcommand\Id{{\bf 1}}

\newcommand\reals{{\mathbb R}}
\newcommand\complexes{{\mathbb C}}
\newcommand\integers{{\mathbb Z}}

\newcommand\onehalf{\frac{1}{2}}
\newcommand\naturals{{\mathbb N}}

\newcommand\ol{\overline}

\newcommand\<{\langle}
\renewcommand\>{\rangle}
\newcommand\junk[1]{}

\newcommand\Coprod{\coprod}

\theoremstyle{plain}

\newcommand\dfn{\bf} 

\newcommand\GLN{{{GL_N}}}
\newcommand\glN{\lie{gl}_N}

\newcommand\der{\partial}
\newcommand\mdeg{{\rm m}\!\deg}
\newcommand\codim{{\rm co}\!\dim}
\newcommand\SN{\mathcal{S}_N}
\newcommand\MNC{{\text{Mat}_N}}
\newcommand\MnC{{M_n}}
\newcommand\rank{{\rm rank\ }}
\newcommand{\comment}[1]{$\star${\sf\textbf{#1}}$\star$}
\newcommand\Br{\mathcal{B}r}
\numberwithin{equation}{section}

\begin{document}

\title{The Brauer loop scheme and orbital varieties}
\author{Allen Knutson}
\address{Allen Knutson, Cornell, Ithaca, NY 14853, USA.}
\email{allenk@math.cornell.edu}
\thanks{AK was supported by NSF grant 0303523.}
\author{Paul Zinn-Justin}
\address{Paul Zinn-Justin, 
UPMC Univ Paris 6, CNRS UMR 7589, LPTHE,
75252 Paris Cedex, France}
\email{pzinn@lpthe.jussieu.fr}
\thanks{PZJ was supported by 
EU Marie Curie Research Training Networks ``ENRAGE'' MRTN-CT-2004-005616, 
ANR programs ``GIMP'' ANR-05-BLAN-0029-01 
and ERC grant ``LIC'' 278124.}
\thanks{The authors would like to thank P.~Di~Francesco for his participation
in the early stages of this project.}
\date{\today}

\begin{abstract}
  A. Joseph invented multidegrees in \cite{Jo} 
  to study {\dfn orbital varieties}, which are the components of an
  {\em orbital scheme}, itself constructed by intersecting a nilpotent orbit
  with a Borel subalgebra. Their multidegrees
  are known as {\em Joseph polynomials}, and these polynomials
  give a basis of a (Springer) representation of the Weyl group.
  In the case of the nilpotent orbit $\{M^2=0\}$, the orbital varieties
  can be indexed by noncrossing chord diagrams in the disc. 
  
  In this paper we study the {\em normal cone\/} to the orbital scheme 
  inside this nilpotent orbit $\{ M^2 = 0 \}$. 
  This gives a better-motivated construction of the {\em Brauer loop scheme\/}
  we introduced in \cite{KZJ}, whose components are indexed by all
  chord diagrams (now possibly with crossings) in the disc.

  The multidegrees of its components, the {\em Brauer loop varieties}, 
  were shown to reproduce
  the ground state of the {\em Brauer loop model\/} in statistical mechanics
  \cite{DFZJ06}. Here, we reformulate and slightly generalize these
  multidegrees in order to express them as solutions of the
  rational quantum Knizhnik--Zamolodchikov equation associated to the
  Brauer algebra.
  In particular, the vector of the multidegrees satisfies
  two sets of equations, corresponding to the $e_i$ and $f_i$ generators of
  the Brauer algebra.
  The proof of the analogous statement in \cite{KZJ} was slightly
  roundabout; we verified the $f_i$ equation using the geometry of
  multidegrees, and used algebraic results of \cite{DFZJ06} to show
  that it implied the $e_i$ equation. We describe here the geometric
  meaning of both $e_i$ and $f_i$ equations in our slightly
  extended setting.

  We also describe the corresponding actions at the level of orbital
  varieties:
  while only the $e_i$ equations make sense directly
  on the Joseph polynomials, 
  the $f_i$ equations also appear if one introduces
  a broader class of varieties. We explain the connection of the latter with
  matrix Schubert varieties.

\end{abstract}

\maketitle

{ \small\tableofcontents}
%
%

\section{Introduction}\label{sec:intro}

The aim of the present paper is to give a full account of the
connection between certain geometric objects and quantum integrable
models, following the ideas presented in \cite{DFZJ06,KZJ,DFZJ05b}.
In \S \ref{ssec:brbr}--\ref{ssec:ncone} we recall the {\em Brauer loop
scheme\/} we introduced in \cite{KZJ}, and give a new interpretation of it
in terms of the more common notion of {\em orbital varieties}. 
In \S\ref{ssec:qKZ} we discuss the {\em quantum Knizhnik--Zamolodchikov
equation\/}, and its relation to a refinement of the
Brauer model studied in \cite{DFZJ06}. 
{\em Multidegrees},
whose definition we recall in \S\ref{ssec:multidegrees},
allow us to connect the two, as explained in \S\ref{ssec:summary},
where we summarize our main results.
All these points are then developed in the rest of the paper.

\subsection{The Brauer loop scheme}\label{ssec:brbr}
Throughout this paper, the base field will always be $\complexes$.
Let $\text{Mat}_\integers$ (resp.\ $R_\integers$) denote the 
vector space of matrices (resp.\ upper triangular) 
matrices with rows and columns indexed by $\integers$. 
Despite the infinitude, any matrix entry in a product $X\cdot Y$,
$X,Y\in R_\integers$, is
a sum of finitely many nonzero terms, so $R_\integers$ is an algebra.
Let $\text{Mat}_{\integers \bmod N} \leq \text{Mat}_\integers$ 
denote the space of matrices with the periodicity
\begin{equation*}
X_{ij} = X_{i+N,j+N} \quad\forall i,j \in \integers;
\end{equation*}
and $R_{\integers \bmod N}$ be the corresponding subalgebra of $R_\integers$:
$R_{\integers\bmod N}=R_\integers \cap \text{Mat}_{\integers\bmod N}$.
If $\MNC$ (resp.\ $R_N$) is the space of ordinary $N\times N$ matrices
(resp.\ upper triangular matrices),
then a typical element of $R_{\integers\bmod N}$ looks like
\begin{equation*}
\begin{matrix}
 \ddots& \ddots &\ddots&\ddots\\
 & & U & L & Z &\ldots  \\
  &&& U & L & Z &\ldots  \\
  &&&& U & L & Z &\ldots \\
  &&0&&& U & L &\ldots \\
  &&&&&& U &\ldots \\
  &&&&&& \ddots
\end{matrix}
\qquad\qquad\qquad
{ U\in R_N \atop L,Z,\ldots \in \MNC}
\end{equation*}
where $U$, on the main diagonal, is upper triangular.
This subalgebra contains the ``shift'' matrix $S$ carrying $1$s just above 
the main diagonal, $S_{ij} = \delta_{i,j-1}$.

Any $M$ in the quotient algebra 
\begin{equation*}
 \MMN := R_{\integers \bmod N} / \< S^N \> 
\end{equation*}
is determined by the entries $M_{ij}$ with $0 \leq i,j-i < N$,
and this algebra is finite dimensional of dimension $N^2$.
(In terms of the picture above, only $U$ and the strict lower triangle of $L$ 
remain well-defined in the quotient. We will use often this splitting
of $\MMN$.)
As explained in \cite{KZJ}, this solvable algebra $\MMN$ is a degenerate
limit of the usual matrix algebra $\MNC$.

\newcommand\wtM{{\widetilde M}}
\newcommand\wtN{{\widetilde N}}

Define the {\dfn Brauer loop scheme} $E \subseteq \MMN$ as the space
of strictly upper triangular matrices $M$ whose square is ``zero'',
i.e., $M^2 \in \< S^N \>$. We will occasionally find it useful to
choose lifts $\wtM \in R_{\integers \bmod N}$ of its elements.
We introduced this scheme in \cite{KZJ}, under a different but 
equivalent definition.

This scheme $E$ is reducible, and we call its top-dimensional components
the {\dfn Brauer loop varieties}. They are naturally indexed by 
{\dfn link patterns} $\pi$ (meaning, involutions of $\integers/N\integers$
with at most one
fixed point), as we now recall from \cite{KZJ}. First, note that
if we pick a representative $\wtM \in R_{\integers \bmod N}$ lying over 
some $M\in E$, then for any
$i \in \integers / N \integers$, 
\begin{align*}
  (\wtM^2)_{i,i+N} &= \sum_j \wtM_{i,j}\, \wtM_{j,i+N} \\
  &= \sum_{\scriptstyle j\atop\scriptstyle i<j<i+N} M_{i,j}\, M_{j,i+N} 
  && \text{since $\wtM$ has zero diagonal,}
\end{align*}
so that $(\wtM^2)_{i,i+N}$ is well-defined (independent of the representative)
even though $\wtM_{i,i+N}$ itself is not because of
the $\< S^N \>$ ambiguity. By abuse of notation,
we shall simply denote it by $(M^2)_{i,i+N}$
in what follows.

\begin{Theorem}[Theorems 2 and 3 of \cite{KZJ}]
  \label{thm:Ecomps} \quad
  \begin{enumerate}
  \item
    Let $M \in E$.\\
    Then there exists a link pattern $\pi$ such that
    $(M^2)_{i,i+N} = (M^2)_{\pi(i),\pi(i)+N}$.
  \item For generic points of the Brauer loop scheme $E$,
    there are no additional equalities among 
    the $\{ (M^2)_{i,i+N} \}$
(i.e., the set $\{(M^2)_{i,i+N}\}$ is of cardinality $\lceil N/2 \rceil$);
so that the link pattern $\pi$ in {\rm(1)} is unique.

  \item The map 
    \begin{equation*} \{\text{top-dimensional components of $E$}\} \longrightarrow 
    \{\text{link patterns}\}, \end{equation*}
    which takes a component to the link pattern of a generic point therein,
    is a bijection.
  \end{enumerate}
\end{Theorem}

Brian Rothbach has shown that $E$ is equidimensional \cite{rothbach},
which allows one to drop ``top-dimensional'' from the above statement.

To a link pattern $\pi$ associate $\underline\pi\in E_\pi$ which is
the matrix with 1's at entries $(i,\pi(i))$ with representatives
mod $N$ such that $i<\pi(i)<i+N$, and 0's elsewhere.
We shall need the following (Theorem 5 of \cite{KZJ}). 

\begin{Theorem}\label{thm:compeqns}
  The irreducible component $E_\pi$ of $E$ corresponding to the link pattern
$\pi$ satisfies the following equations:
  \begin{enumerate}
  \item $M^2 = 0$.
  \item $(M^2)_{i,i+N} = (M^2)_{\pi(i),\pi(i)+N}$.
  \item For any matrix entry $(i,j)$, $i<j<i+N$,
    we have $r_{ij}(M) \leq r_{ij}(\underline\pi)$,
where $r_{ij}$ denotes the rank of the submatrix south-west of entry $(i,j)$. 
In polynomial terms, this 
    asserts the vanishing of all minors of size $r_{ij}(\underline\pi)+1$
    in the submatrix southwest of entry $(i,j)$.
  \end{enumerate}
\end{Theorem}

In fact we conjectured in \cite{KZJ} that
the equations of Thm.~\ref{thm:compeqns} define $E_\pi$ as a scheme.

\subsection{The orbital varieties of \texorpdfstring{$M^2=0$}{M2=0}}\label{ssec:orbvars}
\newcommand\OO{{\mathcal O}}

Let $D \subseteq \MNC$ denote the closure of some conjugacy class of
nilpotent matrices (though we will soon specialize to $D = \{ M : M^2 = 0\}$)
and recall that $R_N$ is the space of upper triangular matrices.
In this generality, the intersection $D \cap R_N$ is called the
{\dfn orbital scheme} of $D$, and its geometric components are called
the {\dfn orbital varieties}. (Conventions differ about whether $D$ should
be a nilpotent orbit or its closure; these issues will not be relevant
for us and we will always take $D$, and its orbital varieties, to be closed.) 
The orbital scheme carries an action by conjugation of $B_N=R_N^\times$,
the group of invertible upper triangular matrices; hence each orbital
variety carries such an action too.

The orbital varieties were shown in \cite{Sp} to all have the same 
dimension $\onehalf \dim D$, and to be naturally indexed by
standard Young tableaux, on the partition determined by the Jordan 
canonical form of generic elements of $D$. 

In this paper we will only be interested in the case $D = \{ M : M^2 = 0\}$
in even dimensions $N=2n$. So the relevant partition is $(2,2,\ldots,2)$,
and the standard Young tableaux correspond in a simple way with
{\em noncrossing\/} chord diagrams. (We will correspond orbital varieties
with noncrossing chord diagrams directly in 
\S\ref{sec:orbvars}, and make no use of Young tableaux in
this paper. There are a multitude of other interpretations of this
Catalan number in \cite{Stanley}.)

The space of matrices is $(2n)^2$-dimensional, with $\dim D$ half that,
and $\dim (D\cap R_N)$ half that again, so $n^2$.
One reason this nilpotent orbit is easier
to deal with than a general one 
is that it is {\dfn spherical}: it has only finitely many
$B_N$-orbits. In particular, each component of $D\cap R_N$ is a $B_N$-orbit closure.
The set of orbits was described in \cite{M1}; we give a new way 
to index the orbits in \S\ref{sec:orbvars}. 

The Brauer loop scheme has two simple connections
to the orbital scheme $D\cap R_N$, for which we will give
deeper reasons in \S\ref{ssec:ncone}. There is an inclusion
\begin{equation*}
 R_N \to \MMN, \qquad
U \mapsto M \quad \text{ where } \quad M_{ij} = 
\begin{cases}
  U_{ij} & 1\leq i \leq j \leq N \\
  0 & 1\leq i \leq N < j   
\end{cases}
\end{equation*}
that takes $D\cap R_N \into E$, and a projection
\begin{equation*}
\MMN \to R_N, \qquad
M \mapsto U  \quad\text{ where }\quad U_{ij} = M_{ij},\ 1\leq i\le j \leq N 
\end{equation*}
that takes $E \onto D\cap R_N$. 
(In the $(U,L)$ notation from \S\ref{ssec:brbr}, 
the two maps are $U \mapsto (U,0)$, $(U,L) \mapsto U$.)
Moreover, the composite $D\cap R_N \into E \onto D\cap R_N$ 
of these two maps is the identity.

\newcommand\Spec{{\rm Spec\ }}
\newcommand\gr{{\rm gr\ }}

\subsection{The normal cone to the orbital scheme}\label{ssec:ncone}

The connection between the Brauer loop scheme and the orbital scheme
is tighter than just indicated, as already implicit in \cite{KZJ}, and as
will be discussed in detail in \S\ref{sec:ncone}. 
We first recall the definition of the {\dfn normal cone} $C_X Y$ to a
subscheme $X \subseteq Y$, both schemes affine.
Say that $Y = \Spec A$ and $X$ is cut out of $Y$ by 
the vanishing of an ideal $I \leq A$. Then $A$ is filtered by powers of
the ideal, $A \geq I \geq I^2 \geq \ldots,$ and $C_X Y$ is defined as the
$\text{Spec}$ of the associated graded algebra
$\gr A := (A/I) \oplus (I/I^2) \oplus (I^2/I^3) \oplus \ldots$.
Note that while there is no natural map $Y\onto X$ reversing the
inclusion $X\to Y$, there is a natural map $C_X Y\onto X$ reversing
a natural inclusion $X\into C_X Y$.

When $X$ and $Y$ are smooth, the projection $C_X Y \onto X$ is a vector bundle,
and the components of $C_X Y$ correspond $1:1$ to the components of $X$.
More generally, if $X^\circ,Y^\circ$ denote the smooth loci of $X$ and $Y$,
respectively, then
$C_X Y$ contains $C_{X^\circ} Y^\circ$ as an open subset, and 
the components of $C_{X^\circ} Y^\circ$ correspond $1$:$1$ to the 
(generically reduced geometric)
components of $X$. However, $C_{X^\circ} Y^\circ$ may miss some of
the components of $C_X Y$, and one can take this as a measure of 
the nonsmoothness of the embedding $X\into Y$.

\junk{
Though we won't use it, we suggest a physical description
(which we learned in a lecture of E.~Witten)
to help understand the role of the ``extra components'' of the normal cone.
Imagine that $Y$ is a phase space carrying a Hamiltonian $H+h$, 
where $H\geq 0$ grows very quickly and $h$ does not.  Then a particle
traveling on $Y$ is likely to stay near the $H=0$ locus, which we call $X$. 
So at low energies, one can generally (really, on $X^\circ$) 
treat the system as having the phase space $X$ with Hamiltonian $h$, 
plus small oscillations in the normal directions to $X^\circ$. 
However, where $H$ vanishes to high order, one can get further off $X$
without high energies, and detect more of the geometry of $Y$; 
this is exactly where $X$ is singular. Having ``extra'' components of the 
normal cone is a sign that $(X,h)$ is not a good approximation to $(Y,H+h)$,
even at low energies, near the projections of these extra components to $X$.
}

When $X,Y$ are each equidimensional, one can measure ``how extra'' a component
of $C_X Y$ is. Since the projection $C_X Y \onto X$ is surjective,
each component of $X$ is the image of some component of $C_X Y$,
but not every component of $C_X Y$ (which, like $Y$, is also
equidimensional) projects onto a component of $X$; it may project to
something lower-dimensional. So to each component of $C_X Y$ we can
associate the codimension inside $X$ of the projection of the component.

In the case at hand, 
$Y$ is the nilpotent orbit and $X$ is the orbital scheme.
We conjectured in \cite[Theorem 10]{KZJ} (in slightly different language)
that in the right coordinates,
\begin{equation*}
 C_{D \cap R_N} D = 
\{M \in \MMN : M^2 = 0\};
\end{equation*}
we were able to prove the $\subseteq$ inclusion,
and that these two schemes agree in top dimension. 
(The scheme $E$ is smaller than these, because in its definition
we impose that $M$ has diagonal entries {\em equal\/} 
to zero, rather than just squaring to zero.
Of course this makes no difference on the level of varieties.)

Here $D\cap R_N$ has many fewer components than $C_{D\cap R_N} D$
(they correspond to noncrossing chord diagrams, rather than all chord diagrams).
In Theorem \ref{thm:PsiJM} we show that the codimension inside $X$
of the projection of a component is exactly the number of crossings
in the corresponding chord diagram.

\subsection{Integrability and the \texorpdfstring{$q$}{q}KZ equation}\label{ssec:qKZ}
In \cite{dGN}, a certain Markov process on the set of (crossing) link patterns
was considered. The motivation was that the Markov matrix is actually
a {\em quantum integrable\/} transfer matrix related to the Brauer algebra
\cite{Brauer} with parameter $\beta$, 
at the special value $\beta=1$ of the parameter (for an explanation 
of the appearance of the Brauer algebra 
in solving the Yang--Baxter equation, see \cite{Jimbo-review},
as well as sect.~\ref{sec:brauer}).
A remarkable conjecture of \cite{dGN} (now a theorem) is that certain
components of the equilibrium distribution eigenvector of this Markov
process can be identified, after dividing them by the smallest
component, with degrees of certain algebraic varieties.

The model was further studied in \cite{DFZJ06}, where several important
properties were shown. First, the model can naturally be made
inhomogeneous, and the introduction of the inhomogeneities $z_i$ make
these equilibrium probabilities be polynomials (again, up to normalization)
in the variables $z_i$. On the geometric side, this correponds to
generalizing degrees to {\em multidegrees}, that is, enlarging the torus
action, as will be explained in the next section.
Secondly, the main method used in \cite{DFZJ06}, inherited from \cite{DFZJ05},
is to write certain ``exchange relations'': these express the effect
of interchange of variables $z_i$, $z_{i+1}$ as a linear operator
acting on the equilibrium distribution vector. The exchange relations
appear in multiple contexts in the study of quantum integrable models,
but in particular, supplemented with an appropriate cyclicity property,
they are related to the so-called quantum Knizhnik--Zamolodchikov
($q$KZ) equation \cite{Smi,FR}.

More progress was made in \cite{KZJ}, where {\em all\/} the components of
the Markov process eigenvector were given geometric meaning: these
are the (multi)degrees of the Brauer loop varieties.
The central role of the exchange relation, already pointed out in
\cite{DFZJ06}, is developed further in \cite{KZJ}. The present work
will complete (in \S\ref{sec:qkz}) the general program outlined
in these two prior papers: the exchange relation will be entirely
explained geometrically as the translation at the level of equivariant
cohomology of certain elementary geometric operations on the
irreducible components of the Brauer loop scheme.

Even though there exists a solution of the Yang--Baxter equation
for arbitrary values of the parameter $\beta$ of the Brauer algebra,
one cannot define a corresponding Markov process, as was the case at $\beta=1$
(the integrable transfer matrix does not possess the Markov property,
and therefore one does not expect the ground state to be simple).
One can however introduce a $q$KZ equation and try to look for certain
polynomial solutions which would generalize the equilibrium vector of
the Markov process at $\beta=1$.  This provides a much more natural
framework to study the Brauer loop scheme, and is what is considered
in the present work. The parameter $\beta$ can be thought as yet
another enlargement of the torus action (the additional circle action
was in fact mentioned in the last section of \cite{KZJ}).

\subsection{Torus actions and multidegrees}\label{ssec:multidegrees}
\newcommand\Sym{{\rm Sym\,}}
Let $T \iso (\complexes^\times)^k$ be a (complex) torus, and consider the
pairs $(X\subseteq W)$ of linear $T$-representations $W$ 
containing $T$-invariant
closed subschemes $X\subseteq W$. To each such pair $X \subseteq W$ we will 
assign a polynomial $\mdeg_W X \in \Sym T^* \iso \integers[z_1,\ldots,z_k]$
called the {\dfn multidegree} of $X$. (Here $\Sym T^*$ denotes the 
symmetric algebra on the lattice $T^*$ of characters of $T$.)
Our reference for multidegrees is \cite{MS}.

This assignment can be computed using the following properties 
(as in \cite{Jo97}):
\begin{enumerate}
\item[1.] If $X=W=\{0\}$, then $\mdeg_W X = 1$.
\item[2.] If the scheme $X$ has top-dimensional components $X_i$, 
  where $m_i>0$ denotes the multiplicity of $X_i$ in $X$, 
  then $\mdeg_W X = \sum_i m_i\ \mdeg_W X_i$. This lets one reduce from
  the case of schemes to the case of varieties (reduced irreducible schemes).
\item[3.] Assume $X$ is a variety, 
  and $H$ is a $T$-invariant hyperplane in $W$.
  \begin{enumerate}
  \item If $X\not\subset H$, then $\mdeg_W X = \mdeg_H (X\cap H)$.
  \item If $X\subset H$, then 
    $ \mdeg_W X = (\mdeg_H X) \cdot ( \text{the weight of $T$ on $W/H$}). $
  \item 
  Combining (a) and (b), we have
  $\mdeg_W (X\cap H) = (\mdeg_W X)\ (\mdeg_W H)$ when $X\not\subset H$. 
We ask
  \footnote{%
    This actually follows from the other axioms, since they imply
    that $\mdeg_W X$ can be computed from the multigraded Hilbert
    series of $X$, and it is easy to relate the Hilbert series of
    $X$ and $X\cap H$.}
  this to hold even when $H \subseteq W$ is just a $T$-invariant 
  hyper{\em surface}.
  \end{enumerate}
\end{enumerate}
One can readily see from these properties that $\mdeg_W X$ is 
homogeneous of degree $\codim_W X$, and is a positive sum of products
of the weights of $T$ on $W$. We explore this further in Lemma 
\ref{lem:mdegineqs}.

The varieties for which we will need the multidegrees are 
the orbital varieties and the Brauer loop varieties.
In both cases, $W$ will be a space of zero-diagonal matrices, 
and the $(N+2)$-dimensional torus $T$ that acts on it will have three parts: 
two dimensions by scaling certain halves of the matrices,
and the other $N$ by conjugating by 
the invertible diagonal matrices in $R_{\integers \bmod N}$. 
We will denote by $\{\A, \B, z_1,\ldots, z_N\}$
the obvious basis of the weight lattice $T^*$.

For orbital varieties, the vector space $W$ will be $R_N^{\Delta=0}$, 
the space of strictly upper triangular $N\times N$ matrices,
where the $\Delta=0$
indicates the subspace of matrices with zero diagonal.
The first circle in $T$ acts by global rescaling, and the second
acts trivially, so the $T$-weights are $\A + z_i - z_j$, $i<j \in 1,\ldots,N$. 
For each noncrossing chord diagram $\pi$, let
\begin{equation*}
 J_\pi := \mdeg_{\lien} \calO_\pi, 
\end{equation*}
which is called the {\dfn (extended) Joseph polynomial} \cite{Jo}
of the orbital variety.
(Indeed, Joseph invented multidegrees for exactly this application.)
The ``extended'' refers to the fact that Joseph did not consider the
scaling action; his polynomials correspond to the specialization $\A=0$.

For Brauer loop varieties, the vector space $W$ will be
$\MMN^{\Delta=0}$, with a similar notation.
Separate the matrix entries into the ``$U$'' group, being
those matrix entries $M_{ij}$ with $\lfloor i/N\rfloor = \lfloor j/N\rfloor$,
and the ``$L$'' group, which are the rest.
(This matches the picture in \S\ref{ssec:brbr}.)
The first circle acts by scaling $U$, and trivially on $L$,
the second circle acts trivially on $U$, and by scaling on $L$.
The weights on $\MMN$ are
$\A + z_i - z_j, \B + z_j - z_i$, $i<j=1,\ldots,N$.
For each chord diagram $\pi$, let
\begin{equation*}
\Psi_\pi := \mdeg_{\MMN^{\Delta=0}} E_\pi 
\end{equation*}
be the {\dfn Brauer loop polynomial}. 
{(In \cite{KZJ}, we did not consider until \S 8 the separate
action on the $U$ and $L$ parts, and before that recovered only the
specialization $\A=\B$ of the polynomials presented here.)

An interesting, explicitly cyclic-invariant, 
reformulation of the weights is obtained by introducing a redundant set of variables
$z_i$, $i\in\integers$, with the relations $z_{i+N}=z_i+\A-\B$. Then the weight of $M_{ij}$ is simply
$\A+z_i-z_j$ for any $i$, $j$. This notation will be used in what follows.

Note that the maps in \S\ref{ssec:orbvars} are $T$-equivariant,
which in Theorem \ref{thm:PsiJM} will help us relate Joseph and Brauer loop
polynomials.


\subsection{Summary of main results}\label{ssec:summary}
Most of our results are extensions of those in \cite{KZJ}, on both
geometric and integrable sides. For the sake
of simplicity, and except when stated otherwise, 
we assume in this paper that $N$ is an even integer. (The extension
to $N$ odd, as in \cite{KZJ}, is straightforward 
but would complicate the geometric
constructions, potentially obscuring the logic of the paper.)

On the geometric side, we reinterpret one of our descriptions from
\cite{KZJ} of the Brauer loop scheme now as being the normal cone
$C_{\{M \in R_N \,:\, M^2=0\}} \{M : M^2=0\}$ 
inside the nilpotent orbit closure $\{M : M^2=0\}$ to the orbital scheme. 
In particular this motivates study of the $B_N$-equivariant projection 
$C_{\{M \in R_N \,:\, M^2=0\}} \{M : M^2=0\} \onto \{M \in R_N \,:\, M^2=0\}$
onto the orbital scheme. In \S \ref{sec:orbvars} we give a
diagrammatic description of the poset of $B_N$-orbits on the orbital scheme
(see in particular Proposition \ref{prop:moves}),
and a formula for their multidegrees, which we call {\em Joseph--Melnikov 
polynomials}. In \S \ref{sec:ncone} we study the multidegrees of components 
of normal cones in general, and use it to show that the Joseph--Melnikov 
polynomials are the leading forms of the Brauer loop polynomials.

In \S \ref{sec:qkz} we consider a set of equations called
quantum Knizhnik--Zamolodchikov equation for the Brauer algebra.
They are a set of compatible difference equations 
which generalize the ones satisfied by the steady state of
the Brauer markov process of \cite{DFZJ06}; we also extend their 
polynomial solution $\Psi$ with one new 
polynomial parameter (Theorem \ref{thm:qkzsol}). This system contains 
``$f$'' and ``$e$'' equations \eqref{eqn:qkzd} and \eqref{eqn:qkze}
corresponding to the two types of generators of the Brauer algebra.
The main goal of the rest of the paper is to identify
this solution of the $q$KZ equation as the vector of multidegrees of
the Brauer loop scheme (Theorem \ref{thm:main}), 
and to interpret geometrically 
equations \eqref{eqn:qkzd} and \eqref{eqn:qkze}.

In \S \ref{sec:geombrauer} we begin by recalling Hotta's construction
of Springer representations via orbital schemes. 
The traditional way to do
this, as via the convolution construction on the Steinberg scheme,
involves slicing an orbital variety by a hyperplane then sweeping the
result out with an $SL_2$ (one for each simple root, $i=1,\ldots,N-1$), 
which we call ``cut then sweep''. 
Hotta's picture less traditionally allows for a ``sweep then cut'' construction,
but this gives nothing new, as the corresponding operators on 
multidegrees just differ by $2$. In either construction, a key
realization is that the new scheme constructed is a schemy union
of orbital varieties.

Moving to the Brauer loop scheme,
it becomes possible to extend to the affine root system, so $i=1,\ldots,N$.
When cutting and sweeping the Brauer loop variety $E_\pi$, 
we have to treat the case $\pi(i)=i+1$ separately from the case that the
chords from $i,i+1$ are different. (It also becomes important to break
the latter case further, according to whether the two chords from $i,i+1$
cross one another.)

In the case $\pi(i)\neq i+1$, we can sweep then cut, and derive
Equation (\ref{eqn:qkzd}), as we did in \cite{KZJ}.
Already this is different from the Hotta situation, in that we must
cut with a quadratic hypersurface, not a hyperplane, to get a
schemy union of Brauer loop varieties.
The case $\pi(i) = i+1$ is harder; we cut $E_\pi$ with a hyperplane,
and identify some of its components as intersections of other
components $E_\rho$ with a quadratic hypersurface. 
This gives an inequality on cycles,
and we develop a corresponding theory of inequalities on multidegrees.
The multidegree inequality we derive is equivalent to
(\ref{eqn:qkze}) (via Equation (\ref{eqn:qkzev3})). We conclude our
geometric argument by invoking the algebra of \S \ref{sec:qkz},
which shows that this upper bound becomes tight after sweeping.
We calculate an example after Corollary \ref{cor:poserrorterm}
to show they are not equal before sweeping. 

\begin{table}[ht]
\caption{Summary of main notations}
\begin{tabular}{|c|c|c|}
\hline \vbox to 4mm{}
Embedding space & $R_N$ & $\MMN$ \\
\hline \vbox to 4mm{}
Scheme & orbital scheme $\calO$ & Brauer loop scheme $E$ \\
\hline \vbox to 4mm{}
Components & $\calO_\pi$ & $E_\pi$ \\
 & ($\pi$ noncrossing link pattern) & ($\pi$ link pattern)\\
\hline \vbox to 4mm{}
Multidegrees & $J_\pi$ & $\Psi_\pi$ \\
\hline \vbox to 5mm{}
Algebra acting & degenerate Brauer $\bar\Br_N(2)$ & affine Brauer $\hat\Br_N(\beta)$\\
(subalgebras) & (Temperley--Lieb, nil-Hecke) & (Temperley--Lieb, affine symmetric group)\\
\hline
\end{tabular}
\end{table}

{\em Remark}: for more on the Temperley--Lieb algebra in a related context, and in particular for a discussion of the $q$KZ equation associated to the Temperley--Lieb algebra, see \cite{DFZJ05b,RTVZJ}.

\section{The orbital varieties of \texorpdfstring{$\{M^2=0\}$}{\{M2=0\}}}\label{sec:orbvars}
\subsection{The poset of \texorpdfstring{$B_N$}{B}-orbits of \texorpdfstring{$\{M^2=0\}$}{\{M2=0\}}}
\label{ssec:posetBorbits}
Recall that $D$ denotes the nilpotent orbit closure $D := \{M : M^2 = 0\}$ 
inside $\text{Mat}_N=\glN$ (in this section the parity condition on $N$ is relaxed
unless stated otherwise).
This nilpotent orbit is much easier to study than a general one,
in that it is {\dfn spherical}:
the Borel subgroup $B_N$ acts on $D$ with finitely many orbits.
(Moreover, the other nilpotent orbits of $\glN$ with this property, 
such as $\{\vec 0\}$, are all contained in $D$.)
This finite set of orbits naturally forms a ranked poset, 
with the rank given by the 
dimension of the orbit, and the partial order by inclusion of orbit closures.
A full description of this ranked poset appears in \cite{R}.

We will not study all the $B_N$-orbits on $D$,
but focus on the $B_N$-invariant orbital scheme $D \cap R_N$,
whose corresponding subposet was determined in \cite{M1,M2}.
Each component of $D\cap R_N$ (i.e., each orbital variety) itself has
finitely many $B_N$-orbits, and in particular is the closure of a
$B_N$-orbit, corresponding to a maximal element of the poset.

\begin{Theorem}\cite{M1}\label{thm:involutions}
  For each orbit $\calO$ of $B_N$ on $D\cap R_N$, there exists a unique
  involution $\pi\in \SN$ such that $\calO = B_N\cdot \pi_<$.
  (Here $\pi_<$ denotes the permutation matrix $\pi$ with its diagonal
  and lower triangle zeroed out, i.e. $(\pi_<)_{ij} = 0$ unless $i<j$.)
\end{Theorem}

We will draw these involutions as chord diagrams on the interval, 
with an arch connecting $i \leftrightarrow \pi(i)$ for $i\neq \pi(i)$,
and a vertical half-line from $i$ for each fixed point $i = \pi(i)$,
drawn so that any two {\dfn curves} (meaning, arch or half-line)
cross transversely and at most once. 
This encoding will make it easy to describe the dimension of the orbit
and the covering relations in the poset. While these were already
computed in \cite{M1,M2},
our description is sufficiently different
that we find it simpler to give independent proofs.

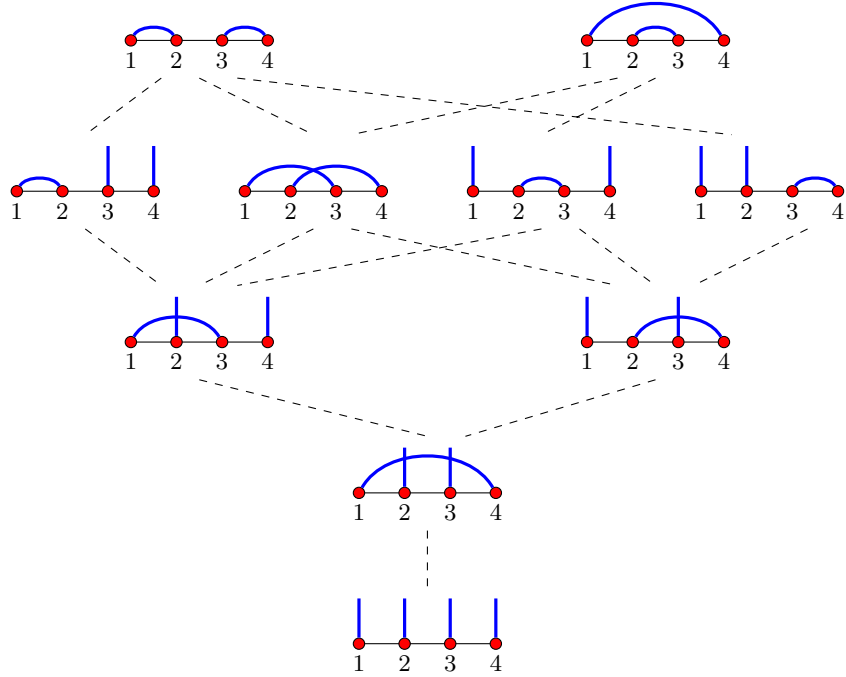
\begin{figure}[htbp]
  \centering
\begin{tikzpicture}
\linkpatternnumberedtrue
\linkpattern[tikzstarted,pos={(-2.5cm,4cm)}]{1/2,3/4}
\linkpattern[tikzstarted,pos={(3.5cm,4cm)}]{1/4,2/3}
\linkpattern[tikzstarted,pos={(-4cm,2cm)},height=1]{1/2,3/3,4/4}
\linkpattern[tikzstarted,pos={(-1cm,2cm)},height=1]{1/3,2/4}
\linkpattern[tikzstarted,pos={(2cm,2cm)},height=1]{1/1,2/3,4/4}
\linkpattern[tikzstarted,pos={(5cm,2cm)},height=1]{1/1,2/2,3/4}
\linkpattern[tikzstarted,pos={(-2.5cm,0cm)},height=1]{1/3,2/2,4/4}
\linkpattern[tikzstarted,pos={(3.5cm,0cm)},height=1]{1/1,2/4,3/3}
\linkpattern[tikzstarted,pos={(0.5cm,-2cm)},height=1]{1/4,2/2,3/3}
\linkpattern[tikzstarted,pos={(0.5cm,-4cm)},height=1]{1/1,2/2,3/3,4/4}
\draw[dashed] (-1.5,3.5) -- (-2.5,2.75) (-1,3.5) -- (0.5,2.75) (-0.5,3.5) -- (6,2.75) (4.5,3.5) -- (1,2.75) (5,3.5) -- (3.5,2.75) (-2.5,1.5) -- (-1.5,0.75) (0.5,1.5) -- (-1,0.75) (1,1.5) -- (4.5,0.75) (3.5,1.5) -- (-0.5,0.75) (4,1.5) -- (5,0.75) (7,1.5) -- (5.5,0.75) (-1,-0.5) -- (2,-1.25) (5,-0.5) -- (2.5,-1.25) (2,-2.5) -- (2,-3.25);

\end{tikzpicture}

\caption{The poset of $B_N$-orbits for $N=4$. The row gives the
    dimension, computable from Theorem \ref{thm:orbdim},
    from $4$ at the top down to $0$ for the orbit $\{0\}$.}
  \label{fig:poset4}
\end{figure}

\begin{Theorem}\label{thm:orbdim}
  Let $\pi\in \SN$ be an involution, drawn as a chord diagram on
  the interval. Then the corresponding $B_N$-orbit has dimension
  \begin{align*}
    \dim B_N\cdot \pi_< 
    &=\#\text{arches}+\#\text{noncrossing pairs $\{$arch, arch$\}$ or 
      $\{$arch,half-line$\}$)}\\
    &=\text{$\#$arches $\cdot$ $(\#$arches $+$ $\#$half-lines$)-\#$crossings}.
  \end{align*}

  The maximum dimension $\lfloor N^2/4\rfloor$ is achieved iff
  $\pi$ has no crossings and at most one half-line, iff 
  $\overline{B_N\cdot \pi_<}$ is an orbital variety of $D$.
\end{Theorem}

More generally, if we require that $\pi$ has at most $k$ $2$-cycles,
then the maximum dimension $k(N-k)$ is achieved iff there are indeed
$k$ $2$-cycles and no crossings, iff $\overline{B_N\cdot \pi_<}$ is an
orbital variety for the nilpotent orbit closure
$\mathcal O = \{M : M^2 = 0, \rank M \leq k\}$.
This latter statement -- that the components of $\mathcal O\cap R_N$ are all
the same dimension -- holds for any nilpotent orbit $\mathcal O$ \cite{Sp};
in fact the dimension is $\onehalf \dim \mathcal O$.

For an involution $\pi$, write
\begin{equation*}
 \rank \pi_{[ij]} := \#\{i\leq a<b\leq j : \pi(a) = b \}, \end{equation*}
i.e., the number of complete arches sitting between positions $i$ and $j$ 
of the chord diagram.

\begin{Theorem}\cite{M2,R}\label{thm:poset}
  For two involutions $\pi,\rho\in \mathcal{S}_n$, 
  \begin{equation*} 
    \overline{B_N\cdot \pi_<} \supseteq B_N\cdot \rho_<
    \quad\Longleftrightarrow\quad
    \rank \pi_{[ij]} \geq \rank \rho_{[ij]} \quad \forall i<j. \end{equation*}
\end{Theorem}

\begin{proof}[Proof of $\Longrightarrow$]
  Let $M_{\geq i,\leq j}$ denote the part of $M$ southwest of $(i,j)$, i.e.,
  \begin{equation*} (M_{\geq i,\leq j})_{kl} = 
  \begin{cases}
    M_{kl} & \text{if $k\geq i$ and $l\leq j$} \\
    0 & \text{otherwise.}
  \end{cases}
  \end{equation*}
  In the case $M = \pi_<$, we have
  $ \rank (\pi_<)_{\geq i,\leq j} = \rank \pi_{[ij]}$.

  The conjugation action of $B_N$ on $M$ restricts, in the following sense, 
  to an action on each southwest part:
  \begin{equation*} (b\cdot M)_{\geq i,\leq j} 
  = (b\cdot M_{\geq i,\leq j} )_{\geq i,\leq j}. \end{equation*}
  This implies that $\rank M_{\geq i,\leq j}$ is invariant under 
  $B_N$-conjugation.

  Consequently, if $M' \in B_N\cdot M$, then
  $\rank M'_{\geq i,\leq j} = \rank M_{\geq i,\leq j}$ 
  for all $i,j$. 
  The semicontinuity of $\rank$ gives us an inequality on the closure:
  \begin{equation*} M' \in \overline{B_N\cdot M}  \quad\Longrightarrow\quad
  \rank M_{\geq i,\leq j} \geq  \rank M'_{\geq i,\leq j} \quad
  \forall i,j. \end{equation*}
  Now apply this to $M' = \rho$, $M = \pi$ to obtain the desired statement.
\end{proof}

We will prove Theorem \ref{thm:orbdim} and the other half of Theorem
\ref{thm:poset} by an analysis of the covering relations in the poset
of orbit closures (the ``moves'' from \cite{R}).
We encourage the reader to reconstruct the poset in Figure
\ref{fig:poset4} from the top down using the following Proposition.

\begin{Proposition}\label{prop:moves}
  Let $\pi\in \SN$ be an involution, with an associated chord diagram
  also called $\pi$.
  Construct a new chord diagram $\rho$ in one of three ways: 
  \begin{enumerate}
  \item If two arches in $\pi$ border a common region, but do not cross,
    make them touch and turn that into a new crossing;
  \item if an arch and a half-line in $\pi$ border a common region,
    but do not cross, make them touch and turn that into a new crossing;
  \item if an arch crosses all the half-lines, 
    and borders the unbounded region, break it into two half-lines.
  \end{enumerate}
  Then $\pi > \rho$ in the poset of orbit closures, i.e.,
  $\overline{B_N\cdot \pi_<} \supset B_N\cdot \rho_<$. (We will later
  prove these to be covering relations, and all of them.)
\end{Proposition}

\begin{proof}
  In each of these cases, we will construct a one-parameter family of group 
  elements $b(t) \in B_N$ such that $\lim_{t\to0} b(t)\cdot \pi = \rho$,
  where $\cdot$ denotes the conjugation action.
  These $\pi$ and $\rho$ only differ in a few columns and rows, and we will
  be able to take $b(t)$ the identity outside those, making it possible
  to write down $b(t)$ in a small space. 

  1. If one arch contains the other, the relevant submatrices are
  \begin{equation*} 
  \begin{pmatrix}
    1&1& &  \\
     &t& &  \\
     & &1& 1 \\
     & & &-t \\
  \end{pmatrix}
  \cdot
  \begin{pmatrix}
    0&0&0&1 \\
     &0&1&0 \\
     & &0&0 \\
     & & &0
  \end{pmatrix}
  = 
  \begin{pmatrix}
    0&0&1&0 \\
     &0&t&1 \\
     & &0&0 \\
     & & &0
  \end{pmatrix}
  \dashrightarrow
  \begin{pmatrix}
    0&0&1&0 \\
     &0&0&1 \\
     & &0&0 \\
     & & &0
  \end{pmatrix}
  \text{as $t\to 0$.} 
  \end{equation*}
  If instead the arches are side by side:
  \begin{equation*} 
  \begin{pmatrix}
    t& & & \\
     &1&1& \\
     & &-t& \\
     & & &1    
  \end{pmatrix}
  \cdot
  \begin{pmatrix}
    0&1&0&0 \\
     &0&0&0 \\
     & &0&1 \\
     & & &0
  \end{pmatrix}
  = 
  \begin{pmatrix}
    0&t&1&0 \\
     &0&0&1 \\
     & &0&t \\
     & & &0
  \end{pmatrix}
  \dashrightarrow
  \begin{pmatrix}
    0&0&1&0 \\
     &0&0&1 \\
     & &0&0 \\
     & & &0
  \end{pmatrix}
  \text{as $t\to 0$.} 
  \end{equation*}

  2. Assume (by symmetry) that the half-line is left of the arch:
  \begin{equation*}
  \begin{pmatrix}
    1&1& \\
     &t& \\
     & &1\\
  \end{pmatrix}
  \cdot
  \begin{pmatrix}
     0&0&0 \\
      &0&1 \\
      & &0
  \end{pmatrix}
  = 
  \begin{pmatrix}
    0&0&1 \\
     &0&t \\
     & &0
  \end{pmatrix}
  \dashrightarrow
  \begin{pmatrix}
    0&0&1 \\
     &0&0 \\
     & &0
  \end{pmatrix}
  \text{as $t\to 0$.} 
  \end{equation*}
  
  3. To prove $\pi > \rho$, 
  we shall not here need to use the condition that the arch crosses all
  half-lines -- it is only included to later ensure that this is a
  covering relation.
  \begin{equation*}
  \begin{pmatrix}
    t&  \\
     &1
  \end{pmatrix}
  \cdot
  \begin{pmatrix}
    0&1 \\
     &0
  \end{pmatrix}
  = 
  \begin{pmatrix}
    0&t \\
     &0
  \end{pmatrix}
  \dashrightarrow
  \begin{pmatrix}
    0&0 \\
     &0
  \end{pmatrix}
  \text{as $t\to 0$.} 
  \end{equation*}
\end{proof}

When looking for the $B_N$-orbits covered in this poset by a given
$B_N$-orbit $B_N\cdot \pi_<$, one must be careful to consider {\em all\/}
the ways to draw the chord diagram of $\pi$, as different drawings may
make different pairs of chords adjacent. 
An example is in Figure \ref{fig:multdraw}.

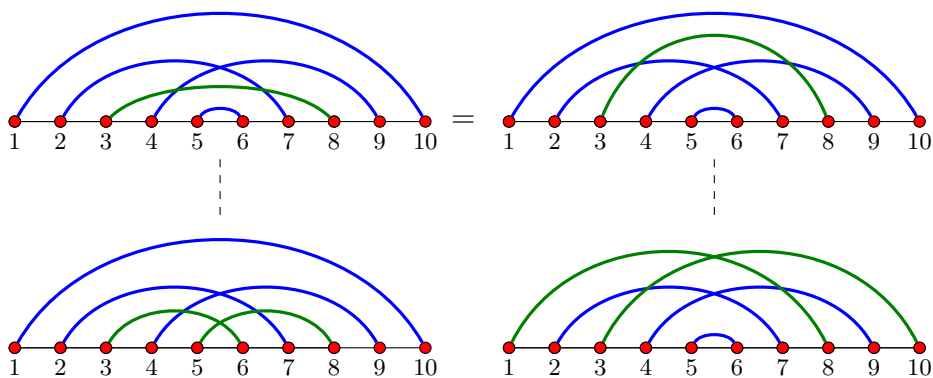
\begin{figure}[htbp]
  \centering
\begin{tikzpicture}
\linkpattern[tikzstarted,numbered,pos={(-3.5cm,0cm)}]{1/10,2/7,4/9,5/6}\linkpattern[tikzstarted,edgecolor=green!50!black,squareness=0.2,pos={(-3.5cm,0cm)}]{3/8}
\node at (3,0) {$=$};
\linkpattern[tikzstarted,numbered,pos={(3cm,0cm)}]{2/7,4/9,5/6,1/10}\linkpattern[tikzstarted,edgecolor=green!50!black,squareness=0.5,pos={(3cm,0cm)}]{3/8}
\draw[dashed] (-0.2,-0.5) -- (-0.2,-1.25) (6.3,-0.5) -- (6.3,-1.25);
\linkpattern[tikzstarted,numbered,pos={(-3.5cm,-3cm)}]{1/10,2/7,4/9}\linkpattern[tikzstarted,edgecolor=green!50!black,pos={(-3.5cm,-3cm)}]{3/6,5/8}
\linkpattern[tikzstarted,numbered,pos={(3cm,-3cm)},size=10]{2/7,4/9,5/6}\linkpattern[tikzstarted,edgecolor=green!50!black,pos={(3cm,-3cm)},squareness=0.4]{1/8,3/10}
\end{tikzpicture}
  \caption{The top two chord diagrams are of the same involution $\pi$,
    with the $3\leftrightarrow 8$ chord placed differently. In either diagram,
    the $3\leftrightarrow 8$ chord is brought into contact with 
    another chord, leading to the lower two diagrams. In those two
    we emphasize the newly crossing chords.
}
  \label{fig:multdraw}
\end{figure}

\begin{proof}[Proof of Theorem~\ref{thm:orbdim}]
  Let $d_\pi$ denote the statistic in Theorem \ref{thm:orbdim}; we wish
  to prove $d_\pi = \dim B_N\cdot \pi_<$. One can do this (as in \cite{M1})
  by computing the commutant in $B_N$ of $\pi_<$, but we find it
  instructive to use Proposition~\ref{prop:moves}.

  First we prove that for each of the moves in Proposition~\ref{prop:moves},
  $d_\pi = d_\rho + 1$. For the first two constructions it is essentially 
  obvious -- the number of arches and curves do not change, and exactly
  one crossing is created.

  For the third move, let $a$ be the number of arches, $N-2a$ the number
  of half-lines, and $c$ the number of crossings, so $d_\pi = a(N-a) - c$.
  When we break the arch, we lose $N-2a$ crossings of it by half-lines,
  so $d_\rho = (a-1)(N-(a-1)) - (c-(N-2a)) = d_\pi - 1$ as desired.
  
  We now embed $\pi$ in a maximal chain in the poset, first building
  upward (using the moves in reverse).
  Alternate between two strategies: replace every crossing with $)($,
  then when there are no crossings, join the leftmost half-line to the
  right-most, making an arch that crosses the remaining half-lines
  and does not cross any other arch twice (or indeed at all).
  This process stops at a noncrossing chord diagram $\pi^+ \geq \pi$
  with at most one half-line.

  To go downward in the poset, look first for a half-line at position $i$
  and an arch $j \leftrightarrow k$ not crossing it. By reflecting, we can 
  assume for discussion that $i<j<k$. Find such an arch with smallest $j$. 
  If there is a half-line at $j-1$ (e.g. if $i=j-1$), then we can apply 
  move \#2 to that half-line and the $j\leftrightarrow k$ arch and thus
  go downward in the poset. If there is an arch at $j-1$ and $\pi(j-1)<j-1$,
  then we can apply move \#1 to make the arches $\pi(j-1)\leftrightarrow j-1$
  and $j\leftrightarrow k$ cross. Otherwise there is an arch at $j-1$
  with $\pi(j-1)>j-1$, hence not crossing the half-line at $i$, 
  contradicting our choice of smallest $j$.

  Once that is done, every half-line crosses every arch. Now if there are
  any arches left, we can apply move \#3 to any arch touching the
  unbounded region. (Then return to the paragraph above.) This process
  stops only when there are no arches, i.e., at the identity
  permutation $\Id$.

  We have thus constructed a sequence $(\pi^+,\ldots,\pi,\ldots,\Id)$,
  each related by a move from Proposition~\ref{prop:moves},
  with $(d_{\pi^+} = \lfloor N^2/4\rfloor, \ldots, d_\pi, \ldots, d_{\Id}=0)$
  therefore decreasing by $1$ at each step. Correspondingly, we can compute 
  the sequence $(\dim B_N\cdot \pi^+_<,\ldots,\dim B_N\cdot \pi_<,\ldots,0)$ 
  of dimensions of the orbits, which must strictly decrease at each step
  (since for any variety, $\dim (\overline X \setminus X) < \dim X$).
  All that remains is to know that 
  $\dim (D\cap R_n) = \onehalf \dim D$ \cite{Sp}, straightforwardly
  computed to be $\lfloor N^2/4\rfloor$, and this is an upper bound on
  $\dim B_N\cdot \pi^+_<$.
\end{proof}

In particular, this computation of the rank function proves that the
moves in Proposition~\ref{prop:moves} are indeed covering relations.

\begin{proof}[Proof of Theorem \ref{thm:poset}, $\Longleftarrow$]
  We first analyze the effect of the moves $\pi\mapsto \rho$ 
  from Proposition~\ref{prop:moves} on the rank function $\rank \pi_{[ij]}$,
  leaving the details of the verification to the reader.
  
  For moves $1$ and $2$, $\rank \rho_{[ij]} = \rank \pi_{[ij]}$ unless
  the interval
  $[ij]$ fully contains one arch and does not contain any end of the
  other curve, in which case $\rank \rho_{[ij]} = \rank \pi_{[ij]}-1$.
  For move $3$, $\rank \rho_{[ij]} = \rank \pi_{[ij]}$ unless
  the interval $[ij]$ fully contains the arch, 
  in which case $\rank \rho_{[ij]} = \rank \pi_{[ij]}-1$.

  In particular, $\rank \rho_{[ij]} \leq \rank \pi_{[ij]}$ for all $i<j$,
  and for some $[ij]$ the inequality is strict, making inductive arguments
  possible.

  Now let $\pi \neq \rho$ be chord diagrams on the interval (possibly with
  half-lines) such that 
  $  \rank \pi_{[ij]} \geq \rank \rho_{[ij]} \quad \forall i<j. $
  We wish to show that there exists a move $\pi\mapsto \pi'$ such that
  we again have $\rank \pi'_{[ij]} \geq \rank \rho_{[ij]} \quad \forall i<j $.
  Then by Proposition \ref{prop:moves} and induction, 
  we can infer that
  $\overline{B_N\cdot \pi_<} \supset B_N\cdot \rho_<$. 

  Here it is useful to distinguish move (1a) where the
  two arches to be crossed are next to each other and move (1b) where the
  two arches are nested one inside the other.

  Let $b<c$ be such that $\rank \pi_{[bc]} > \rank \rho_{[bc]}$ with
  $c-b$ minimized. Then $\pi$ necessarily has an arch connecting $b$ and $c$.
  Next consider any pair $a\ge 0$, $d\le N+1$ such that 
  $\rank \pi_{[ij]} > \rank \rho_{[ij]}$ 
  for all $a<i\le b$ and $c\le j<d$, and such that $a$ is minimum,
  $d$ is maximum, for that condition. There are four cases,
  in each of which we use a covering
  relation from Proposition \ref{prop:moves} to construct a $\pi' \lessdot \pi$
  that is still $\geq \rho$:
\begin{itemize}
\item $\pi$ possesses an arch inside the square $[a,b[\ \times\ ]c,d]$,
i.e., $c<\pi(i)\le d$ for some $a\le i<b$. Pick $i$ 
such that $\pi(i)-i$ minimal, and
apply covering relation (1b) to
arches $b,c$ and $i,\pi(i)$.

In all other cases we assume that $\pi$ has no arch inside 
$[a,b[\ \times\ ]c,d]$ (i.e., no arch ``covers'' $(b,c)$). 
Note that this implies $\rank \pi_{[ij]}=\rank\rho_{[ij]}+1$
for $a< i\le b$, $c\le j<d$, and furthermore that one cannot
have simultaneously $a>0$ and $d<N+1$.

\item $a=0$ and $d=N+1$. Then arch $b,c$ is connected to the outside.
If there exists a half-line outside arch $b,c$, 
then apply covering relation (2) to $b,c$ and the closest half-line. 
If not, apply covering relation (3) to $b,c$.

\item $a=0$ and $d<N+1$. If there exists a half-line between $c$ and
$d$ (including $d$), apply covering relation (2) to the arch $(b,c)$ and
the closest half-line. If not apply covering relation (1a) to $(b,c)$
and $(d,\pi(d))$.

\item $a>0$ and $d=N+1$. If there exists a half-line between $b$ and
$a$ (including $a$), apply covering relation (2) to the arch $(b,c)$ and
the closest half-line. If not apply covering relation (1a) to $(b,c)$
and $(\pi(a),a)$.
\end{itemize}
\end{proof}

Since the poset is finite, and given any pair $\pi > \rho$ we found
a Proposition \ref{prop:moves} covering relation $\pi \gtrdot \pi'$
with $\pi' \geq \rho$, by induction we have the

\begin{Corollary*}[of proof]
  The covering relations in Proposition \ref{prop:moves} are all
  the covering relations.  
\end{Corollary*}

\subsection{Joseph--Melnikov polynomials}

Let $\pi \in \SN$ be an involution, and $\overline{B_N\cdot \pi_<}$ 
the corresponding orbit closure in the orbital scheme of $\{M^2=0\}$.
Then we define the {\dfn Joseph--Melnikov polynomial} $J_\pi$ to be the
multidegree of $\overline{B_N\cdot \pi_<}$ exactly as we did the
(extended) Joseph polynomials, i.e., inside the strictly upper
triangular matrices, with respect to the action of the scaling circle and 
the diagonal matrices. In particular, if $\pi$ has no crossings,
then the orbit closure is an orbital variety, and 
the Joseph--Melnikov polynomial is its Joseph polynomial.

Though we shall not make direct use of it, we point out one case 
of particular interest.
Let $\rho \in \mathcal{S}_n$ be an arbitrary permutation, and associate an 
involution $\pi$ of $1,\ldots,N=2n$ as follows:
\begin{equation*} \pi(i) = 
\begin{cases}
  n + \rho(n+1-i) & i=1,\ldots,n \\
  n+1-\rho^{-1}(i-n)   & i=n+1,\ldots,N
\end{cases}
\qquad \pi =
\begin{pmatrix}
  0 & w_0 \rho \\ 
  \rho^T w_0 & 0 
\end{pmatrix}
\end{equation*}
Such involutions were already considered in section~5 of \cite{KZJ}, forming
the so-called {\em permutation sector}.
Then $\pi_<$ is not only upper triangular, 
but supported in the upper right quarter of this $2n\times 2n$ matrix, 
where it matches the permutation matrix of $\rho$ (except for
being upside down).
The action of $B_N \subseteq \MNC$ by conjugation splits into separate
actions of $B_n \subseteq \MnC$ by left and right multiplication, 
and $\overline{B_N\cdot \pi_<}$ is linearly isomorphic 
to a {\em matrix Schubert variety}. 
Our reference for these varieties is \cite[ch. 15]{MS}.

\begin{Proposition}\label{prop:doubschub}
  Let $\rho \in \mathcal{S}_n$, and construct $\pi \in \SN$ as above. Then
  the Joseph--Melnikov polynomial $J_\pi$ is essentially the double Schubert
  polynomial $S_\rho$ of $\rho$:
  \begin{equation*} J_\pi = 
  S_\rho(\A+z_n,\ldots,\A+z_1;\ z_{n+1},\ldots,z_{2n}) 
  \prod_{1\leq i<j \leq n} (\A+z_i-z_j) 
  \prod_{n+1\leq i<j \leq 2n} (\A+z_i-z_j) 
  \end{equation*}
\end{Proposition}

\begin{proof}
  One definition of the double Schubert polynomial is as the multidegree of
  $\overline{B_{n,-} \rho B_n} \subseteq \MnC$, where $B_-$ denotes the
  lower triangular matrices, with respect to the $2n$-dimensional torus
  formed from the diagonal matrices in $B_{n,-}$ and $B_n$. 
  One could separately include the action of the scaling circle, too,
  but this is not traditional, since the scalar matrices in $B_n$ (or in $B_{n,-}$)
  already give the scaling action.
  
  To compare these two multidegrees, we must first relate the two
  ambient spaces. The Joseph--Melnikov polynomial is defined using the
  space of strictly upper triangular $2n\times 2n$ matrices, 
  which we identify with the space of $n\times n$ matrices (thought of
  as the upper right corner) times two triangles' worth of matrix entries.
  Those triangles account for the latter two factors in the formula. The action by conjugation,
  when restricted to these upper-right square matrices, becomes left and right action by upper triangular
  matrices: therefore one must additionally reverse the indices $i\mapsto n+1-i$ of the rows to recover
  matrix Schubert varieties.  
  
  To compare the action of the torus in $\MNC$ plus the scaling circle
  to the action of the diagonal matrices in $B_{n,-}$ and $B_n$, it is
  easiest to look at the weight of the $(i,j)$ entry on $\MnC$:
  for the first torus the weight is $\A + z_i - z_j$, 
  for the second the weight is $x_{n+1-i} - y_j$ 
  (with respect to the usual notation, in which the multidegree is
  $S_\rho(x_1,\ldots,x_n;\ y_1,\ldots, y_n)$). This suggests the
  variable substitution we used in the given formula.

  (This substitution is not unique; for example,
  $S_\rho(z_n,\ldots,z_1;\ z_{n+1}-\A,\ldots,z_{2n}-\A)$
  would work equally well.
  The nonuniqueness can be traced to the
  fact that the torus in $\MNC$ does not act faithfully
  on the subspace containing $B_N\cdot \pi_<$.)
\end{proof}

Recall that multidegrees have an automatic positivity property
\cite{MS}: the multidegree of $X \subseteq V$ is a positive sum of
products of $T$-weights from $V$ (with no more repetition of weights than 
occurs in $V$). One reason to move beyond Joseph polynomials to the
larger family of Joseph--Melnikov polynomials is to give an 
inductive formula for them that is {\em manifestly positive\/} in this sense. 
The following is adapted from \cite{R}, which deals with the
subtler case of arbitrary $B_N$-orbits in $\{M^2=0\}$.

\junk{
First, we factor out a trivial contribution to $J_\pi$. Let $Q(\pi)$
denote the subspace of $\MMN$ 
\begin{equation*} 
  Q(\pi) := \{ M \ :\ \forall i,j, M_{ij} = 0 \text{ unless } \exists a,b
  \text{ such that } i\leq a < b\leq j, \pi(a)=b \}. \end{equation*}
That is to say, only keep those matrix entries that have a $1$ from $\pi_<$
to their southwest. It is easy to see that $B_N\cdot \pi_< \subseteq Q(\pi)$,
so we can define
\begin{equation*} 
  J'_\pi := \mdeg_{Q(\pi)} \overline{B_N\cdot \pi_<}. \end{equation*}
Then $J_\pi$ factors as $J'_\pi$ times $\mdeg_{\MMN} Q(\pi)$, 
where the latter is a product of $T$-weights.}

\begin{Theorem}\label{thm:JMformula}\cite{R}
  Let $\pi$ be an involution, and $a<b$ a minimal chord in $\pi$, 
  i.e., $\pi(a)=b$ and $\not\exists c,d$ with $a<c<d<b$, $\pi(c) = d$.

  Let $\rho$ vary over the set of involutions such that $\pi$ covers $\rho$
  in the poset of $B_N$-orbits, and there is no chord connecting $a,b$. Then
  for each such $\rho$ we have $\A+z_a-z_b \divides J_\rho$, and

\junk{  \end{equation*} J'_\pi 
  = \sum_\rho \text{\bf some annoying product of weights} \cdot 
J'_\rho \end{equation*}
  from which one can calculate $J_\pi$.}
\begin{equation*}J_\pi=\sum_\rho  {J_\rho\over \A+z_a-z_b}.\end{equation*}
\end{Theorem}

\begin{proof}[Proof sketch.]
  Part of this is quite direct from
  the properties we used to define multidegrees.
  We slice $\overline{B_N\cdot \pi_<}$ with the hyperplane $\{ M_{ab} = 0 \}$,
  which does not contain it since $\pi(a)=b$.
  By the other condition on $a,b$, 
  the intersection is again $B_N$-invariant. 

  Hence the intersection is supported on $\bigcup_\rho \overline{B_N\cdot \rho}$,
  and it remains to check that the multiplicities are all $1$, 
  a tangent space calculation done in \cite{R}. 

  (It is interesting to note that the same construction, applied to
  more general $B_N$-orbit closures in $\{M^2=0\}$, can produce
  multiplicities $1$ or $2$.)  
\end{proof}

Combining Theorem \ref{thm:JMformula} with Proposition \ref{prop:doubschub},
one obtains an inductive positive formula for double Schubert polynomials. 
This turns out to be exactly the ``transition formula'' of Lascoux
\cite{lascoux}.

\Ex Consider the case of maximal rank in size $N=2n=4$.  There are $3$
involutions of $\{1,2,3,4\}$ without fixed points, which we denote
according to their cycles: $(12)(34)$, $(14)(23)$, $(13)(24)$. The
first two are noncrossing, while the third one has one crossing.  The
corresponding Joseph--Melnikov polynomials are
\begin{align*}
  J_{(12)(34)}&=(\A+z_2-z_3)(2\A+z_1-z_4) &\qquad M_{23}=0\quad (M^2)_{14}=0\\
  J_{(14)(23)}&=(\A+z_1-z_2)(\A+z_3-z_4) &\qquad M_{12}=M_{34}=0\\
  J_{(13)(24)}&=(\A+z_1-z_2)(\A+z_2-z_3)(\A+z_3-z_4) 
  &\qquad M_{12}=M_{23}=M_{34}=0.
\end{align*}
(These subvarieties are complete intersections, using the indicated
equations on $\{M_{ij}\}$, 
so property 3(c) of $\mdeg$ explains why these multidegrees factor.) 
The first two correspond to orbital varieties, whereas the third
corresponds to a higher codimension $B_N$-orbit. The last two form the
``permutation sector'', that is, once divided by their common factor
$(\A+z_1-z_2)(\A+z_3-z_4)$, we recover the specializations of
the double Schubert polynomials
$S_1=1$ and $S_2(\A+z_2,\A+z_1;z_3,z_4)=\A+z_2-z_3$.


\section{Specialization of multidegrees and normal cones} 
\label{sec:ncone}

In \S \ref{ssec:leading} we prove some general results about
multidegrees in normal cones, without explicit reference to the Brauer
loop scheme or orbital varieties. In \S \ref{ssec:specializing}
we give a divisibility/vanishing condition on a multidegree.
In \S \ref{ssec:application} we use these 
results to relate Brauer loop polynomials and Joseph--Melnikov polynomials,
and to prove geometrically some algebraic results from \cite{DFZJ06}.

\subsection{The leading form of a multidegree}\label{ssec:leading}

Fix two finite-dimensional (complex) vector spaces $V,W$, 
and let $\complexes^\times$
act on $V\oplus W$ with weight $0$ on $V$ and weight $1$ on $W$.
Since we will be interested in it geometrically rather than linearly,
we will usually denote this space $V\times W$.

\begin{Lemma}\label{lem:projintersect}
  Let $X \subseteq V\times W$ be a closed
  $\complexes^\times$-invariant subscheme of $V\times W$.
  Then the projection of $X$ to $V$ is $X\cap V$ 
  (and is, in particular, closed).
\end{Lemma}

\begin{proof}
  Let $\pi: V\times W \to V$ denote the projection, which acts as
  the identity on $V$.
  Then $\pi(X) \supseteq \pi(X\cap V) = X\cap V$, proving one inclusion.

  For the opposite inclusion, let $(\vec v,\vec w) \in X$, 
  with $\pi(\vec v,\vec w) = (\vec v,\vec 0)$.
  Then by the invariance, $t\cdot (\vec v,\vec w) \in X$ for all
  $t\in\complexes^\times$, and $t\cdot (\vec v,\vec w) = (\vec v, t\vec w)$.

  By the assumption that $X$ is closed, we know that the limit
  $\lim_{t\to 0} (\vec v,t\vec w) \in X$. That limit is $(\vec v,\vec 0)$.
  This proves the opposite inclusion.
\end{proof}

If we drop the invariance, the conclusions fail: let $X$ be the hyperbola
$vw=1$ in the plane, whose intersection with $w=0$ is empty, and
whose projection to the $V$-axis hits everything but $0$.

There is a normal cone implicit in the conditions in lemma
\ref{lem:projintersect}, in that $X \iso C_X (X\cap V)$. However, we
will stick to the language used in that lemma for the rest of the section.

In addition,
in the rest of this section $V,W$ will carry actions of a torus $T$,
making $V\oplus W$ a representation of $T\times \complexes^\times$.
Let $\B$ be a generator of the weight lattice of $\complexes^\times$.
Then multidegrees of $(T\times \complexes^\times)$-invariant subschemes 
of $V\times W$ are elements of the polynomial ring $(\Sym T^*)[\B]$.

In the following Proposition, we will often want to pick out those
terms of an element of $(\Sym T^*)[\B]$ carrying the leading power of $\B$.
\newcommand\initB{{\rm init}_\B}
Denote this operator by $\initB$, 
for example $\initB(\B x + \B^2 y + \B^2 z) = \B^2(y+z)$,
and call $\initB p$ the \B{\dfn -leading form of $p$}.

\begin{Proposition}\label{prop:mdegproj}
  Let $X \subseteq V \times W$ be a $T$-invariant irreducible subvariety of
  a direct sum of $T$-representations, and $\Pi$ the projection of $X$
  to $V$. (Hence $\Pi$ is closed by Lemma \ref{lem:projintersect}, 
  and irreducible.) Let $\initB \mdeg_{V\times W} X = \B^p m$,
  $m\in \Sym T^*$.

  Then $p = \dim W + \dim \Pi - \dim X$, and $m$ is a positive
  integer multiple of $\mdeg_V \Pi$.
\end{Proposition}

\begin{proof}
  The proof is by induction on $\dim W$. The base case $\dim W = 0$ 
  is easy enough: $X = \Pi$, $p=0$, and $m_p = 1 \mdeg_V \Pi$.

  For $\dim W>0$, let $H$ be a $T$-invariant hyperplane in $W$,
  with $\lambda$ the $T$-weight on $W/H$. 
  There are two cases to consider: $X \subseteq H$ (easy)
  or $X \not\subseteq H$ (harder).

  If $X \subseteq H$, 
  then $\mdeg_{V\times W} X = (\B + \lambda) \mdeg_{V\times H} X$,
  hence $\initB \mdeg_{V\times W} X = \B\ \initB \mdeg_{V\times H} X$.
  By the inductive hypothesis, 
  $\initB \mdeg_{V\times H} X$ is a positive multiple 
  of $\B^{\dim H + \dim \Pi - \dim X} \mdeg_V \Pi = \B^{p-1} \mdeg_V \Pi$.
  Hence $\B\, \initB \mdeg_{V\times H} X$
  is $\B \cdot \B^{p-1} \mdeg_V \Pi = \B^p \mdeg_V \Pi$ as claimed.

  If $X \not\subseteq H$, 
  then $\mdeg_{V\times W} X = \mdeg_{V\times H} (X\cap H)$.
  Let $\{X_i\}$ be the 
  components of $X\cap H$ of
  dimension $\dim X - 1$, appearing with multiplicities $m_i$,
  so $\mdeg_{V\times H} (X\cap H) = \sum_i m_i \mdeg_{V\times H} X_i$.
  Note that $\dim W - \dim X = \dim H - \dim X_i$.

  By Lemma \ref{lem:projintersect}, the projection of 
  $X\cap H$ to $V$ is $\Pi$. Hence the projection of any $X_i$ to $V$,
  call it $\Pi_i$, is contained in $\Pi$. 
  Since $\Pi$ is irreducible, either $\Pi_i = \Pi$ or is of lower dimension.
  In the latter case, 
  $\dim H + \dim \Pi_i - \dim X_i < \dim W + \dim \Pi - \dim X$,
  so by the induction hypothesis this term $m_i \mdeg_{V\times H} X_i$
  does not contribute to the $\B$-leading term of 
  $\sum_i m_i \mdeg_{V\times H} X_i$.

  So far we have shown that 
  \[ \initB\, \mdeg_{V\times W} X 
  = \initB \sum_i m_i\, \mdeg_{V\times H} X_i 
  = \sum_i m_i\, \initB\, \mdeg_{V\times H} X_i \]
  where we sum over only those $X_i$ that project onto $\Pi$.
  By the induction hypothesis, $\initB \mdeg_{V\times H} X_i$ 
  is a positive multiple of $\B^{\dim H + \dim \Pi - \dim X_i} \mdeg_V \Pi$.
  So every term in the sum has $\B$-leading term a positive multiple of
  $\B^p \mdeg_V \Pi$, and adding them together we get a positive multiple
  of $\B^p \mdeg_V \Pi$, as claimed.
\end{proof}

In fact $X$ was assumed irreducible only to ensure that $\Pi$ is
irreducible. 
\junk{
  If $\Pi$ is not irreducible, 
  then nearly the same proof shows that $\initB \mdeg_{V\times W} X$
  is a positive combination of the multidegrees of the components of $\Pi$.
}

With more use of equivariant cohomology, one can identify the positive 
integer multiple as follows (a statement we neither prove nor use). 
Pick a general point of $\Pi \subseteq V$,
and look at its preimage $P$ in $X$, thought of as a subvariety of $W$.
This preimage is automatically $\complexes^\times$-invariant, so defines
a projective variety, whose degree is the desired multiple.
In particular, the multiple is $1$ (as it will be in our application)
if and only if $P$ is a linear subspace.

\subsection{Specializing to stabilizer subgroups}\label{ssec:specializing}

If one defines the multidegree topologically (rather than by the
inductive definition from \S\ref{ssec:multidegrees}),
the following lemma is obvious.

\begin{Lemma}\label{lem:translate}
  Let $X \subseteq V$ be a $T$-invariant subscheme in a $T$-representation $V$,
  and $\vec v \in V^T$ be a $T$-invariant vector. 
  \begin{enumerate}
  \item The translate $X + \vec v$ is also $T$-invariant, with
    the same multidegree as $X$.
  \item If $X \not\ni \vec v$, then $\mdeg_V X = 0$.
  \item If $X$ contains and is smooth at $\vec v$, 
    then $\mdeg_V X = \prod ($the weights in $V / T_{\vec v} X)$.
  \end{enumerate}
\end{Lemma}

\begin{proof}
  \begin{enumerate}
  \item 
    The invariance is obvious. 
    For the equality of multidegrees, we proceed by induction
    on dimension. In any dimension, we can reduce to the case that $X$
    is a variety by axiom (2) of multidegrees. Hereafter we exclude the
    trivial case $\vec v = \vec 0$ (so in particular $\dim V > 0$).
    
    If $V$ is one-dimensional (the base case), then it is the trivial
    representation, and $X$ is either $V$ or a finite set. Then the
    same is true for $X + \vec v$, giving the matching multidegrees $1$ or $0$
    respectively.
    
    If $\dim V > 1$,
    we can pick a $T$-invariant hyperplane $H$ containing $\vec v$.
    If $H \supseteq X$, then $H \supseteq (X + \vec v)$, and by axiom 
    (3b) both have zero multidegree.
    
    Otherwise $\mdeg_V X = \mdeg_H (X\cap H)$ and
    $\mdeg_V (X + \vec v) = \mdeg_H ((X+ \vec v) \cap H)$ by axiom (3a).
    Since $X \cap H, (X + \vec v) \cap H$ satisfy the conditions of the
    lemma, and are lower-dimensional schemes than $X$,
    by induction their multidegrees are the same.
  \item
    By part (1), we may replace $X$ by $X - \vec v$, reducing to the
    case $\vec v = \vec 0$. If $\dim V = 0$, then $X = \emptyset$, so
    $\mdeg_V X = 0$. If $\dim V > 0$, it contains a $T$-invariant
    hyperplane $H$. Then $X \cap H$ lives in a smaller-dimensional space
    but still does not contain $\vec 0$, so by induction
    $\mdeg_H (X\cap H) = 0$, and $\mdeg_V X$ is either equal to 
    or a multiple of that.
  \item 
    As in part (2), we reduce to the case that $\vec v = \vec 0$.
    By the result of part (2), we can shrink $X$ to the union of its
    primary components passing through the origin. By the smoothness
    hypothesis, there is only one component, so $X$ is now reduced
    and irreducible. 

    If $\dim X = 0$, then $X = \{\vec 0\}$, and $\mdeg_V X$ is
    the product of all weights on $V$. If $\dim X > 0$, then we
    can pick a $T$-invariant complement to $T_{\vec 0} X$ that is
    {\em properly} contained in $V$,
    and extend it to a $T$-invariant hyperplane $H \leq V$.
    Then $X\cap H$ is transverse at $\vec 0$, making $X\cap H$ 
    again smooth at the origin, so
    $\mdeg_V X = \mdeg_H (X \cap H) = 
    \prod ($the weights in $H / T_{\vec 0} (X \cap H)) = 
    \prod ($the weights in $V / T_{\vec 0} X)$, with the
    middle equality by induction on $\dim X$.
  \end{enumerate}
  
  (The topological proofs are as follows. The family 
  $\{X + \lambda \vec v\}_{\lambda \in [0,1]}$ is an equivariant homology
  between $X$ and $X+\vec v$, so $X$ and $X+\vec v$ define the same 
  equivariant cohomology class on $V$. Now assuming $\vec v = \vec 0$,
  the family $\{\lambda^{-1} X\}_{\lambda \in (0,1]}$ is an equivariant
  Borel-Moore homology of $X$ to its $\lambda \to 0$ limit.
  If $\vec 0 \notin X$, this limit is the empty set defining the class $0$.
  If $\vec 0 \in X$, this limit is $T_{\vec 0} X$, whose 
  associated class is the product indicated.)
\end{proof}

The following lemma (which does not involve normal cones) is stated
in a somewhat roundabout way, as the proof indicates; however, this
formulation is how it will actually be used in practice.

\begin{Lemma}\label{lem:multvanish}
  Let $X \subseteq V$ be a $T$-invariant subvariety of a $T$-representation,
  and $\vec v \in V$. Let 
  \[ S := \{t \in T : t\vec v \in \complexes \vec v \} \qquad \geq \qquad
  S_0 := \{t \in T : t\vec v = \vec v \} \]
  be the projective and affine stabilizers of $\vec v$.
  Let $\mdeg^T_V X, \mdeg^S_V X,\mdeg^{S_0}_V X$ denote the $T$-, $S$-,
  and $S_0$-multidegrees of $X$. 
  \begin{enumerate}
  \item There are natural maps $\Sym T^* \onto \Sym\ S^* \onto \Sym\ S_0^*$ from
    the specialization of variables $T^* \onto S^* \onto S_0^*$,
    and they take $\mdeg^T_V X \mapsto \mdeg^S_V X \mapsto \mdeg^{S_0}_V X$.
  \item If $\vec v \notin X$, then
    $\mdeg^S_V X$ is a multiple of the $S$-weight on $\vec v$, 
    and $\mdeg^{S_0}_V X = 0$.
  \item If $\vec v$ is a smooth point in $X$, 
    then $\mdeg^{S_0}_V X = \prod \{$ the $S_0$-weights in 
    $V / T_{\vec v} X \}$.
  \end{enumerate}
\end{Lemma}

\begin{proof}
  The image of $\mdeg_V X$ in $\Sym\, S^*$ is just the $S$-multidegree. 
  So we may as well restrict to the $S$-action from the beginning.
  By axiom (2) of multidegrees, we may assume $X$ is reduced and irreducible.

  The first statement follows trivially from the axioms inductively
  defining multidegrees. 

  Let $L$ be a one-dimensional $S$-invariant subspace containing $\vec v$.
  (If $\vec v \neq \vec 0$, then $L$ is of course unique.)
  Let $H$ be an $S$-invariant complementary hyperplane. 
  Let $\Pi \subseteq H$ be the closure of the image of the projection
  of $X$, and $d$ the degree of the projection map $X \to \Pi$.

  Now apply \cite[Theorem 2.5]{KMY}, which says that
  $\mdeg_V X = d\ \mdeg_V \Pi$, plus a term that vanishes under
  the assumption $\complexes \vec v \not\subset X$.
  Then by property (3a) of multidegrees, $\mdeg_V \Pi$ is a multiple
  of the $S$-weight on $L$, and its $S_0$-multidegree is a multiple
  of the $S_0$-weight on $L$, which is zero.

  The third statement is just an application of Lemma \ref{lem:translate}.
\end{proof}

This result gives the most information when $S$ is as large as possible,
which means $\vec v$ is very special. Unfortunately it seems that it
is then very likely to be in $X$, and not be a smooth point.
So we give a criterion for this smoothness:

\begin{Lemma}\label{lem:smoothpoint}
  Let $V$ be a scheme carrying an action of a group $G$.
  ($V$ need not be a vector space.) 
  Let $Y\subseteq V$ be an irreducible subvariety, 
  and let $X = \overline{G\cdot Y}$ (with the reduced scheme structure).

  Then a general point $y\in Y$ is a smooth point of $X$, 
  and its tangent space $T_y X$ is $\lie{g} \cdot T_y Y$.
\end{Lemma}

\begin{proof}
  Assume the contrary: then every point in $Y$ is a singular point
  of $X$. Since $G$ acts by automorphisms, all of $G\cdot Y$ consists
  of singular points in $X$, and by semicontinuity $X$ consists
  only of singular points. But since $X$ is reduced, its smooth 
  locus is open dense, a contradiction.

  Let $Y^\circ \subseteq Y$ be the open subset of $y\in Y$ with the
  minimum dimensional $G$-stabilizer (among points in $Y$).
  Hence $y$ is not in the closure of any $G$-orbit on $Y$. 
  Our genericity condition on $y$ will be that $y$ is a smooth 
  point of $Y^\circ$.
  So $T_y X = T_y (G\cdot Y^\circ)$. By the definition of $Y^\circ$,
  the map $G \times Y^\circ \onto G \cdot Y^\circ$ is a submersion,
  enabling us to compute the tangent spaces of the target.
\end{proof}

\begin{Proposition}\label{prop:vanishcriterion}
  Let $V$ be a representation of a group $G$,
  with $Y$ a subspace fixed pointwise by a torus $S_0$,
  i.e., $Y \leq V^{S_0}$. 
  Then 
  \[ \mdeg^{S_0}_V \overline{G\cdot Y} \neq 0 
  \quad\Longleftrightarrow\quad
  \overline{G\cdot Y} = \overline{G\cdot V^{S_0}}
  \quad\Longleftrightarrow\quad
  \overline{G\cdot Y} \supseteq V^{S_0} 
  \quad \text{(e.g. if $Y = V^{S_0}$)}. \]
  More specifically, $\mdeg^{S_0}_V \overline{G\cdot Y}
  = \prod \left( \text{the $S_0$-weights in } V/(\lie{g} \cdot Y)\right)$.
\end{Proposition}

\begin{proof}
  Since $Y \leq V^{S_0}$, we know
  $\ol{G\cdot Y} \subseteq \ol{G\cdot V^{S_0}}$.
  In general, $\overline{G\cdot A} \supseteq B$
  iff $\overline{G\cdot A} \supseteq \overline{G\cdot B}$.
  Together these establish the equivalence of the latter two conditions.

  The first condition implies the third by Lemma \ref{lem:multvanish} part 2.
  All that remains for the equivalences is to show the second implies the first.

  Before doing that, we compute 
  $\mdeg^{S_0}_V \overline{G\cdot Y}$ exactly, using 
  Lemma \ref{lem:multvanish}, part 3, with $\vec v = y$; 
  this gives the claimed formula.

  The $S_0$-weights in $V/(\lie{g} \cdot Y)$ form a subset of the
  $S_0$-weights in $V/Y$, and the condition $Y = V^{S_0}$ says that
  none of these are zero, so the product is nonzero.
\end{proof}

We will actually use very little of Proposition \ref{prop:vanishcriterion};
we felt it was worth including because many of the varieties
for which multidegrees have been computed are of this form
$\overline{G\cdot V^{S_0}}$ (e.g. \cite{KZJ,KM}).
\junk{
  A particularly interesting case of this setup occurs when $S_0$ is 
  contained inside a Borel subgroup $B_G$ of $G$, under which $Y$
  is invariant; then the natural map $G \times^{B_G} Y \to \overline{G\cdot Y}$
  is proper, and is called a {\dfn Kempf collapsing} 
  (see e.g. \cite{KempfCollapsing}).
}

\subsection{Application to the Brauer loop scheme}\label{ssec:application}

For simplicity, we assume $N$ to be even, $N=2n$.
\newcommand\hatE{{\widehat E}}
Recall from the introduction that $R_{\integers \bmod N}$ 
denotes the algebra of infinite periodic upper triangular matrices,
and contains $R_N$ as a subalgebra in a natural way.
Let $R_{\integers \bmod N}^\times =B_{\integers \bmod N} \geq R_N^\times = B_N$ denote
their multiplicative groups,
which consist of those matrices with nonzero diagonal entries.
Then the conjugation action of $B_{\integers \bmod N}$ on
$R_{\integers \bmod N}$ descends to an action on the Brauer loop scheme.

Consequently the projection map in \S\ref{ssec:orbvars},
from the Brauer loop scheme to the orbital scheme, is $B_N$-equivariant. 
This fact was used in \cite{KZJ} to determine the 
top-dimensional components of the Brauer loop scheme.

\begin{Theorem}\label{thm:PsiJM}
  Let $\pi \in \SN$ be an involution with no fixed points, and $E_\pi$
  the corresponding Brauer loop variety. Then the dimension of the
  projection of $E_\pi$ to $\overline{B_N\cdot \pi_<}$ is $n^2-c$,
  where $c$ is the number of crossings in $\pi$'s chord diagram.

  The $\B$-leading form $\initB \Psi_\pi$ of the Brauer loop
  polynomial $\Psi_\pi$ is $\B^{n^2-n-c} J_\pi$, where $J_\pi$ is the
  corresponding Joseph--Melnikov polynomial.
\end{Theorem}

\begin{proof}
  Let $\Pi$ denote the image of the projection.
  By the $B_N$-equivariance, $B_N\cdot \pi_< \subseteq \Pi$, and since
  $\Pi$ is closed, $\overline{B_N\cdot \pi_<} \subseteq \Pi$.
  Then we apply Theorem \ref{thm:orbdim}.

  For the latter statement, we use the decomposition of $\MMN^{\Delta=0}$
  into $R_N$ plus its unique $T$-invariant complement;
  these will be the $V$ and $W$ of Proposition \ref{prop:mdegproj}.
  The predicted exponent $\dim W + \dim \Pi - \dim X$ is then
  ${2n \choose 2} + (n^2 - c) - 2n^2 = n^2-n-c$.  
\end{proof}

\Ex We list the Brauer loop polynomials in size $N=2n=4$:
\begin{align*}
\Psi_{(12)(34)}=&(\A+z_2-z_3)(\B+z_4-z_1)\\
&(\A^2+2\A\B+\B z_1-\A z_2-z_1z_2+\A z_3+z_2z_3-\B z_4+z_1z_4-z_3z_4)\\
\Psi_{(14)(23)}=&(\A+z_1-z_2)(\A+z_3-z_4)\\
&(\A^2+\A\B+\B^2-\B z_1+\A z_2+z_1z_2-\A z_3-z_2z_3+\B z_4-z_1z_4+z_3z_4)\\
\Psi_{(13)(24)}=&(\A+z_1-z_2)(\A+z_2-z_3)(\A+z_3-z_4)(\B+z_4-z_1)\\
\end{align*}
(As the factorizations suggest,
only the last one corresponds to a complete intersection, with the
obvious linear equations.)  We leave to the reader to check that
the $\B$-leading forms are the Joseph--Melnikov
polynomials of the example of the previous section. 
Note that the cyclic shift of indices $z_i\mapsto z_{i+1}$ (recall
that $z_{i+N}=z_i+\A-\B$) exchanges the first two polynomials and leaves the
third invariant.

\begin{Proposition}(see \cite[Lemma 2]{DFZJ06})\label{prop:specializeForIndep}
  Let $\pi,\rho$ be link patterns. Then the specialization of $\Psi_\pi$ under 
  the identifications $\{\A = \B = 0, y_i = y_{\rho(i)}\}$ is zero 
  unless $\pi=\rho$,
  in which case it is nonzero. 

  Consequently, the polynomials $\{\Psi_\pi\}$ are linearly
  independent over $\integers$.
\end{Proposition}

\begin{proof}
  Let $Y_\pi = \overline{\pi T}$, the vector space of matrices with
  entries only in the same positions as the permutation matrix $\pi$.
  This $Y_\pi$ is $T$-stable, and fixed pointwise by
  $ S^\pi_0 := \{ \text{diag}(\ldots,t_i,\ldots) : t_i = t_{\pi(i)} \} . $
  By \cite[Theorem 3]{KZJ}, 
  $E_\pi = \overline{B_{\integers \bmod N} \cdot Y}$.

  Similarly define $Y_\rho$ and $S^\rho_0$, and
  let $M$ be a generic element of $Y_\rho$.
  First we prove the vanishing, where $\rho\neq \pi$.
  By \cite[Theorem 5 part 2]{KZJ}, $M \notin E_\pi$.
  Then apply Lemma \ref{lem:multvanish} part 2 to see that $\Psi_\pi$ 
  vanishes under the joint specialization $\{\A = 0, y_i = y_{\rho(i)} \}$.

  For the nonvanishing when $\pi = \rho$, we apply
  Proposition \ref{prop:vanishcriterion}. 

  Linear independence over $\integers$ is then proved by the standard
  argument: take a purported linear relation and specialize at
  $\{\A = \B = 0, y_i = y_{\rho(i)} \}$ to isolate the coefficient of $\Psi_\pi$.
  Repeat for each $\pi$.
\end{proof}

Though we shall not need it, 
it is not hard to actually compute the specialization of $\Psi_\pi$ at
$\{\A = \B = 0, y_i = y_{\rho(i)} \}$ (and is essentially done in 
\cite[Lemma 2]{DFZJ06}); it is the product $\prod (y_i - y_j)$ over
entries $(i,j)$ in the ``diagram'' of the affine permutation $\pi$,
as defined at the end of  \cite[\S 3]{KZJ}.
(It should have been called the ``Rothe diagram'' there, to 
distinguish it from the chord diagram.)

\section{Interlude: 
  Polynomial solution of the Brauer \texorpdfstring{$q$}{q}KZ equation}
\label{sec:qkz}

This section is independent from \S \ref{sec:orbvars}, \ref{sec:ncone}.
It might thus seem that we use the same notation for a priori 
unrelated quantities; but this is made in order to facilitate
the identification that will be made in \S\ref{sec:geombrauer}.

\subsection{The Brauer algebra}\label{sec:brauer}
As before, let $N$ be an even positive integer, $N=2n$.
The 
{\dfn Brauer algebra}\/ $\Br_N(\beta)$ is a $\complexes[\beta]$-algebra 
defined by generators $f_i$ and $e_i$, $i=1,\ldots,N-1$, and relations
\begin{align}\label{eqn:defbrauer}
e_i^2&=\beta e_i& e_ie_{i\pm 1}e_i&=e_i\\
\notag f_i^2&=1& f_if_{i+1}f_i&=f_{i+1}f_if_{i+1}\\
\notag f_ie_i&=e_if_i=e_i& f_i e_{i\pm 1} e_i &= f_{i\pm 1} e_i&
\notag e_i e_{i\pm 1} f_i &= e_i f_{i\pm 1}\\
\notag e_i e_j&=e_j e_i& f_i f_j&=f_j f_i& e_i f_j&=f_je_i& |i-j|>1
\end{align}
where indices take all values for which the identities make sense. It is
well-known that the 
$f_i$ generate a subalgebra isomorphic to the symmetric group 
algebra $\complexes[\mathcal{S}_N]$, whereas the $e_i$ generate a subalgebra
isomorphic to the Temperley--Lieb algebra with parameter $\beta$.

For each $i=1,\ldots,N-1$ define the $R$-matrix \cite{MR,MNR}
\begin{equation}\label{eqn:Rmat}
  \check R_i(u)
  ={\A(\A-u)+\A\, u\, e_i + (1-\beta/2)u(\A-u) f_i\over (\A+u)(\A-(1-\beta/2)u)}
\end{equation}
which is an element of $\Br_N(\beta)\otimes \complexes(\A,u)$, i.e.,
both $u$ and $\A$ are formal parameters, but we emphasize the $u$-dependence.

Using the defining relations of the Brauer algebra, one can show that the 
$\check R_i$ satisfy the {\em Yang--Baxter equation}
\begin{equation}\qquad
  \check R_i(u)\check R_{i+1}(u+v)\check R_i(v)
  =\check R_{i+1}(v)\check R_i(u+v)\check R_{i+1}(u)\qquad i=1,\ldots,N-2
\end{equation}
(as a relation in $\Br_N(\beta)\otimes \complexes(\A,u,v)$)
and the unitarity equation
\begin{equation}\qquad\qquad
  \check R_i(u)\check R_i(-u)=1\qquad\qquad i=1,\ldots,N-1.
\end{equation}

Finally we introduce a representation space for $\Br_N(\beta)$. 
Consider the vector space
$V$ with a canonical basis indexed by {\dfn link patterns}, i.e.,
involutions of
$\{1,\ldots,N\}$ without fixed points,
and draw them as pairings of points in the
upper-half plane, in such a way that they cross transversely and at
most once; such ``crossings'' in a link pattern only appear between
the pairings $(i,\pi(i))$ and $(j,\pi(j))$ if $i<j<\pi(i)<\pi(j)$ up
to switching $i\leftrightarrow\pi(i)$, $j\leftrightarrow\pi(j)$,
$(i,\pi(i))\leftrightarrow (j,\pi(j))$. Often it will be
convenient to identify $\{1,\ldots,N\}$ with $\integers/N\integers$,
in which case one can alternatively draw link patterns as pairings of
points in a disc, with the $N$ points placed in the counterclockwise
order on the boundary circle (as was done in \cite{KZJ}).

The representation of the Brauer algebra on $V\otimes \complexes[\beta]$ can be
described by the action of its generators on the canonical basis.  
The generator $f_i$ acts on a link pattern $\pi$ as the transposition
$(i,i+1)$ acting by conjugation, whereas $e_i$ acts as
\begin{equation}
e_i\cdot\pi=
\begin{cases}
\beta\, \pi &\mbox{if $\pi(i)=i+1$}\\
\pi'&\mbox{otherwise,
where $\pi'$ has cycles $(i,i+1)$,
$(\pi(i),\pi(i+1))$ and otherwise agrees with $\pi$}
\end{cases}\end{equation}
This action is best understood graphically, see Fig.~\ref{fig:lps}.
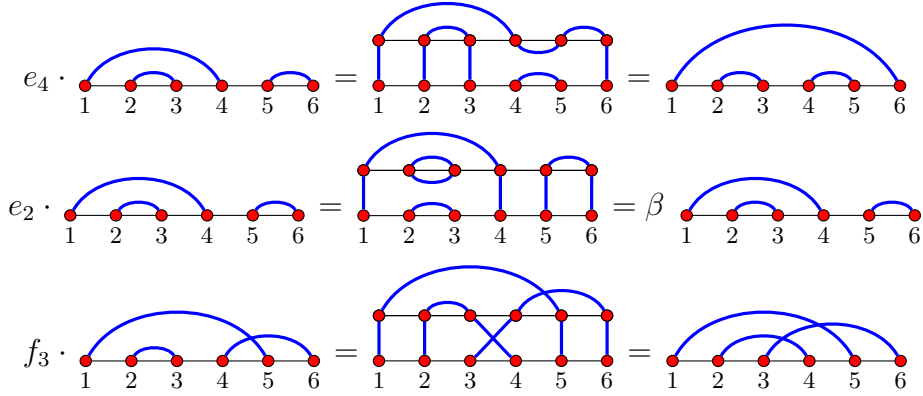
\begin{figure}
  \centering
$e_4\cdot 
\linkpattern[numbered]{1/4,2/3,5/6}
= 
\tanglelinkpattern[numbered]{1/1',2/2',3/3',4/5,4'/5',6/6'}{1/4,2/3,5/6}
= 
\linkpattern[numbered]{1/6,2/3,4/5}
$

$e_2\cdot 
\linkpattern[numbered]{1/4,2/3,5/6}
= 
\tanglelinkpattern[numbered]{1/1',2/3,2'/3',4/4',5/5',6/6'}{1/4,2/3,5/6}
=\beta\ \linkpattern[numbered]{1/4,2/3,5/6}
$

$f_3\cdot 
\linkpattern[numbered]{1/5,2/3,4/6}
=
\begin{tikzpicture}[/linkpattern/every linkpattern,baseline=-\linkpatternunit]%
\linkpattern[tikzstarted,numbered=false]{1/5,2/3,4/6}
\begin{scope}[yshift=-0.5*\linkpatternunit]
\linkpattern[tangle,,tikzstarted,size=12,numbering=halftangle,height=0.5,squareness=0]{1/1',2/2',3/4',3'/4,5/5',6/6'}
\end{scope}
\end{tikzpicture}%
= \linkpattern[numbered]{1/5,2/4,3/6}
$
  \caption{Link patterns and Brauer algebra generators acting on them.}
  \label{fig:lps}
\end{figure}

We need one more operator $r$ acting on $V$: it conjugates by the cycle
$(1 2\cdots N)$.  To avoid confusion we {\em do not\/} identify $r$
with the corresponding abstract element $f_{n-1} f_{n-2} \cdots f_2 f_1$ 
of the symmetric group (see
the remark at the end of Appendix \ref{sec:affweyl}).  
If link patterns are drawn on a
circle, then $r$ is the operator that rotates them one step counterclockwise.

\subsection{Brauer \texorpdfstring{$q$}{q}KZ equation}
Let us now introduce one more parameter $\epsilon$ and consider the
following set of equations (called rational {\dfn quantum 
Knizhnik--Zamolodchikov equation}, or difference Knizhnik--Zamolodchikov
equation, or $q$KZ equation in short):
\begin{align} \label{eqn:qkza}
\check R_i(z_{i}-z_{i+1})\Psi_N(z_1,\ldots,z_N)&=\Psi_N(z_1,\ldots,z_{i+1},z_i,\ldots,z_N)\qquad i=1,\ldots,N-1\\
r^{-1}\Psi_N(z_1,\ldots,z_N)&=\Psi_N(z_2,\ldots,z_N,z_1+\epsilon)\label{eqn:qkzb}
\end{align}
In general, 
$\Psi_N$ is a $V$-valued function of $\A,\epsilon,z_1,\ldots,z_N,\beta$, i.e.,
$\Psi_N=\sum_\pi \Psi_\pi\, \pi$, with a dependence on the parameters
that we discuss now.

We are interested here in {\em polynomial}\/ solutions 
in $\A,\epsilon,z_1,\ldots,z_N$.
On general grounds, one does not expect that there exist such
solutions of the system of equations (\ref{eqn:qkza}--\ref{eqn:qkzb})
-- that is, for generic values of the parameters $\beta$, $\A$,
$\epsilon$.  We shall see below that such a solution exists if
$\beta={2(\A-\epsilon)\over 2\A-\epsilon}$, i.e., if one considers that
$\Psi$ belongs to the quotient
$V\otimes \complexes[\A,\epsilon,z_1,\ldots,z_N,\beta]/
\left<(2\A-\epsilon)\beta-2(\A-\epsilon)\right>$.
But first we need to introduce some convenient notation and
reformulations of the $q$KZ equation.

We extend the variables $z_1,\ldots,z_N$ to an infinite set $\{ z_i,
i\in\integers\}$ by imposing the pseudo-cyclicity condition
$z_{i+N}=z_i+\epsilon$.  Define $K=\complexes[\A,\epsilon,z_i,
i\in\integers,\beta]/\<z_{i+N}-z_i-\epsilon, i\in\integers\>$.  Let
$\tau_i$ be the operator appearing in the r.h.s.\ of
Eq.~(\ref{eqn:qkza}), that in our periodic notations is the
automorphism of $K$ that exchanges variables $z_{i+kN}$ and
$z_{i+1+kN}$ for all $k\in\integers$. Similarly, define for future use
$\sigma$ to be the operator appearing in the r.h.s.\ of
Eq.~(\ref{eqn:qkzb}), that is the automorphism of $K$ that shifts
$z_i\mapsto z_{i+1}$ for all $i\in\integers$.  Finally, introduce
generators of the {\em affine\/} Brauer algebra
$\mathcal{\hat B}_N(\beta)$
 $e_i$, $f_i$ for all
$i\in\integers/N\integers$ that satisfy the same relations
(\ref{eqn:defbrauer}) as the Brauer algebra
$\hat\Br_N(\beta)$, but
with all indices modulo $N$.  In particular the $f_i$ are generators of
the affine Weyl group (of type A) $\mathcal{\hat S}_N$, 
see Appendix \ref{sec:affweyl}.  The representation 
of $\Br_N(\beta)$ on $V$ 
can be extended to $\hat\Br_N(\beta)$
in a natural way ($f_N$ being conjugation by the transposition $(1,N)$
and $e_N$ creating cycles $(1,N)$, $(\pi(1),\pi(N))$), in such a way
that (as operators on $V$) they satisfy $e_{i+1}=r e_i r^{-1}$ and
$f_{i+1}=r f_i r^{-1}$ for all $i\in\integers/N\integers$.  Then we
have the
\begin{Proposition} If $\Psi_N$ is a solution of 
  Eqs.~(\ref{eqn:qkza}--\ref{eqn:qkzb}), then
  \begin{equation}\label{eqn:qkzc}
    \check R_i(z_{i}-z_{i+1})\Psi_N=\tau_i\Psi_N\qquad i\in\integers
  \end{equation}
\end{Proposition}

\begin{proof}
  If $i\ne 0 \mod N$, then Eq.~(\ref{eqn:qkzc}) is nothing but
  Eq.~(\ref{eqn:qkza}).  Furthermore, our new notations make the
  following formulae valid: $\check R_{i+1}(z_{i+1}-z_{i+2})
  =\sigma r \check R_i(z_{i+1}-z_i) r^{-1}\sigma^{-1}$ and
  $\tau_{i+1}=\sigma \tau_i \sigma^{-1}$.  Eq.~(\ref{eqn:qkzb}),
  namely $r\sigma \Psi_N=\Psi_N$, then allows us to shift the index
  $i$ of Eq.~(\ref{eqn:qkza}) to arbitrary values.
\end{proof}

We need to rewrite Eqs.~(\ref{eqn:qkzc}) more explicitly. Expanding in
the basis, we find that for a given link pattern $\pi$, the
equation involves the preimage of $\pi$ under $e_i$ and $f_i$.
We are led to the following dichotomy:

\begin{itemize}
\item either $\pi(i)\ne i+1$, in which case the preimage of $\pi$ under $e_i$ 
  is empty, and a small calculation results in
  \begin{equation}\label{eqn:qkzd}
    \Psi_{f_i\cdot \pi} 
    = \Theta_i \Psi_\pi\qquad 
    \Theta_i := \frac{(\A+z_i-z_{i+1})(\A+(1-\beta/2)(z_{i+1}-z_{i}))\tau_i
      -\A(\A-z_i+z_{i+1})}{(1-\beta/2)(z_i-z_{i+1})(\A-z_i+z_{i+1})}
  \end{equation}
\item or $\pi(i)=i+1$, in which case $f_i\cdot\pi=\pi$ and we find this time
  \begin{equation}\label{eqn:qkze}
    \sum_{\pi'\ne\pi,e_i\cdot\pi'=\pi} \Psi_{\pi'}=\Delta_i \Psi_\pi\qquad 
    \Delta_i:= \frac{(\A+z_i-z_{i+1})(\A+(1-\beta/2)(z_{i+1}-z_i))(\tau_i-1)}
    {\A(z_{i}-z_{i+1})}
  \end{equation}
\end{itemize}

\newcommand\wtK{\widetilde K}
The $\Theta_i$ and $\Delta_i$ are operators on
$\wtK := K[(1-\beta/2)^{-1},\A^{-1},(\A-z_i+z_{i+1})^{-1},(z_i-z_{i+1})^{-1},
i\in \integers]$ with the following properties:
\begin{Lemma} \label{lem:affinesym}
The $\Theta_i$, $i\in\integers$, satisfy the affine Weyl group relations
\begin{equation}
\Theta_i^2=1\qquad (\Theta_i\Theta_{i+1})^3=1\qquad \Theta_i\Theta_j=\Theta_j
\Theta_i\quad |i-j|>1
\end{equation}
\end{Lemma}

In what follows we shall be particularly interested in the equations
(\ref{eqn:qkzd}) involving $f_i$ and $\Theta_i$.  In view of Lemma
\ref{lem:affinesym}, one can define a group morphism $\Theta: s\mapsto
\Theta_s$ from $\mathcal{\hat S}_N$ to the invertible operators on $\wtK$
such that $f_i\mapsto \Theta_i$. Note however that the action of
$\Theta_i$ on $\Psi_\pi$ only makes sense when $\pi(i)\ne i+1$. We are
therefore led to a more relevant groupoid structure.  Define
\begin{multline}\label{eqn:defred}
\mathcal{\hat S}_{\pi,\pi'}
:= \{ s\in \mathcal{\hat S}_N, s\cdot\pi=\pi' : \exists (i_1,\ldots,i_k)
\text{ such that} \\
 s=f_{i_k}\cdots f_{i_1}\text{ and }
\ \forall \ell=1,\ldots,k\ (f_{i_{\ell-1}}\cdots f_{i_1}\cdot\pi)(i_\ell)\ne i_\ell+1
\}
\end{multline}
Graphically, $\hat{S}_{\pi,\pi'}$ is the set of ``affine permutations'' 
that map $\pi$ to $\pi'$ without creating ``tadpoles'', i.e., lines
that cross themselves as in \tikz[baseline=0]{\path[use as bounding box] (-0.5,-0.5) rectangle (0.8,0.5);\draw[/linkpattern/edge] (-0.5,-0.5) .. controls (1,1.5) and (1,-1.5) .. (-0.5,0.5);}.

\begin{Proposition} \label{prop:groupoid}
  The $(\mathcal{\hat S}_{\pi,\pi'})$ form a groupoid, 
  and $\mathcal{\hat S}_{\pi,\pi'}\ne\emptyset$ $\forall \pi,\pi'$; $\Theta$ is 
  a groupoid morphism; and if $\Psi_N$ satisfies Eqs.~(\ref{eqn:qkzd}),
  \begin{equation}\label{eqn:groupoid}
    \Psi_{\pi'}=\Theta_s \Psi_\pi \qquad 
    \forall \pi,\pi'\text{ and }\forall s\in \mathcal{\hat S}_{\pi,\pi'}
  \end{equation}
\end{Proposition}

\begin{proof}
The only non-trivial statement is that $\mathcal{\hat S}_{\pi,\pi'}\ne\emptyset$ $\forall \pi,\pi'$.
Indeed we shall show that $\mathcal{S}_{\pi,\pi'} \neq \emptyset$,
where $\mathcal{S}_{\pi,\pi'}=\mathcal{\hat S}_{\pi,\pi'}\cap \mathcal{S}_N$.
As $\mathcal{S}_N$ acts transitively by conjugation on involutions
without fixed points, we may pick $s_0\in \mathcal{S}_N$ such that
$s_0\cdot\pi=\pi'$, and any decomposition of it in terms of
generators: $s_0=f_{i_k}\cdots f_{i_1}$. Consider the new word
obtained from $f_{i_k}\cdots f_{i_1}$ by removing each $f_{i_\ell}$
such that $(f_{i_{\ell-1}}\cdots f_{i_1}\cdot\pi)(i_\ell)=
i_\ell+1$. Since in this case $f_{i_\ell}\, f_{i_{\ell-1}}\cdots
f_{i_1}\cdot\pi=f_{i_{\ell-1}}\cdots f_{i_1}\cdot\pi$, the new word
defines a permutation $s$ such that $s\cdot\pi=\pi'$ but now also
satisfies the defining property of $\mathcal{\hat S}_{\pi,\pi'}$; which is
therefore non-empty.
\end{proof}

\begin{Proposition} \label{prop:zerosol}
Let $\Psi$ be a solution of Eqs.~(\ref{eqn:qkzc}) and $\pi$ a link pattern. 
If $i, j\in \integers$, $i<j$, are such that
$\pi(\{ i,i+1,\ldots,j\}) \cap \{ i,i+1,\ldots,j\} = \emptyset$ (mod $N$ is implied), 
then $\A+z_i-z_j$ divides $\Psi_\pi$.
\end{Proposition}
\begin{proof} Induction on $j-i$.

Rewrite Eq.~(\ref{eqn:qkzc}) as $\Psi_N=\check R_i(z_{i+1}-z_i)\tau_i \Psi_N$,
and look at its component $\pi$,
noting that $\check R_i(z_{i+1}-z_i)=(\A+z_i-z_{i+1})(F_1 + F_2 f_i)+ F_3 e_i$
where $F_1,F_2,F_3$ are some rational fractions of $z_i,z_{i+1},\A$ without pole at $z_{i+1}=\A+z_i$. The hypothesis
implies $\pi(i)\ne i+1$, i.e\ $\pi \not\in Im\, e_i$,
so that we can ignore the third term.

If $j-i=1$, we conclude directly that $\A+z_i-z_{i+1} \divides \Psi_\pi$.

If $j-i>1$, we note that both $\pi$ and $f_i\cdot\pi$ satisfies the hypotheses for the interval $\{i+1,\ldots,j\}$
so that we can use the induction to conclude that $\A+z_{i+1}-z_j$ divides both $\Psi_\pi$ and $\Psi_{f_i\cdot\pi}$,
or equivalently that $\A+z_i-z_j$ divides both $\tau_i \Psi_\pi$ and $\tau_i\Psi_{f_i\cdot\pi}$, hence also $\Psi_\pi$.
\end{proof}

\subsection{Solution of \texorpdfstring{$q$}{q}KZ equation}
We can now state the main theorem of this section:
\begin{Theorem} \label{thm:qkzsol} 
  There exists a solution of Eqs.~(\ref{eqn:qkza}--\ref{eqn:qkzb})
where $\Psi\in V\otimes K/\left<(2\A-\epsilon)\beta-2(\A-\epsilon)\right>$;
more specifically, 
substituting $\beta={2(\A-\epsilon)\over 2\A-\epsilon}$, there exists
a solution of Eqs.~(\ref{eqn:qkza}--\ref{eqn:qkzb}) such that its components
 $\Psi_\pi$ are homogeneous polynomials of degree $2n(n-1)$
{\em with integer coefficients}\/
in the variables $\A,\epsilon,z_1,\ldots,z_N$.
  This solution is unique up to scaling. 
\end{Theorem}

\begin{proof}
  The proof will be similar to the statement of \cite{DFZJ06} that it
  generalizes.  First, introduce the ``base'' link pattern $\pi_0$:
  $\pi_0(i)=i+n,\quad i=1,\ldots,n$. Applying
  Prop.~\ref{prop:zerosol}, we find a product of $2n(n-1)$ factors in
  $\Psi_{\pi_0}$.  To ensure the correct degree, we must impose (up to
  numerical normalization):
  \begin{equation}\label{eqn:basecase}
    \Psi_{\pi_0}=\prod_{i=1}^N \prod_{j=i+1}^{i+n-1} (\A+z_i-z_j)
  \end{equation}
  or in terms of the original variables,
  $\Psi_{\pi_0}=
  \prod_{1\leq i<j\leq 2n\atop j-i< n} (\A+z_i-z_j)
  \prod_{1\leq i<j\leq 2n\atop j-i> n} (\B+z_j-z_i)$, 
  with the convenient notation $\B:=\A-\epsilon$.
  
  Next we have the following elementary lemma:
  \begin{Lemma} \label{lem:qkzuniq}
    $\Psi_N$ is determined uniquely by Eq.~(\ref{eqn:basecase}) 
    and Eqs.~(\ref{eqn:qkzd}).
  \end{Lemma}
  
  \begin{proof}
    If Eqs.~(\ref{eqn:qkzd}) are satisfied, we can apply
    Prop.~\ref{prop:groupoid} and state that
    \begin{equation}\label{eqn:defpsi}
      \Psi_{\pi}=\Theta_s \Psi_{\pi_0} \qquad 
      \forall \pi\text{ and }\forall s\in \mathcal{\hat S}_{\pi_0,\pi}
    \end{equation}
    with $\mathcal{\hat S}_{\pi_0,\pi}\ne\emptyset$. Thus, if $\Psi_{\pi_0}$
    is fixed by Eq.~(\ref{eqn:basecase}), all $\Psi_\pi$ are fixed.
  \end{proof}
  
  Now that we have proved uniqueness, we want to show existence of a
  solution. We thus consider Eqs.~(\ref{eqn:basecase}--\ref{eqn:defpsi})
  as {\em defining\/} the entry $\Psi_\pi$, and all we need to
  prove is that this definition is independent of the choice of $s$.
  Equivalently we need to show that
  \begin{equation}\label{eqn:thetacheck}
    \Psi_{\pi_0}=\Theta_s \Psi_{\pi_0}\qquad
    \forall s \in \mathcal{\hat S}_{\pi_0,\pi_0}.
  \end{equation}
  We need the additional
  
  \begin{Lemma}\label{lem:stab}
    $\mathcal{\hat S}_{\pi_0,\pi_0}$ is the subgroup of $\mathcal{\hat S}_N$
    generated by the $f_i f_{i+n}$, $i\in\integers/N\integers$.
  \end{Lemma}
  The proof is in Appendix \ref{sec:affweyl}.
  
  We thus need to show that $\Theta_i\Psi_{\pi_0}=\Theta_{i+n} \Psi_{\pi_0}$, 
  which we do using the following alternative formula for $\Theta_i$:
  \begin{equation}\label{eqn:alttheta}
    \Theta_i=-1-\Big(\frac{\A}{1-\beta/2}+z_{i+1}-z_i\Big)
    (\A+z_i-z_{i+1})
    \der_i {1\over \A+z_i-z_{i+1}}\qquad 
    \der_i:={1\over z_{i}-z_{i+1}}(1-\tau_i)
  \end{equation}
  This $\der_i$ is a {\dfn divided difference operator}; 
  it has the important property that $\der_i(f g)=f\ \der_i g$ if $f$ is
  symmetric in $z_i$, $z_{i+1}$ (equivalently, if $\der_i f=0$). 
  From Eq.~(\ref{eqn:basecase}) we deduce that
  \begin{equation*}
    \Psi_{\pi_0}=(\A+z_i-z_{i+1})(\A+z_{i+1}-z_{i+n})(\A+z_{i+1+n}-\epsilon-z_i)(\A+z_{i+n}-z_{i+n+1})\times S
  \end{equation*}
  where $S$ is symmetric in both
  $z_i\leftrightarrow z_{i+1}$ and $z_{i+n}\leftrightarrow z_{i+n+1}$. Performing the computation we find
  \begin{equation*}
    \Theta_i\Psi_{\pi_0}-\Theta_{i+n} \Psi_{\pi_0}=\frac{-\A+\A\beta+\epsilon-\beta\epsilon/2}{1-\beta/2}(\A+z_i-z_{i+1})(\A+z_{i+n}-z_{i+n+1})(z_i-z_{i+1}-z_{i+n}+z_{i+n+1})S
  \end{equation*}
  which is zero if the condition $\beta=\frac{2(\A-\epsilon)}{2\A-\epsilon}$ is satisfied.
  
  Thus, $\Psi_N$ is well-defined. Furthermore, one notes that when one
  computes $\Psi_\pi$ using Eq.~(\ref{eqn:defpsi}), i.e., by acting
  with successive $\Theta_{i_\ell}$, due to the definition of
  $\mathcal{\hat S}_{\pi_0,\pi}$, each $\Theta_{i_\ell}$ acts on a
  $\Psi_{\pi'}$ (with $\pi'=f_{i_{\ell-1}}\cdots
  f_{i_\ell}\cdot\pi_0$) such that $\pi'(i_\ell)\ne i_\ell+1$.
  Therefore we can apply Prop.~\ref{prop:zerosol}, which says that
  $\A+z_{i_\ell}-z_{i_\ell+1}\ |\ \Psi_{\pi'}$, and conclude using the
  alternative form (\ref{eqn:alttheta}) for $\Theta_i$ (with
  $\A/(1-\beta/2)=\A+\B$) that the polynomial character of
  $\Psi_\pi(z_1,\ldots,z_N,\A,\B)$ and integrality of its coefficients
  are preserved by the successive actions of the $\Theta_i$.  Since
  the operators $\Theta_i$ are also degree-preserving, all $\Psi_\pi$
  are of degree $2n(n-1)$.
  
  We now want to show that $\Psi$ is a solution of
  Eqs.~(\ref{eqn:qkza}--\ref{eqn:qkzb}).
  
  Eq.~(\ref{eqn:qkzb}) written in components is
  $\Psi_{r\cdot\pi}=\sigma \Psi_\pi$.
  
  Using $\Theta_{i+1}=\sigma \Theta_i \sigma^{-1}$,
  we find that if $\Psi_{\pi}=\Theta_s \Psi_{\pi_0}$, 
  $\Psi_{r\cdot\pi}
  =\sigma \Theta_s \sigma^{-1}\Psi_{r\cdot\pi_0}$. All that needs to
  be checked is that $r\cdot\pi_0=\pi_0$, which is obvious from its
  definition, and that $\Psi_{\pi_0}=\sigma \Psi_{\pi_0}$, which
  follows from the explicitly rotation-invariant expression
  (\ref{eqn:basecase}).
  
  To check Eq.~(\ref{eqn:qkza}), we consider separately the two cases
  of Eqs.~(\ref{eqn:qkzd}) and (\ref{eqn:qkze}).
  Eq.~(\ref{eqn:qkzd}), and more generally Eq.~(\ref{eqn:groupoid}),
  are tautologies with our definition (\ref{eqn:defpsi}) of
  $\Psi_\pi$: they express the fact that $\Theta$ is a groupoid morphism.
  
  Eq.~(\ref{eqn:qkze}) requires a calculation. By the usual rotational
  invariance argument ($r e_i r^{-1}=e_{i+1}$ and
  $\sigma\Delta_i\sigma^{-1}=\Delta_{i+1}$), one can assume that
  $i=1$.  The $f_j$, $j=3,\ldots,N-1$ generate a subgroup 
  $\mathcal{S}_{N-2}$ which acts transitively by conjugation on the link
  patterns such that $\pi(1)=2$ (i.e., involutions without fixed
  points of $N-2$ elements). By the same argument as in the proof of
  Prop.~\ref{prop:groupoid}, for any pair of such involutions
  $\mathcal{\hat S}_{\pi,\pi'}\cap \mathcal{S}_{N-2}\ne\emptyset$, and we
  thus have $s\in\mathcal{S}_{N-2}$ such that $\Psi_{\pi'}=\Theta_s
  \Psi_\pi$. Act on both sides of Eq.~(\ref{eqn:qkze}) with
  $\Theta_s$.  Using $\Delta_1\Theta_j=\Theta_j\Delta_1$ and
  $e_1^{-1}(f_j(\pi))=f_j(e_1^{-1}(\pi))$ for $j\ne 1,2,N$ , we deduce
  that equations~(\ref{eqn:qkze}) for any $\pi$ and $\pi'$ are
  equivalent.  One can then check it for a special case, the simplest
  being the following: choose $\pi$ with cycles $(1,2)$ and
  $(j,j+n-1)$, $j=3,\ldots,n+1$. Noting that link patterns in
  $e_1^{-1}(\pi)$ come in pairs $\pi_j$, $f_i\cdot\pi_j$,
  $j=3,\ldots,n+1$, where $\pi_j(1)=j$ and $\pi_j(2)=j+n-1$, we conclude that 
  $\sum_{\pi'\ne\pi, e_1\cdot\pi'=\pi}
  \Psi_{\pi'}=(1+\Theta_1)\sum_{j=3}^{n+1} \Psi_{\pi_j}$. 
  Using the expression (\ref{eqn:alttheta}) for $\Theta_1$, we find
  that Eq.~(\ref{eqn:qkze}) is equivalent to the fact that
  $\Phi:=\Psi_\pi-{2\A-\epsilon\over \A+z_1-z_2}\sum_{j=3}^{n+1} \Psi_{\pi_j}$ 
  is symmetric in $z_1$, $z_2$.  It is left to the
  reader to check that the following expression holds:
  \begin{equation}
    \Phi=
    \prod_{\scriptstyle 3\le i<j \le 2n\atop\scriptstyle j-i<n-1}(\A+z_i-z_j)
    \prod_{\scriptstyle 3\le i<j \le 2n\atop\scriptstyle j-i>n-1}(\B+z_j-z_i)
    \prod_{j=3}^{n+1} (\A+z_1-z_j)(\A+z_2-z_j)
    \prod_{i=n+2}^{2n} (\B+z_i-z_1)(\B+z_i-z_2)
  \end{equation}
  (we recall that $\B=\A-\epsilon$)
  so that in particular it is symmetric in $z_1$, $z_2$.
\end{proof}

For later reference, we restate equations (\ref{eqn:qkzd}, \ref{eqn:qkze}) with
the specialization $\beta=\frac{2(\A-\epsilon)}{2\A-\epsilon}$
and $\B = \A-\epsilon$, making $1-\beta/2 = \A/(\A+\B)$. 
We will also clear denominators, some of
which can be subsumed into the divided difference operator $\der_i$.
In the same dichotomy,
\begin{itemize}
\item if $\pi(i)\ne i+1$:
  \[
    \Psi_{f_i\cdot \pi} 
    = \Theta_i \Psi_\pi\qquad 
    \Theta_i := \frac{(\A+z_i-z_{i+1})(\A+(1-\beta/2)(z_{i+1}-z_{i}))\tau_i
      -\A(\A-z_i+z_{i+1})}{(1-\beta/2)(z_i-z_{i+1})(\A-z_i+z_{i+1})}
\]
which one can rewrite equivalently:
\begin{equation}\label{eqn:qkzdv2}
  (\A+\B + z_{i+1} - z_i)(\A+z_i-z_{i+1}) (-\partial_i) \frac{\Psi_\pi}{\A+z_i-z_{i+1}} = 
  \Psi_\pi + \Psi_{f_i\cdot \pi}. 
\end{equation}

\item if $\pi(i)=i+1$:
  \begin{equation}\label{eqn:qkzev2}
    (\A+\B+z_{i+1}-z_i)(\A+z_i-z_{i+1}) (-\der_i) \Psi_\pi
    = (\A+\B) \sum_{\pi'\ne\pi,e_i\cdot\pi'=\pi} \Psi_{\pi'}
  \end{equation}
\end{itemize}

The summands in the last sum come in pairs $\{\pi',f_i\cdot\pi'\}$.
There is a natural way to pick representatives from these pairs
\[ \varepsilon(\pi,i) := \left\{ \rho \neq \pi\ :\ e_i\cdot \rho = \pi,
\text{ and the chords emanating from $i,i+1$ cross each other} \right\}
 \]
which will be motivated by geometry in \S \ref{sec:ei}.
We can then use (\ref{eqn:qkzdv2}) to rewrite that sum:
\begin{align*}
  \sum_{\pi'\ne\pi,e_i\cdot\pi'=\pi} \Psi_{\pi'}
  &= \sum_{\pi' \in \varepsilon(\pi,i)} \left( \Psi_{\pi'} + \Psi_{f_i\cdot \pi'}\right) \\
  &= \sum_{\pi' \in \varepsilon(\pi,i)} 
(\A+\B + z_{i+1} - z_i)(\A+z_i-z_{i+1}) (-\partial_i) \frac{\Psi_\pi}{\A+z_i-z_{i+1}}
\end{align*}
So in the presence of (\ref{eqn:qkzdv2}), Equation (\ref{eqn:qkzev2})
is equivalent to (for $\pi(i)=i+1$)
  \begin{equation*}
    (-\der_i) \Psi_\pi =  (\A+\B) \sum_{\rho \in \varepsilon(\pi,i)} 
    (-\der_i) \frac{\Psi_\rho}{\A+z_i-z_{i+1}}
  \end{equation*}
We clear out the denominator by multiplying both sides by 
$(\A+z_i-z_{i+1})(\A+z_{i+1}-z_i)$, obtaining
\begin{itemize}
\item if $\pi(i)=i+1$:
  \begin{equation}
    \label{eqn:qkzev3}
    (-\der_i) (\A+z_i-z_{i+1})(\A+z_{i+1}-z_i) \Psi_\pi 
    =  (\A+\B) \sum_{\rho \in \varepsilon(\pi,i)} 
    (-\der_i) (\A+z_{i+1}-z_i) \Psi_\rho
  \end{equation}
\end{itemize}
or
\[
     (\A+z_i-z_{i+1})(\A+z_{i+1}-z_i) \Psi_\pi 
    =  (\A+\B) (\A+z_{i+1}-z_i) \sum_{\rho \in \varepsilon(\pi,i)} \Psi_\rho
\quad+\quad \text{something symmetric in $z_i,z_{i+1}$}.
\]
We will learn more about this last term (though not completely
calculate it) in \S \ref{ssec:lowerbound}.

\section{Geometric interpretation of the Brauer action}
\label{sec:geombrauer}

In \cite{KZJ} we gave a geometric derivation of the action of the
$f_i$ generators of the (affine) Brauer algebra
$\mathcal{\hat B}_N(\beta)$ on 
the Brauer loop polynomials $\{\Psi_\pi\}$ at $\A=\B$. 
Using
algebraic results of \cite{DFZJ06}, we showed that this implied an
action of the $e_i$ generators, but did not give a geometric interpretation
thereof; this will be one of the main results of this paper.

In this section we review first Hotta's construction \cite{Ho} of the action of
the $e_i$ generators on the vector space spanned by Joseph polynomials. 
We also discuss the analogue of the $f_i$ action on 
Joseph--Melnikov polynomials.

We then move beyond the orbital scheme to the Brauer loop scheme. We
review the action of the $f_i$ generators from \cite{KZJ}. 
This leads us to
identify the solution of $q$KZ equation discussed in the previous
section with the multidegrees of the irreducible components of the
Brauer loop scheme.
We then give a
direct geometric interpretation of the action of the $e_i$ generators on
the Brauer loop polynomials. 

In the diagrams of spaces in this section, we write
\[ X \sim Y \qquad \text{if 
$\dim X = \dim Y > \dim (X\setminus Y \cup Y \setminus X)$}
\]
which implies that their multidegrees are equal. (In each case we hope
and expect that actually $X=Y$, but don't prove or need this.)

\subsection{Recall: Hotta's construction of Springer representations}
\label{ssec:hotta}
Let $D$ be the closure of a nilpotent orbit of $\GLN$ (though Hotta's
construction works for other groups as well), and $D\cap R_N$ the
corresponding orbital scheme, with components $\{D_\tau\}$.
Note that each $D_\sigma \subseteq R_N^{\Delta=0}$, 
i.e., the diagonal of a nilpotent upper triangular matrix vanishes.
Fix a particular orbital variety, $D_\sigma$, which we recall to be
automatically $B_N$-invariant.

Fix $i \in \{1,2,\ldots,N-1\}$. There is a corresponding $GL_2$ subgroup
of $\GLN$, consisting of matrices $M$ that look like the identity matrix
except in entries $M_{ab}$ with $a,b\in \{i,i+1\}$. Call this subgroup
$GL_2^{(i)}$, and let $B^{(i)} := GL_2^{(i)} \cap B_N$.

\[
 GL_2^{(i)} := \left\{
\begin{pmatrix}
  1& &&&&&&\\
  &\ddots&&&&&&\\
  &&1&&&&&\\
  &&&x&y&&&\\
  &&&z&w&&&\\
  &&& & &1&&\\
  &&& & &&\ddots&\\
  &&& & &&&1
\end{pmatrix} \right\}
\qquad
B^{(i)} := \left\{
\begin{pmatrix}
  1& &&&&&&\\
  &\ddots&&&&&&\\
  &&1&&&&&\\
  &&&x&y&&&\\
  &&&0&w&&&\\
  &&& & &1&&\\
  &&& & &&\ddots&\\
  &&& & &&&1
\end{pmatrix}
  \right\}
\]

\subsubsection{Cutting then sweeping}
There are two cases. If every $M \in D_\sigma$ has $M_{i,i+1}=0$,
then $D_\sigma$ is $GL_2^{(i)}$-invariant. This is the boring case.

Otherwise $D_\sigma \cap \{M_{i,i+1}=0\}$ is codimension $1$ in
$D_\sigma$, and is itself $B_N$-invariant.  Let its geometric components
be $\{D'_{\tau}\}$ (the precise indexation being unspecified yet)
appearing with multiplicities $\{ m_{\sigma\tau} \in \naturals\}$.
Each such $D'_\tau$ is $B_N$-invariant hence $B^{(i)}$-invariant, so 
\begin{equation*} 
  \dim \left( GL_2^{(i)}\cdot D'_\tau \right)
  \leq \dim D'_\tau + \dim \left( GL_2^{(i)}/B^{(i)}\right)
  = (\dim (D\cap R_N) - 1) + 1 = \dim (D\cap R_N).
\end{equation*}
Plainly $GL_2^{(i)}\cdot D'_\tau \subseteq D$, and it is also easy
to see that $GL_2^{(i)}\cdot D'_\tau \subseteq R_N$. So if the above
dimension inequality is tight, $GL_2^{(i)}\cdot D'_\tau$ must again be
an orbital variety, say $D_\tau$ (which tells us how to index the $D'_\tau$
-- by the $D_\tau$ they sweep out to).

\begin{center}
\begin{tikzcd}
  D_\sigma \arrow{dr}{\text{cut}} & & \bigcup_\tau GL_2^{(i)}\cdot D'_\tau & \sim & 
     \bigcup_\tau D_\tau \text{ with multiplicity $m_{\sigma\tau}$} \\
   & \bigcup_\tau D'_\tau \arrow{ur}{\text{sweep}} \\
 \end{tikzcd}
\end{center}

Having pursued the geometry, we now give the corresponding multidegree
calculation. We introduce to that effect to the $GL_2^{(i)}$-invariant
hyperplane
\begin{equation*} 
  R_N^- := R_N^{\Delta=0} \cap \{M_{i,i+1} = 0\} 
\end{equation*}
and use the equality (property 3(a) from \S\ref{ssec:multidegrees})
$\mdeg_{R_N^{\Delta=0}} D_\sigma=\mdeg_{R_N^-}(D_\sigma\cap\{ M_{i,i+1}=0 \})$.
Cutting with $\{ M_{i,i+1}=0 \}$ results in the decomposition
\begin{equation}\label{eqn:hottafirst}
  \mdeg_{R_N^{\Delta=0}} D_\sigma=\mdeg_{R_N^-}(D_\sigma\cap\{ M_{i,i+1}=0 \})
  = \sum_\tau m_{\sigma\tau}\ \mdeg_{R_N^-} D'_\tau.
\end{equation}

To understand the effect of sweeping out a $B^{(i)}$-invariant variety 
using $GL_2^{(i)}$, we need the following special case of a result from
\cite{Jo,BBM}, spelled out in the present language in \cite[Lemma 1]{KZJ}.

\begin{Lemma}\label{lem:divdiff}
  Let $V \leq \MNC$ be a subspace invariant under $B_N$ and $GL_2^{(i)}$.
  Let $X$ be a variety in $V$ invariant under $B_N$ and rescaling, 
  with multidegree $\mdeg_V X$. If the generic fiber of the map 
  \begin{equation*} 
    \mu: (GL_2^{(i)} \times X)/B^{(i)} \to V, \qquad [g,x] \mapsto g\cdot x 
  \end{equation*}
  is finite over ${\rm Image}\ \mu$, 
  call its cardinality $k$; otherwise let $k=0$. 
  (The latter occurs iff $X$ is $GL_2^{(i)}$-invariant.) Then
  \begin{equation*} 
    k\ \mdeg_V ({\rm Image}\ \mu) = -\der_i \ \mdeg_V X 
  \end{equation*}
  where $\der_i$ is the divided difference operator,
  $\der_i p = (p - r_i\cdot p)/(z_i - z_{i+1})$, defined using 
  the reflection $r_i$ that switches $z_i\leftrightarrow z_{i+1}$.  
\end{Lemma}

Applying the lemma amounts to
applying $-\partial_i$ to Equation~(\ref{eqn:hottafirst}).
The right hand side 
$\sum_\tau m_{\sigma\tau}\ \mdeg_{R_N^-} D'_{\tau} $ becomes
\begin{equation*}
  \sum_\tau m_{\sigma\tau} (-\partial_i) \mdeg_{R_N^-} D'_{\tau} 
  =  \sum_\tau m_{\sigma\tau}\ \mdeg_{R_N^-} D_\tau
  ={1\over\A+z_i-z_j}\sum_\tau m_{\sigma\tau}\ \mdeg_{R_N^{\Delta=0}} D_\tau
\end{equation*}
where the latter sums are over only those $\tau$ such that $D'_{\tau}$
is not $GL_2^{(i)}$-invariant (and where the last equality comes from axiom
3(b) of \S\ref{ssec:multidegrees}).
We finally obtain for\footnote{%
  We could equally well have defined $J_\sigma$ as $\mdeg_{R_N} D_\sigma$,
  incurring a factor of $A^N$. Joseph could not have, since he
  implicitly works at $A=0$, in his neglect of the dilation action.}
$J_\sigma=\mdeg_{R_N^{\Delta=0}} D_\sigma$:
\begin{equation}\label{eqn:hotta}
  -(\A+z_i-z_{i+1})\der_i J_\sigma=\sum_\tau m_{\sigma\tau} J_\tau
\end{equation}
This equation is only valid if $D_\sigma\not\subset \{M_{i,i+1}=0\}$;
however it is easy to see that in the case $D_\sigma\subset \{M_{i,i+1}=0\}$, 
it is still satisfied if one conventionally\footnote{%
  This is slightly strange in that $m_{\sigma\tau}$ is otherwise
  nonnegative, but this negativity is essentially unavoidable 
  if we want to construct matrices that square to $1$. That only happens
  for integer matrices of constant sign in the uninteresting case of 
  permutation matrices (times $\pm 1$) -- but those would not
  generate an irreducible $S_n$ representation.}
sets $m_{\sigma\sigma}=-2$, $m_{\sigma\tau}=0$ ($\sigma\ne\tau$).

Usually, Eq.~(\ref{eqn:hotta}) is rewritten as
\begin{equation*}
\left(-r_i+\A \der_i\right) J_\sigma=-J_\sigma-\sum_\tau m_{\sigma\tau} J_\tau
\end{equation*}
which shows that the $\integers$-span of the $\{J_\tau\}$
is closed under the action of the operators $\{-r_i + \A\partial_i\}$.
These operators are easily seen to satisfy the $\SN$ Coxeter relations,
but what is more, this representation is irreducible, and
each irreducible representation of $\SN$ arises from a unique nilpotent orbit.

Above, we had split the description into two cases according to whether
$D'_\sigma$ was $GL_2^{(i)}$-invariant or not, but this was only in an
attempt to aid understanding rather than mathematically necessary;
the equivariant cohomology calculation underlying Lemma \ref{lem:divdiff} 
does not actually require that one distinguish the two cases. That
calculation is based on the pushforward of the fundamental class along
the map $\mu$, and this pushforward vanishes when the generic fiber is
positive-dimensional.  Correspondingly, in that case $\mdeg_{\MNC} X$
is symmetric in $\{z_i,z_{i+1}\}$, thus annihilated by $\partial_i$.

\subsubsection{Example: The case \texorpdfstring{$D=\{M^2=0\}$}{D=\{M^2=0\}}}\label{ssec:hottasquare}

We include the results in this subsection only to illustrate
the formula above, and do not pause to give details of this calculation.

For simplicity let $N=2n$.
Then $\sigma$ is encoded by a link pattern on the interval with no crossings,
and $D_\sigma$ is $GL_2^{(i)}$-invariant iff $\sigma$ has no arch
connecting $i \leftrightarrow i+1$. 

Assume now that $\sigma$ has such an arch.
We have already considered the geometry of the hyperplane section
in Theorem \ref{thm:JMformula}; $D_\sigma \cap \{M_{i,i+1}=0\}$ has
$\overline{B_N\cdot \tau'_<}$ as a component iff $\tau'$ is constructed
from $\sigma$ 
\begin{itemize}
\item by pulling the $i\leftrightarrow i+1$ arch to touch
  some other arch, then creating a crossing there, or
\item (if there are no crossings, and no arch containing
  $i\leftrightarrow i+1$) by breaking the $i\leftrightarrow i+1$ arch
  into two vertical lines.
\end{itemize}
Moreover, these components $D'_{\tau} = \overline{B_N\cdot \tau'_<}$
show up with multiplicity $1$ (they are generically reduced).

The next step is to sweep each such $\overline{B_N\cdot \tau'_<}$ 
using $GL_2^{(i)}$. If $\tau'$ is of the second type listed above, 
then $\overline{B_N\cdot \tau'_<}$ is already $GL_2^{(i)}$-invariant.
Hence the $k$ in Lemma \ref{lem:divdiff} is $0$, and we can ignore
these components. 
For the $\tau'$ of the first type, 
$GL_2^{(i)} \cdot \overline{B_N\cdot \tau'_<} = \overline{B_N\cdot \tau_<}$,
where $\tau$ is constructed from $\tau'$ by replacing the new (unique)
crossing with $)($. In this case $k=1$.

This leads to the following characterization of the $\tau$ that arise from $\sigma$
in this way. Given an arbitrary chord diagram $\tau$ without crossings,
define $e_i \cdot \tau$ by replacing the two arches
\begin{equation*} i \leftrightarrow \tau(i), \quad i+1 \leftrightarrow \tau(i+1)
\qquad\dashrightarrow\qquad
i \leftrightarrow i+1, \quad \tau(i) \leftrightarrow \tau(i+1). \end{equation*}
Then the $\tau$ constructed above are those such that 
$\tau \neq \sigma$, $e_i\cdot \tau = \sigma$. 
We obtain finally the following recurrence for extended Joseph polynomials:
\begin{equation}\label{eqn:hottafinal}
\qquad\qquad\qquad
-(\A+z_i-z_{i+1})\der_i J_\sigma=\sum_{\tau\ne\sigma,\ e_i\cdot\tau=\sigma}
J_{\tau} 
\qquad\qquad\qquad
\textrm{if}\ \sigma(i)=i+1
\end{equation}
The Hotta construction only applies to orbital varieties, that is to
$\sigma$ noncrossing. However Eq.~(\ref{eqn:hottafinal}) still makes
sense for more general $\sigma$ (i.e., for smaller $B_N$-orbits inside $D$);
clearly the procedure outlined above works equally well if $\sigma$ has
crossings, as long as the arch $(i,i+1)$ is replaced in the preimages
$\tau$ by a pair of noncrossing arches. This condition on $\tau$, as
well as Eq.~(\ref{eqn:hottafinal}) itself, will naturally come out of
the more general Brauer construction.

In the course of the derivation of Eq.~(\ref{eqn:hottafinal}), we have
obtained the following result, related to sweeping only. First note
that in the calculation above the effect of sweeping on our multidegrees 
with respect to the upper triangle $R_N$ is given by the operator
$\tilde\der_i
:=\widehat{(\A+z_{i}-z_{i+1})}\der_i \widehat{1\over \A+z_{i}-z_{i+1}}
=\widehat{1\over \A+z_{i+1}-z_i}\partial_i \widehat{(\A+z_{i+1}-z_i)}$
(where $\widehat f$ denotes the multiplication-by-$f$ operator).
Then if $\sigma$ is any chord diagram 
such that $i$ and $i+1$ are connected to distinct arches that cross each other,
\begin{equation}\label{eqn:fiJeq}
  \qquad\qquad\qquad
  -\tilde\der_i J_\sigma=J_{\bar f_i^{-1}\cdot\sigma}
  \qquad\qquad\qquad
  \text{if $(i,\sigma(i))$ and $(i+1,\sigma(i+1))$ cross}
\end{equation}
where $\bar f_i^{-1}\cdot\sigma$ (the notation will be explained in
\S\ref{ssec:degenbrauer}) is by definition the chord diagram
in which the crossing of these two arches
is replaced with a )(.

In the permutation sector, Eq.~(\ref{eqn:fiJeq}) is nothing but the
usual recursion relation satisfied by (double) Schubert polynomials
(the extra conjugation of the divided difference operator coming from
the factors in Prop.~\ref{prop:doubschub}).

\subsubsection{An alternate construction:
  sweeping then cutting}\label{ssec:hottasweepcut}
In the Hotta construction we started with an orbital variety, 
intersected it with the hyperplane $\{M_{i,i+1}=0\}$, 
then swept out each component using $GL_2^{(i)}$ to get another orbital variety.

There is an alternate geometric construction, 
in which we start with an orbital variety, and sweep it out using $GL_2^{(i)}$. 
This is no longer upper triangular (unless $D_\sigma$ is $GL_2^{(i)}$-invariant),
but almost; the only lower-triangle entry that may appear\footnote{%
  This construction seems slightly simpler (or, requiring less art) in
  the following sense: rather than guessing in advance the condition
  $\{M_{i,i+1}=0\}$ that will work well with the sweeping operation,
  we just sweep and look through $E$'s equations for which conditions to 
  re-impose. ``It's easier to ask forgiveness than it is to get permission"
  -- Adm. Grace Hopper} 
is $M_{i+1, i}$. So we intersect with the hyperplane $\{M_{i+1\ i}=0\}$ to get
a schemy union of what turn out to be orbital varieties.

It is easiest to compare the two approaches by looking at the multidegrees.
In the cut-then-sweep approach, the geometry computed
\begin{equation*} 
  -\partial_i\ (\A + z_i - z_{i+1}) \mdeg_\MNC D_\sigma 
\end{equation*}
whereas in this sweep-then-cut approach, the geometry computes
\begin{equation*} 
  -(\A + z_{i+1} - z_i)\ \partial_i\ \mdeg_\MNC D_\sigma. 
\end{equation*}
(noting that we use here multidegrees with respect to the full space $\MNC$; 
compare also with Eq.~(\ref{eqn:hotta})).

The operators are very simply related, differing only by $2$,
\begin{equation*} -\partial_i \widehat{(\A + z_i - z_{i+1})} 
=- \widehat{(\A + z_{i+1} - z_i)} \partial_i\ -2 
 \end{equation*}
and as such we do not learn much from this new construction that was not
already evident in Hotta's cut-then-sweep construction.

The cut-then-sweep construction has reached full flower in the
construction of convolution algebras in geometric representation theory;
see e.g. \cite{CG,N}. In that more general setting we do not know
an analogue of the sweep-then-cut construction.

\subsection{The actions on the Brauer loop scheme}
\label{ssec:braueraction}

We now explore these constructions in the context of the Brauer loop
scheme.  The main difference is that $GL_2^{(i)}\cdot E$ violates {\em two} 
of $E$'s equations, one linear, one quadratic.

Most of the calculations in this section will be in the context of
infinite periodic upper triangular matrices, $R_{\integers \bmod N}$.
In this context we will use $GL_2^{(i)}$ to denote 
\begin{multline*} GL_2^{(i)} := \Big\{\wtM \in \text{Mat}_{\integers \bmod N}
    \ :\ \det \left[{M_{i,i}\ M_{i,i+1}\atop M_{i+1,i}\ M_{i+1,i+1}}\right]\neq 0,\\
    M_{jk} \neq \delta_{jk} \implies\exists a\in\integers\text{ s.t. }
j,k\in \{i+aN,i+1+aN\} \Big\}
\end{multline*}
and $B^{(i)}$ its intersection with $R_{\integers\bmod N}$.
Note that if $s \in GL_2^{(i)}$, $\wtM \in \text{Mat}_{\integers \bmod N}
$,
then
\begin{equation*} (s \wtM s^{-1})_{jk} = \wtM_{jk} 
  \qquad \text{ unless $j$ or $k$ is in $\{i,i+1\} \bmod N$}. 
\end{equation*}

Since $i$ will be fixed in this section, we could omit it from
the notation, and we will do so when considering the following notation
(used only in this section):
for $L\in \text{Mat}_{\integers\bmod N}$, let 
\begin{align*}
  L' =& \text{
    the $2\times 2$ submatrix of $L$ using rows $i,i+1$ and columns $i,i+1$,
  and}\\
  L^\# =& \text{
the $2\times 2$ submatrix of $L$ using rows $i,i+1$ and columns $i+N,i+1+N$}
\end{align*}
e.g. at $N=6$, 
\[
 L =
\begin{pmatrix}
  \ddots &&&&&& \\
  \cdots & [L'] & {b\ c\atop f\ g} & {d\ e\atop h\ i} & [L^\#] & \cdots \\
  \cdots &      & {j\ k\atop q\ r} & {l\ m\atop s\ t} & {n\ p\atop u\ v} & \cdots \\
  \cdots &  &   & {w\ x\atop \gamma\ \delta} & {y\ z\atop \epsilon\ \iota} & {\alpha\ \beta\atop \kappa\ \lambda} \\
  \cdots &&& & [L'] & \cdots \\
  &&&&&& \ddots
\end{pmatrix}
\]

We will need some lemmas about slicing and sweeping in this context.
The first one produces useful representatives for many arguments,
and was already implicitly used in \cite{KZJ}. 

\newcommand\wtrho{{\widetilde\rho}}
\begin{Lemma}\label{lem:genericelts}
  Let $\rho$ be a partial involution of $\{1,\ldots,N\}$,
  i.e. an injective partially defined function 
  such that for any $i$, if $\rho(i)$ and then $\rho(\rho(i))$ are
  defined, $\rho(\rho(i))=i$. Pick $s_1,\ldots,s_N$ generic, and
  define $\wtrho \in \text{Mat}_{\integers\bmod N}$ by
  \[ \wtrho_{jk} = 
  \begin{cases}
    s_{j'} &\text{if $\pi(j)= k\bmod N$, $k>j>k-N$, $j'= j\bmod N$} \\
    0 &\text{otherwise.}
  \end{cases}
  \]
  Construct an involution $\rho'$ from $\rho$ by ``promoting each
  leftward move to a $2$-cycle, and making all others fixed points.''
  In formulae,
  \[ \rho'(j) = 
  \begin{cases}
    \rho(j) &\text{if $\rho(j)$ is defined and $\rho(j)<j$} \\
    j' &\text{if $\rho(j')$ is defined, $\rho(j')=j$, and $j'>j$} \\
    j &\text{otherwise}      
  \end{cases}
  \]
  Then
  \begin{itemize}
  \item $\wtrho^2_{i,i+N} \neq 0$ iff $\rho(i),\rho(\rho(i))$ are both defined.
    If $\wtrho^2_{i,i+N} = \wtrho^2_{j,j+N} \neq 0$, then $i=j$ 
    or $i = \rho(j)$.
  \item Let $M$ be the $N\times N$ submatrix of $\wtrho$ using rows 
    and columns $\{1,\ldots,N\}$. Then $M^2=0$ and $M$ is upper triangular,
    so by Theorem \ref{thm:involutions}, $M$ is $B_N$-conjugate to the
    strict upper triangle of the permutation matrix of
    a unique involution; this involution is $\rho'$.
  \end{itemize}
\end{Lemma}

\begin{proof}
  First calculate
  \[ (\wtrho^2)_{j,j+N} = 
  \begin{cases}
    s_j s_{\rho(j)} &
    \text{if $\rho(j)$ and then $\rho(\rho(j))$ are both defined} \\
    0 &\text{otherwise}
  \end{cases}
  \]
  This, and the genericity of the $\{s_i\}$, imply the first two statements.
  
  That $M$ is strictly upper triangular is tautological.
  That $M^2$ is zero follows from the conditions on $\rho$. 
  Then the matrix $M$ and the strict upper triangle of $\rho'$ have
  entries in the same places, so are even $T$-conjugate, thus $B_N$-conjugate.
\end{proof}

\begin{Lemma}\label{lem:extraeqns}
  Let $X\subseteq E$, thought of inside $\text{Mat}_{\integers\bmod N}$.
  So for $L\in X$, the $2\times 2$ matrices $L',L^\#$ 
  are strictly upper triangular and upper triangular, respectively.

  Then the scheme
  \[ \big( GL_2^{(i)}\cdot X\big) 
  \cap \left\{L : L' \text{ is strictly upper triangular}, 
  (L^2)^\# \text{ is upper triangular} \right\} \]
  is contained inside $E$ scheme-theoretically, and the set
  \[ \big( GL_2^{(i)}\cdot X\big) 
  \cap \left\{L : L' \text{ is upper triangular}, 
  (L^2)^\# \text{ is upper triangular} \right\} \]
  is contained inside $E$ set-theoretically. 

  If in addition $X \subseteq \{L : L' = \left[{0\atop 0}{0\atop 0}\right] \}$,
  then the scheme
  \[ \big( GL_2^{(i)}\cdot X\big) 
  \cap \left\{L : (L^2)^\# \text{ is upper triangular} \right\} \]
  is contained inside $E$ scheme-theoretically.
\end{Lemma}

\begin{proof}
It is easy to see that for $s\in GL_2^{(i)}$, 
we have $(s\cdot L)' = s' \cdot L'$, $(s\cdot L)^\# = s'\cdot L^\#$.

The matrices $s\cdot \wtM \in GL_2^{(i)}\cdot X$ 
come very close to satisfying $E$'s defining equations:
\begin{align*}
  (s \wtM s^{-1})_{jk} 
  &= \wtM_{jk} &&\text{unless $j,k \in \{i,i+1\}\ (\bmod N)$} \\
  &= 0 &&\text{for $j\geq k$, since $\wtM\in E$}\\
  (s \wtM s^{-1})^2_{\ jk} 
  &= (s \wtM^2 s^{-1})_{jk} \\
  &= \wtM^2_{jk} &&\text{unless $j,k \in \{i,i+1\}\ (\bmod N)$} \\
  &= 0 &&\text{for $j+N>k$, since $\wtM\in E$}
\end{align*}
The only ones that are satisfied on $E$ and not necessarily on
$GL_2^{(i)}\cdot X$ can be rewritten, for $L \in GL_2^{(i)}\cdot X$, 
as the $3+1$ equations
\[ L' \text{ is strictly upper triangular}, 
  (L^2)^\# \text{ is upper triangular}. \]
The first three equations are implied {\em set-theoretically} 
by the single equation $L_{i+1,i} = 0$, since $L'$ is nilpotent
(being conjugate to $\wtM'$, which is strictly upper triangular
for $\wtM\in E$).

So for any subscheme $X\subseteq E$, to intersect $GL_2^{(i)} \cdot X$ with $E$ 
it suffices set-theoretically to intersect with the hypersurfaces defined by
\[ \wtM_{i+1,i} = 0,\qquad (\wtM^2)_{\ i+1,i+N} = 0,\] 
whose multidegrees
are $\A + z_{i+1} - z_i$, $2\A + z_{i+1} - z_{i+N}=\A+\B+z_{i+1}-z_i$
respectively. 

Finally, if 
$X\subseteq \left\{\wtM: \wtM' = \left[{0\atop 0}{0\atop 0}\right] \right\}$,
then $(s\cdot \wtM)' = s'\cdot \wtM' = \left[{0\atop 0}{0\atop 0}\right]$
so the first three equations are automatic.
\end{proof}

The calculation above was significantly more complicated in \cite{KZJ},
where we did not make proper use of $R_{\integers \bmod N}$.

\subsubsection{Sweeping then cutting: the \texorpdfstring{$f_i$}{f_i} action}
\label{ssec:sweepcutF}

We now recall the results of \cite[\S 4.2]{KZJ}, 
adapted to the present discussion. We have been able to streamline
Proposition 6 (the technical heart) of that paper enough that 
we include here a complete proof, repeating some parts of 
the one from \cite{KZJ}.

Begin with a chord diagram $\pi$ with no ``little arch'' connecting
$i$ to $i+1$, and the corresponding Brauer loop variety 
$E_\pi \subseteq \MMN$. Sweep $E_\pi$ out using the $GL_2^{(i)}$ action.
Lifting its elements to $R_{\integers \bmod N}$, an element
lying over the sweep $GL_2^{(i)} \cdot E_\pi$ looks like 
\begin{equation*} 
  s \wtM s^{-1}, \qquad\text{where} 
  \quad s\in GL_2^{(i)},\quad \wtM^2 \in \< S^N \>, 
  \quad (M^2)_{i,i+N} = (M^2)_{\pi(i),\pi(i)+N}. 
\end{equation*}

The diagram of the spaces encountered in this construction:
\begin{center}
\begin{tikzcd}
\ &GL_2^{(i)}\cdot E_\pi \arrow{rd}{cut} & GL_2^{(i)} \times^{B^{(i)}} E_\pi \arrow[twoheadrightarrow]{l}\\
E_\pi\arrow{ur}{sweep}\arrow[hookrightarrow]{rr} & & 
(GL_2^{(i)}\cdot E_\pi \ \cap\ \{M : M^2_{i+1,i+N} = 0\}) \rlap{$=: F
  \sim E_\pi \cup E_{f_i\cdot \pi}$} \\
\end{tikzcd}
\end{center}

{\em Calculating the degree.}
We first show the map
$GL_2^{(i)} \times^{B^{(i)}} E_\pi \to GL_2^{(i)} \cdot E_\pi$ has degree $1$. \break
We must select an $\wtM \in GL_2^{(i)}\cdot E_\pi$ and compute the fiber 
$\left\{ [s, \wtM'] : s\cdot \wtM' = \wtM \right\}$ lying over it.
By Theorem \ref{thm:Ecomps}, for general elements $\wtM$ 
lying over $E_\pi$ we know $\wtM^2_{i,i+N},\wtM^2_{i+1,i+1+N}$ are different. 
Fix such an $\wtM = 1\cdot \wtM \in GL_2^{(i)}\cdot E_\pi$. 
Then $(\wtM^2)^\#$ is upper triangular with distinct diagonal entries,
and 
\begin{equation*} 
  \wtM^2 = s\cdot \wtM'^2 \quad\implies\quad
  (\wtM^2)^\# = \left(s\cdot \wtM'^2\right)^\# = s' \cdot (\wtM'^2)^\# 
\end{equation*}
where $(\wtM'^2)^\#$ is also upper triangular. By considering the
eigenvalues of the $2\times 2$ matrices $(\wtM^2)^\#, (\wtM'^2)^\#$,
we see their diagonals must either agree or be reversed.
Since $\wtM' \in E_\pi$, by Theorem \ref{thm:Ecomps} their diagonals
must agree. Since $s'$ conjugates a $2\times 2$ upper triangular matrix 
with distinct eigenvalues to another with the same diagonal,
$s'$ too must be upper triangular, so $s \in B^{(i)}$. 
This shows that the fiber
$\left\{ [s, \wtM'] : s\cdot \wtM' = \wtM \right\}$ over $\wtM$ is a point,
hence the degree of the map is $1$. 
In particular, $E_\pi$ is not $GL_2^{(i)}$-invariant.

{\em The necessary extra equation.} 
By this non-invariance, $\dim (GL_2^{(i)} \cdot E_\pi) > \dim E_\pi = \dim E$.
Hence the irreducible variety $GL_2^{(i)} \cdot E_\pi$ does not lie in $E$,
and by the last conclusion in Lemma \ref{lem:extraeqns}
the only equation we need re-impose is $(\wtM^2)_{i+1,i+N} = 0$.
Let $F$ denote the intersection of $GL_2^{(i)} \cdot E_\pi$ and
$\left\{ \wtM : (\wtM^2)_{i+1,i+N} = 0 \right\}$. 

Since $F$ is a
Cartier divisor in the irreducible variety $GL_2^{(i)} \cdot E_\pi$, 
its geometric components are all of codimension $1$ therein,
hence of the same dimension as $E$. So as a set, $F$ is a union of
some of the top-dimensional
components of $E$. Since $F \subseteq E$ (by Lemma \ref{lem:extraeqns}),
and $E$ is generically reduced along its top-dimensional components, 
$F$ is too.

Hence the scheme $F$ is a union of some $\{E_\rho\}$,
up to embedded components, which do not affect multidegree calculations.

{\em The geometric components of $F$.}
It is easy to see that $E_\pi$ is a component:
\begin{equation*} 
  E_\pi = 1 \cdot E_\pi \subseteq \left( GL_2^{(i)} \cdot E_\pi \right)
  \cap \left\{(\wtM^2)_{i+1,i+N} = 0\right\} 
  =: F. 
\end{equation*}
To determine which other $E_{\pi'}$ are in $F$, 
we use Theorem \ref{thm:Ecomps}, 
which characterizes the components using the functions $\{ \wtM^2_{j,j+N}\}$.
\begin{equation*} 
  (s \wtM s^{-1})^2_{j, j+N} = (s \wtM^2 s^{-1})_{j, j+N} \qquad
  = \wtM^2_{j, j+N} \text{ unless $j \in \{i,i+1\}\bmod N$} 
\end{equation*}
hence 
\begin{equation*}
  (s \wtM s^{-1})^2_{j, j+N} = (s \wtM s^{-1})^2_{\pi(j), \pi(j)+N} 
  \qquad \text{ unless $j$ or $\pi(j) \in \{i,i+1\}\bmod N$}, 
\end{equation*}
so by Theorem \ref{thm:Ecomps} (and assuming now, 
without loss of generality, that $1\leq i,j\leq N$)
\begin{equation*} 
  E_{\pi'} \subseteq F \qquad\implies\qquad
  \pi'(j) = \pi(j) \quad \text{ for } \quad j\notin \{i,i+1,\pi(i),\pi(i+1)\}. 
\end{equation*}
(This calculation, too, was more complicated in \cite{KZJ}
for not using $R_{\integers \bmod N}$.)

There are three ways to link up $\{i,i+1,\pi(i),\pi(i+1)\}$,
namely $\pi$, $f_i\cdot \pi$, and $e_i\cdot \pi$, 
but $E_{e_i\cdot \pi}$ (connecting $i\leftrightarrow i+1$)
does not satisfy $(s\wtM s^{-1})_{i, i+1} = 0$. Hence
\begin{equation*} E_{\pi'} \subseteq F
\qquad\implies\qquad
\pi' = \pi \ \text{ or }\ \pi' = f_i\cdot \pi. \end{equation*}
This leaves two possibilities for the {\em set\/} $F$:
\begin{equation*} 
  F \quad \text{ is either }\quad E_\pi \quad 
  \text{ or } \quad E_\pi \cup E_{f_i\cdot \pi}. 
\end{equation*}
To rule out $F = E_\pi$, 
we exhibit an element of $F \setminus E_\pi$.
Let $\wtM \in \text{Mat}_{\integers\bmod N}$ be the element
constructed in Lemma \ref{lem:genericelts}.
In particular $\wtM \in E_\pi$, and $\wtM \notin E_\rho$ for any $\rho\neq\pi$.
Let $s \in GL_2^{(i)}$ be the element with 
$s' = \left[ {0\atop 1}{-1\atop 0} \right]$. 
Then $(s\cdot\wtM)^2 = s\cdot\wtM^2$ is supported on the same superdiagonal, 
but with the $(i,i+N),(i+1,i+N+1)$ elements exchanged. 
So $s\cdot\wtM \in E_{f_i\cdot \pi}$, and not in $E_\pi$, as intended.

(It does not follow from this that $F = E_\pi \cup E_{f_i\cdot \pi}$
as a scheme, {\em even if} one also believes that $E$ is reduced: 
$F$ could conceivably still have embedded components along the
intersections of $E_\pi,E_{f_i\cdot\pi}$ with other components of $E$.)

{\em Multidegrees.}
That completes our geometric analysis of $F$, and we move on now 
to the consequences for multidegrees:
\begin{align}
  \mdeg F &= (\A+\B + z_{i+1} - z_i) (-\tilde\partial_i\ \mdeg E_\pi) 
  \qquad \text{by Lemma~\ref{lem:divdiff}}  \label{eqn:mdegF1} \\
  \mdeg F &= \mdeg E_\pi + \mdeg E_{f_i\cdot \pi} \label{eqn:mdegF2}
\end{align}
To use Lemma \ref{lem:divdiff} required computing its factor $k$, 
which we showed at the beginning to be $1$,
and the factor $\A+\B + z_{i+1} - z_i$ is
the weight of the equation $(M^2)_{i+1,i+N} = 0$. 
The extra conjugation 
$\tilde\partial_i
:=\widehat{1\over \A+z_{i+1}-z_i}\partial_i \widehat{(\A+z_{i+1}-z_i)}$
of $\partial_i$ by the multiplication operator, as compared to
Lemma~\ref{lem:divdiff}, comes from the fact that we consider
multidegrees with respect to upper triangular matrices.

\junk{
  First,
  \begin{equation*} 
    \mdeg F \quad \text{ is either }\quad
    \mdeg E_\pi \quad \text{ or } \quad \mdeg E_\pi + \mdeg E_{f_i\cdot \pi} 
  \end{equation*}
  because the difference between $F$ and its reduction is in lower dimension
  which does not affect the multidegree. Second, by Lemma~\ref{lem:divdiff}
  \begin{equation}
    \mdeg F = (\A+\B + z_{i+1} - z_i) (-\tilde\partial_i \mdeg E_\pi)
  \end{equation}
  (the factor $k$ required in Lemma \ref{lem:divdiff} is $1$, as we
  showed at the beginning) where the factor $\A+\B + z_{i+1} - z_i$ is
  the weight of the equation $(M^2)_{i+1,i+N} = 0$, and we recall that
  $\tilde\partial_i
  :=\widehat{1\over \A+z_{i+1}-z_i}\partial_i \widehat{(\A+z_{i+1}-z_i)}$,
  the extra conjugation of $\partial_i$ compared to
  Lemma~\ref{lem:divdiff} coming from the fact that we consider
  multidegrees with respect to upper triangular matrices.
  
  Under the first possibility, $\mdeg F = \mdeg E_\pi$, we get
  \begin{equation*} 
    \mdeg E_\pi = (\A+\B + z_{i+1} - z_i) (- \tilde\partial_i \mdeg E_\pi); 
  \end{equation*}
  applying $\tilde\partial_i$ to both sides we get
  \begin{equation*} 
    \tilde\partial_i \mdeg E_\pi = 2 \tilde\partial_i \mdeg E_\pi 
  \end{equation*}
  but $\tilde\partial_i \mdeg E_\pi = \mdeg (GL_2^{(i)} \cdot E_\pi) \neq 0$,
  since all the $T$-weights live in a half-space and the multidegree
  is a positive sum of products of weights.
  
  Hence the second possibility must hold, and 
}

We conclude that
\begin{equation}\label{eqn:fiaction}
  -(\A+\B + z_{i+1} - z_i) \tilde\partial_i \Psi_\pi = 
  \Psi_\pi + \Psi_{f_i\cdot \pi}. 
\end{equation}
which is Eq.~\eqref{eqn:qkzdv2} with $\Psi_\pi=\mdeg E_\pi$.

(The possibility $F = E_\pi$ we ruled out above can
also be excluded using multidegrees, which was how we did it
in \cite{KZJ}.
Assuming $\mdeg F = \mdeg E_\pi$ leads quickly to the equation
$\tilde\partial_i \mdeg E_\pi = 2 \tilde\partial_i \mdeg E_\pi$.
But $\tilde\partial_i \mdeg E_\pi$ is the multidegree of a 
subscheme of a representation whose $T$-weights all live in a half-space,
so cannot be zero; see lemma \ref{lem:mdegineqs} to come.)

An interesting difference between this construction and the one in
\S\ref{ssec:hottasweepcut} is the use of a {\em quadratic\/} equation 
$(M^2)_{i+1,i+N} = 0$ rather than a linear equation $M_{i+1, i} = 0$.

\subsection{Connection with the \texorpdfstring{$q$}{q}KZ equation}
We can now formulate at last a main result of this paper:
\begin{Theorem}\label{thm:main}
  The following two vector-valued polynomials 
  $\Psi=\sum_\pi \Psi_\pi
  \pi\in V[\A,\epsilon,z_1,\ldots,z_N]$ coincide: 
  \begin{enumerate}
  \item[(i)] the solution of
    Eqs.~(\ref{eqn:qkza}--\ref{eqn:qkzb}) given by Theorem \ref{thm:qkzsol} 
    with its normalization fixed by Eq.~(\ref{eqn:basecase}), and 
  \item[(ii)] the vector of multidegrees
    $\Psi_\pi=\mdeg_{\MMN^{\Delta=0}} E_\pi$ of the irreducible components 
    of the Brauer loop scheme (with the identification $\B:=\A-\epsilon$).
  \end{enumerate}
\end{Theorem}

\begin{proof}
  The multidegrees $\Psi_\pi=\mdeg_{\MMN^{\Delta=0}} E_\pi $ are by
  definition homogeneous polynomials in
  $\integers[\A,\epsilon,z_1,\ldots,z_N]$, of degree the codimension
  of the $E_\pi$, which is shown in Theorem 3 of \cite{KZJ} to be
  $2n(n-1)$. Furthermore, they satisfy Eq.~\eqref{eqn:fiaction},
  which is identical to
  Eq.~\eqref{eqn:qkzdv2}.  Thus, they fulfill all the
  hypotheses of Lemma~\ref{lem:qkzuniq} to ensure uniqueness of the
  solution of Eqs.~(\ref{eqn:qkza}--\ref{eqn:qkzb}) up to
  normalization. The latter is fixed by considering the base case
  $\pi_0(i)=i+n$: as stated in the proof of Prop.~5 of \cite{KZJ},
  $E_{\pi_0}$ is a linear variety given by the equations $M_{ij}=0$,
  $i=1,\ldots,N$, $j=i+1,\ldots,i+n-1$, hence its multidegree matches
  Eq.~(\ref{eqn:basecase}).
\end{proof}

\begin{Corollary}\label{cor:fgivese}
  The vector of multidegrees $\Psi_\pi=\mdeg_{\MMN^{\Delta=0}} E_\pi$
  of the Brauer loop varieties satisfies Equation (\ref{eqn:qkzev3})
  (the $e_i$ action).
\end{Corollary}

In the next subsection we give a geometric interpretation of this latter 
equation. 

\subsection{Geometry of the \texorpdfstring{$e_i$}{ei} action}
\label{sec:ei}

In this section we give a new geometric construction, 
promised after Corollary 4 of \cite{KZJ}, to handle the case $\pi(i)=i+1$.
The stages of the construction are:
\begin{itemize}
\item Cut $E_\pi$ with $\{M : M_{i\ i+1} = 0\}$, producing $F_1$;
\item throw away the $GL_2^{(i)}$-invariant components, giving
  $\bigcup_{\varepsilon(\pi,i)} X_\rho$ (defined below), and
\item show $X_\rho$ matches 
  $E_\rho \cap \left\{M \in E : (M^2)_{i,i+N} = (M^2)_{i+1,i+1+N}\right\}$
  up to lower-dimensional components.
\end{itemize} 

\newcommand\bsl{\backslash}

\begin{center}
\begin{tikzcd}
  E_\pi\arrow{dr}{\text{cut}} & & & & \bigcup_{\varepsilon(\pi,i)} E_\rho \arrow{dl}{\text{cut}}\\
        &E_\pi \cap H =: F_1& \supseteq & \bigcup_{\varepsilon(\pi,i)} X_\rho & &\\
      \end{tikzcd}
    \end{center}

Taking multidegrees, we will reproduce Equation (\ref{eqn:qkzev3}).
This is not quite a geometric {\em proof}, as we will only follow
the geometry close enough to get an upper bound, and invoke
Equation (\ref{eqn:qkzev3}) to show the bound is tight.

\junk{

The algebraic details of this proof, in \S\ref{ssec:lowerbound}, 
only make reference to $F_1$ and not explicitly $F_2,F_3$. 
However they only involve $\mdeg F_1$ in the form $(-\partial_i) \mdeg F_1$,
which is a sign that the computation is ``really'' about $F_2$, not $F_1$.
And since $F_2 \not\subseteq E$, it seems more natural to follow it to 
$F_3 \subseteq E$ (set containment). In \S\ref{ssec:f3} we follow
this to determine $F_3$, up to lower-dimensional embedded components.

Even to determine the dimensions of $F_1,F_2,F_3$ will require work,
as each cutting operation may or may not decrease the dimension by $1$.
It will turn out they do indeed, so $F_1,F_3$ are codimension $1$ in $E$,
and $F_2$ is the same dimension as $E$ (but is not contained in it,
even set-theoretically).  To analyze them, we will need to make use of
a certain family of codimension $1$ subvarieties of $E$.
}

\subsubsection{The subvarieties \texorpdfstring{$\{X_\rho\}$}{\{Xp\}}}\label{sec:Xrho}

Consider the link patterns $\rho \neq \pi$ such that $e_i\cdot \rho = \pi$,
and define
\begin{equation}\label{eq:defX}
  X_\rho := \left( E_\rho 
  \cap \left\{M \in E : (M^2)_{i,i+N} = (M^2)_{i+1,i+1+N}\right\} \right).
\end{equation}
Since that intersection is a Cartier divisor in the variety $E_\rho$, and the 
equation is nontrivial (it is not satisfied on the point $\widetilde\rho$
from Lemma \ref{lem:genericelts}), 
the set $X_\rho$ has pure codimension $1$ in $E$.

Each such $\rho$ agrees with $\pi$ away from $i,i+1,\pi(i),\pi(i+1)$.
There are three ways to hook up these four spots, namely $\pi$, $\rho$, and
$f_i\cdot \rho$, and as such these $\rho$ come in pairs.
Recall that $\varepsilon(\pi,i)$ denotes the set of such $\rho$
such that the chords emanating from $i,i+1$ cross each other.

\begin{Proposition}\label{prop:Xrho}
  Let $\rho\in \varepsilon(\pi,i)$.
  Then $X_\rho$ is irreducible,
  whereas $X_{f_i\cdot\rho}$ has two geometric components, one of which
  is $X_\rho$.
  Both schemes are generically reduced.
\end{Proposition}

\junk{
\comment{AK: I think we want to ditch the $Y_\rho$ stuff here, which 
  served to study $F_3$. But we MUST show $X_\rho$ is generically reduced}

As this proposition already suggests, it will be much more convenient
hereafter to {\em redefine}\/ $X_\rho$ as the reduction of the $X_\rho$ above.
}

In fact, we conjecture that $E_\rho$ is normal,
which would imply that $X_\rho$ is reduced. Of course $X_\rho$ and its reduction
have the same multidegree, which is what most interests us.

\junk{AK doesn't see where we use $X_{f_i\cdot\rho}$ anywhere. PZJ:
  agreed, not necessary, but would be nice to mention explicitly the
  structure of $F_3$}

\begin{proof}
By cyclic invariance we can assume that $i=1$.
We shall use in this proof the $(U,L)$ decomposition of \S\ref{ssec:brbr},
as well as the equations of Thm.~\ref{thm:compeqns}.
Recall from \cite{KZJ} that the projection $(U,L)\mapsto U$
allows one, using Thm.~\ref{thm:involutions}, to decompose $E$ as a disjoint
union of preimages $F_\alpha$ 
of orbits $B_N\cdot\alpha_<$, where $\alpha$ is an involution.
Furthermore this projection makes $F_\alpha$ a {\em vector bundle} over 
$B_N\cdot\alpha_<$, implying in particular that it is irreducible.

We use this decomposition $E = \bigsqcup_\alpha F_\alpha$ into 
irreducible varieties via the following easily proved statement:
if $X\subseteq E$ is set-theoretically equidimensional, 
then exactly one $X \cap F_\alpha$ is of that dimension iff
$X$ is irreducible. 

Let $\rho\in \varepsilon(\pi,i)$, $\alpha$ be an involution and consider
$X_\rho\cap F_\alpha$. We shall show
that if $\alpha\ne\rho$, the result is empty or of dimension strictly less than
$\dim E-1=2n^2-1$. 

$F_\alpha$ is a vector bundle over $B_N\cdot\alpha_<$; its dimension
was computed in the proof of Thm.~3 of \cite{KZJ} and found
to be $2n^2-\frac{1}{2}\#\text{fixed-points}$ 
(fixed points being half-lines in the
language of \S\ref{sec:orbvars}). Let us now impose additional equations
(from those in the proof of Thm.~3 of \cite{KZJ}) on the fiber
by intersecting with $X_\rho$. We choose the particular fiber
for which the upper triangular part is $U=\alpha_<$.

Requiring any element $M\in F_\alpha$
to satisfy the rank equations (3) of Thm.~\ref{thm:compeqns}
for $E_\rho$ implies that $\alpha_<\le \rho_<$
(with respect to the order from \S\ref{ssec:posetBorbits}), 
and in particular that $\alpha(i)\ne i+1$ (this is where we use that $i=1$). 
Now suppose that $i$ and $i+1$ are not both fixed points of $\alpha$.
Then the equation $(M^2)_{i,i+N} = (M^2)_{i+1,i+1+N}$
or more explicitly $L_{i\leftrightarrow\alpha(i)}=L_{i+1\leftrightarrow\alpha(i+1)}$
($a\leftrightarrow b:=\max(a,b),\min(a,b)$)
if neither are fixed points, or $L_{i\leftrightarrow\alpha(i)}=0$
or $L_{i+1\leftrightarrow\alpha(i+1)}=0$ if one of them is,
is an additional linear equation {\em on the fiber}\/ 
which reduces its dimension 
by $1$ and keeps it reduced. If $\alpha$ has fixed points,
we are already done since the resulting dimension is less 
than $2n^2-1=\dim E-1$. So we assume in what follows
that either $\alpha$ has no fixed points at all, or $\alpha(i)=i$ and 
$\alpha(i+1)=i+1$ are the only two (all other possibilities having
dimension too low). In both cases we are already in dimension $2n^2-1$.

Now we make use of $\alpha\neq\rho$, forcing $\alpha_< < \rho_<$. 
Note that $\pi_<\ge\rho_<$ and $(f_i\cdot\rho)_<\ge\rho_<$, 
so that $\alpha \notin \{\pi, \rho, f_i\cdot \rho\}$.
Hence $\alpha$ cannot 
be equal to $\rho$ outside $\{i,i+1,\rho(i),\rho(i+1)\}$ (paying
attention to the special case where $i$ and $i+1$ are fixed points); 
there is at least one more pair $(j,\rho(j))$ distinct from these
four elements such that $\alpha(j)\ne\rho(j)$. 
It is now easy
to check that equation (2) of Thm.~\ref{thm:compeqns} for $E_\rho$
(with $i$ replaced with $j$)
produces one more linear equation on the fiber
which reduces further the dimension by $1$. We conclude
that $X_\rho\cap F_\alpha$ is a subset of a vector bundle over $B_N\cdot\alpha_<$
which is of dimension $\dim E-2$.

Finally, if $\alpha=\rho$, we have
\[
X_\rho\cap F_\rho=F_\rho
\cap \left\{M \in E : (M^2)_{i,i+N} = (M^2)_{i+1,i+1+N}\right\}
\]
and conclude immediately from the equations above
that $X_\rho\cap F_\rho$, just like $F_\rho$, is a vector bundle over
$B_N\cdot\pi_<$ and therefore irreducible. 
Now look at the decomposition
\[ X_\rho = \Coprod_\alpha (X_\rho \cap F_\alpha)
= (X_\rho \cap F_\rho) \cup \Coprod_{\alpha\neq \rho} (X_\rho \cap F_\alpha):
\]
the first piece $X_\rho \cap F_\rho$ is reduced and irreducible, 
the other pieces $X_\rho \cap F_{\alpha\neq\rho}$ are of lower dimension,
and $X_\rho$ is set-theoretically equidimensional. Putting these facts together,
we see that $X_\rho$ is irreducible, and 
contains $X_\rho \cap F_\rho$ as a dense subset.

Just because $X_\rho \cap F_\rho$ is reduced and open dense in $X_\rho$,
we cannot directly infer that $X_\rho$ is generically reduced. 
To show this, a tangent space
calculation is needed; 
it is very similar to the proof of Lemma \ref{lem:singorb} in Appendix B,
so that we skip it here.

Now we take up $X_{f_i\cdot\rho} \cap F_\alpha$, whose analysis is similar.
Note that the inequality $\alpha_<\le (f_i\cdot\rho)_<$
has two interesting solutions, namely $\alpha=\rho$ or $\alpha=f_i\cdot\rho$.
We leave the reader to check that
the same dimension computation as above shows that
if $\alpha\notin \{ \rho,f_i\cdot\rho \}$, then the dimension of
$X_{f_i\cdot\rho} \cap F_\alpha$ is strictly less than $\dim E-1$.
Again, the condition of lying in $X_{f_i\cdot\rho}$ is a linear condition
on the fibers of this vector bundle, so the intersections are reduced
and irreducible.

If $\alpha=f_i\cdot\rho$, 
\[
X_{f_i\cdot\rho}\cap F_{f_i\cdot\rho}=F_{f_i\cdot\rho}
\cap \left\{M \in E : (M^2)_{i,i+N} = (M^2)_{i+1,i+1+N}\right\}
\]
so that $Y_\rho:=\overline{X_{f_i\cdot\rho}\cap F_{f_i\cdot\rho}}$
is a geometric component of $X_{f_i\cdot\rho}$ of dimension
$\dim E-1$.

If $\alpha=\rho$, we have
\[
X_{f_i\cdot\rho}\cap F_\rho=E_{f_i\cdot\rho}\cap (X_\rho\cap F_\rho)\subset X_\rho
\quad \text{as a set}
\]
so by dimensionality, $\overline{X_{f_i\cdot\rho}\cap F_\rho}=X_\rho$ as a set.

The tangent space calculation to show generic reducedness of $Y_\rho$
is in Lemma \ref{lem:regorb} of Appendix B.
\end{proof}

\subsubsection{\texorpdfstring{$F_1$}{F_1}}\label{ssec:f1}
First we note that the equation $M_{i,i+1} = 0$ does not hold on
the point $\widetilde\pi \in E_\pi$ constructed in Lemma \ref{lem:genericelts}.
So 
\[
 F_1 := E_\pi \cap \{M_{i,i+1} = 0\} 
\]
is a Cartier divisor in the irreducible variety $E_\pi$,
hence equidimensional of dimension $\dim E - 1$, and
by axiom (3c) of multidegrees,
\[ \mdeg F_1 = (\A + z_i - z_{i+1}) \mdeg E_\pi \]
where the linear factor is the multidegree of 
the hyperplane $\{M_{i,i+1} = 0\}$.

Since that hyperplane is $B_N$-invariant
inside $R_{\integers \bmod N}$, so is $F_1$ and each of its components.

\begin{Lemma}\label{lem:XrhoinF1}
  Let $\rho\in \varepsilon(\pi,i)$.
  Then $X_\rho \subseteq F_1$.
\junk{  , and is not $GL_2^{(i)}$-invariant.
  Moreover, the map 
  $GL_2^{(i)} \times^{B_N} X_\rho \to GL_2^{(i)} \cdot X_\rho$ is birational.}
\end{Lemma}

\begin{proof}
  A slightly more explicit description of $X_\rho$, found in Appendix B,
  implies (Lemma \ref{lem:inclX}) that $X_\rho\subset E_\pi$.
  Since $\rho(i)\ne i+1$, $E_\rho\subseteq \{ M_{i,i+1}=0\}$.
  So $X_\rho\subseteq E_\rho\cap E_\pi \subseteq E_\pi\cap\{ M_{i,i+1}=0\}=F_1$.
  \junk{AK: so do we need that lemma or don't we? PZJ: I think we do,
    $X_\rho$ as they're defined are not obviously in $E_\pi$}
\junk{
  We know 
  \begin{align*}
    \mdeg X_\rho 
 &= \mdeg E_\rho\ \mdeg\ \{M \in E: (M^2)_{i,i+N} = (M^2)_{i+1,i+1+N}\}\\
    &= (\A+\B)\ \mdeg E_\rho   
  \end{align*}
  If $X_\rho$ were $GL_2^{(i)}$-invariant, then its multidegree would be
  symmetric in $z_i,z_{i+1}$ by Lemma \ref{lem:divdiff}, but $E_\rho$'s
  multidegree isn't by \comment{which?}.

  \comment{is it easy to show that $\partial_i E_\rho$ has some 
    monomial with coefficient $1$, rather than doing more $X_\rho$ geometry,
    in order to show the birationality claim?}
}
\end{proof}

So the $X_\rho$ are (by construction distinct) irreducible components of $F_1$.
At this point, there would seem to be three ways that the set 
$\union_{\rho \in \varepsilon(\pi,i)} X_\rho$ might be strictly a lower bound on $F_1$:
\begin{enumerate}
\item We don't know that $F_1$ is generically reduced along each $X_\rho$.
\item $F_1$ may have some other components that are not $GL_2^{(i)}$-invariant.
\item $F_1$ may have some other components that are $GL_2^{(i)}$-invariant.
\end{enumerate}
It will turn out that (1) and (2) don't actually occur, but (3) does
(see the example after Corollary \ref{cor:poserrorterm}).

\subsubsection{Inequalities on multidegrees}
\label{ssec:mdegbounds}

Since we only have a bound $F_1 \supseteq \union_{\rho \in \varepsilon(\pi,i)} X_\rho$,
it becomes natural to seek similar ``bounds'' on the multidegrees,
and we develop that theory now.

Let $W$ be a representation of a torus $T$. Assume that all weights
$\lambda_1,\ldots,\lambda_{\dim W}$ of $W$ lie in a proper cone in $T^*$.
Then we can define a proper cone in $\Sym T^*$, 
consisting of $\naturals$-combinations of squarefree monomials 
$\{\prod_{i\in S} \lambda_i : S \subseteq \{1,\ldots,\dim W\} \}$ 
in the weights in $W$. 
We can define a partial order $\leq$ 
on $\Sym T^*$ where $f\leq g$ if $g-f$ lies in this cone.

\begin{Lemma}\label{lem:mdegineqs}
  Under the framework just described,
  for any nonempty $T$-invariant scheme $X \subseteq W$,
  $\mdeg_W X > 0$.
  
  Let $X,Y \subseteq W$ be two $T$-invariant subschemes of the same
  dimension, with $X\subseteq Y$. Then $\mdeg_W X \leq \mdeg_W Y$.
  This inequality is an equality iff $Y\setminus X$ is of lower dimension
  than $X,Y$.

  If $Y$ is generically
  reduced along its top-dimensional components, then $X$ is also. 
  Assume now that $\mdeg_W X = \mdeg_W Y$. 
  If $Y$ is reduced and equidimensional, then $X = Y$.
\end{Lemma}

(This last conclusion also appears in \cite[Lemma 1.7.5]{KM} and,
stated less generally, in \cite{Mar02}.)

\begin{proof}
  The condition that the weights lie in a proper cone is equivalent
  to the existence of a linear functional $T^* \to \reals$ taking
  each weight in $W$ to a strictly positive number. That functional 
  extends to a ring homomorphism $\Sym T^* \to \reals$, in which 
  $\naturals$-combinations of monomials in these weights (other than
  the null combination) also go to strictly positive numbers.

  With this one can show that this cone only intersects its negative
  in $\{0\}$, which says that $f\leq g\leq f \implies f=g$. The other
  properties of a partial order are obvious. The fact that
  $\mdeg_W X$ is a nontrivial sum of products of monomials in $W$
  is easy to prove by induction from the axiomatic definition 
  of multidegrees.

  Let $\{Y_i\}$ be the top-dimensional components of $Y$, with 
  multiplicities $\{m_i(Y)\}$, so $\mdeg_H Y = \sum_i m_i(Y) \mdeg_H Y_i$.
  Since $X\subseteq Y$ and they have the same dimension, each
  $\dim X$-dimensional component of $X$ is a component of $Y$ 
  hence of some $Y_i$; let $m_i(X)$ be the multiplicity of $Y_i$ in $X$.
  Then $X \subseteq Y$ implies $m_i(X) \leq m_i(Y)$, so
  $\mdeg_H Y - \mdeg_H X = \sum_i (m_i(Y)-m_i(X)) \mdeg_H Y_i$.
  Since each $\mdeg_H Y_i > 0$, this sum is $\geq 0$ with equality
  iff $m_i(Y) = m_i(X)$ for each $i$, iff $X$ contains all of
  $Y$'s top-dimensional components and is generically equal to $Y$
  on each one. That is equivalent to $Y\setminus X$ being of lower dimension.

  Since $Y\supseteq X$ of the same dimension, if $Y$ is generically
  reduced along its top-dimensional components, then $X$ is also.
  If $Y$ is reduced and equidimensional, then $Y$ is the union of
  its geometric components. If $\mdeg_H X = \mdeg_H Y$, 
  then all of these components appear in $X$; 
  hence $X \supseteq Y$, but we were given $X \subseteq Y$.
\end{proof}

This combines nicely with Lemma \ref{lem:divdiff}:

\begin{Lemma}\label{lem:divdiffbound}
  Let $V \leq \MNC$ be a subspace invariant under $B_N$ and $GL_2^{(i)}$.
  Let $X$ be an equidimensional scheme in $V$ invariant under 
  $B_N$ and rescaling, with multidegree $\mdeg_V X$.
  If each geometric component of $X$ is $GL_2^{(i)}$-invariant, then 
  $\der_i\ \mdeg X = 0$. Otherwise,
  \[ (-\der_i) \ \mdeg X\ \geq\ \mdeg (GL_2^{(i)} \cdot X), \]
  where $GL_2^{(i)} \cdot X$ is given the reduced scheme structure.
  This is an equality iff the degrees $\{k_j\}$ from lemma \ref{lem:divdiff} 
  are $0$ or $1$, and each $GL_2^{(i)}$-non-invariant component $X_j$
  is generically reduced.

  In particular, $\der_i\ \mdeg X = 0$ implies each geometric component
  of $X$ is $GL_2^{(i)}$-invariant.

  Also, if $p \in \Sym T^*$ has $p \geq 0$, then $(-\der_i) p \geq 0$ also.
\end{Lemma}

\begin{proof}
  Break $X$ into its components $X_j$, with multiplicities $m_j>0$, 
  and apply Lemma \ref{lem:divdiff},
  \begin{align*}
    (-\der_i) \mdeg X 
    &= (-\der_i) \sum_{X_j} m_j\ \mdeg X_j
    = \sum_{X_j} m_j (-\der_i) \mdeg X_j \\
    &= \sum_{X_j} m_j k_j\ \mdeg (GL_2^{(i)} \cdot X_j)
    = \sum_{X_j:\ k_j\neq 0} m_j k_j\ \mdeg (GL_2^{(i)} \cdot X_j) \\
    &\geq \sum_{X_j:\ k_j\neq 0} \mdeg (GL_2^{(i)} \cdot X_j),
  \end{align*}
  using also the fact that $\mdeg (GL_2^{(i)} \cdot X_j) > 0$.
  For this to be an equality, each $k_j \neq 0$ must be equal to $1$,
  and the corresponding $m_j$ must be $1$. For the left side to be zero,
  all $k_j$ must be zero, which by Lemma \ref{lem:divdiff} says that
  the components are $GL_2^{(i)}$-invariant.

  For the last conclusion, it is enough to check the case $p$ a
  product of distant weights, the multidegree of a coordinate subspace $W$.  
  If $W$ is $GL_2^{(i)}$-invariant, then $\der_i p = 0$, and otherwise
  we apply the inequality above.
\end{proof}

\subsubsection{The multidegree of this lower bound 
  \texorpdfstring{$\union_{\rho \in \varepsilon(\pi,i)} X_\rho$}{Up Xp}}\label{ssec:lowerbound}

It will be convenient to work in and compute multidegrees relative to a
$GL_2^{(i)}$-invariant space $P_i$, 
one dimension larger than $R_{\integers\bmod N}$ in that it includes the 
possibility that the $\wtM_{i+1,i}$ entry just below the diagonal 
may be nonzero. Then for $X\subseteq R_{\integers\bmod N}$ and $T$-invariant,
\[ \mdeg_{P_i} X = (\A+z_{i+1}-z_i)\ \mdeg X \]
and in particular,
\[
 \mdeg_{P_i} E_\pi = (\A+z_{i+1}-z_i) \Psi_\pi. 
\]

Without the $-\der_i$, the left side of Equation (\ref{eqn:qkzev3}) is
\begin{align*}
  & (\A+z_i-z_{i+1})(\A+z_{i+1}-z_i) \Psi_\pi  \\
  &=  (\A+z_{i+1}-z_i) \mdeg_{P_i} E_\pi \\
  &=  \mdeg_{P_i} \left(E_\pi \cap \{ M_{i\ i+1} = 0\}\right)  
  && \text{by Axiom 3(c) of multidegrees} \\
  &\geq  \mdeg_{P_i} \bigcup_{\rho \in \varepsilon(\pi,i)} X_\rho 
  =   \sum_{\rho \in \varepsilon(\pi,i)} \mdeg_{P_i} X_\rho 
  && \text{by Lemma \ref{lem:XrhoinF1}} \\
  &=  \sum_{\rho \in \varepsilon(\pi,i)} \mdeg_{P_i} 
  \left(E_\rho \cap \{ M : (M^2)_{i,i+N} = (M^2)_{i+1,i+1+N}\}\right) 
  && \text{by Proposition \ref{prop:Xrho}}\\
  &= (\A+\B)  \sum_{\rho \in \varepsilon(\pi,i)} \mdeg_{P_i} E_\rho 
  && \text{by Axiom 3(c) of multidegrees} \\
  &= (\A+\B) (\A+z_{i+1}-z_i) \sum_{\rho \in \varepsilon(\pi,i)} \Psi_\rho 
\end{align*}
Applying $-\der_i$ to both sides, Lemma \ref{lem:divdiffbound} gets us
Equation (\ref{eqn:qkzev3}), but only as an inequality.

Invoking Corollary \ref{cor:fgivese}, which says that resulting 
inequality is tight,
we learn that $F_1$ is generically reduced along each $X_\rho$ and has
no other non-$GL_2^{(i)}$-invariant components. Conversely, a direct
proof of those geometric facts would serve as a proof of Equation
(\ref{eqn:qkzev2}) without using Equation (\ref{eqn:qkzdv2}).

\begin{Corollary}\label{cor:poserrorterm}
  If $\pi(i)=i+1$, then the polynomial
  \[ (\A+z_i-z_{i+1}) \Psi_\pi 
  \quad-\quad  (\A+\B)  \sum_{\rho \in \varepsilon(\pi,i)} \Psi_\rho
  \]
  lies in $\naturals\left[ \{\A+z_i-z_j,\B+z_j-z_i : i< j\} \right]$.
  When multiplied by $\A+z_{i+1}-z_i$, it is symmetric 
  under $z_i \leftrightarrow z_{i+1}$.
\end{Corollary}

\begin{proof}
  This is the $\mdeg_{R_{\integers \bmod N}}$ of the closure of 
  \[ \left( E_\pi \cap \{M_{i,i+1} = 0\} \right)   \quad\setminus \quad
  \bigcup_{\rho\in \varepsilon(\pi,i)} 
  \left( E_\rho \cap \{ (M^2)_{i,i+N} = (M^2)_{i+1,i+N+1} \} \right), \]
  which when multiplied by $\A+z_{i+1}-z_i$ (to turn it into the $\mdeg_{P_i}$)
  becomes the slack term in the chain of inequalities above.
\end{proof}

For example, if $2n=4$, $\pi = (23)(14)$, and $i=2$, 
so $\varepsilon(\pi,i)=\{(13)(24)\}$, then this set difference is
the space of matrices of the form
\[
\begin{matrix}
\ddots & &&&&&& \\
& 0 & 0 & 0 & m_{14} &        &        & &\\
&   & 0 & 0 & 0      & m_{25} &        & &\\
&   &   & 0 & 0      & m_{35} & m_{36} & &\\
&   &   &   & 0      & m_{45} & m_{36} & m_{47} &\\
&&&&&\ddots &&&
\end{matrix}
\qquad\qquad \in R_{\integers \bmod N}
\]
with multidegree (in $R_N$)
\[ (A+z_2-z_4)(A+z_3-z_4)\ \ (A+z_1-z_2)(A+z_1-z_3)\ \ (A+z_2-z_3). \]

\subsection{From Brauer loop polynomials 
  back to Joseph--Melnikov polynomials}
\label{ssec:degenbrauer}
Finally, we comment on the connection between the results of
\S~\ref{ssec:hotta} concerning Joseph--Melnikov polynomials
$J_\pi$ and those of \S~\ref{ssec:braueraction} concerning Brauer
loop polynomials $\Psi_\pi$.  Recall that these two sets of
multidegrees are related by Theorem~\ref{thm:PsiJM}. Thus taking the
$\B\to\infty$ limit in the equations
(\ref{eqn:qkzdv2},\ref{eqn:qkzev2}) satisfied by the $\Psi_\pi$
should result in equations satisfied by the $J_\pi$.

Let us first consider Eq.~(\ref{eqn:qkzdv2}). For any link pattern
$\pi$ (i.e., fixed-point-free involution), there are
only two possibilities: either $f_i\cdot\pi$ has one more crossing
than $\pi$, or one fewer. If it has one more crossing, then the right
hand side is of lower degree in $\B$ than the left hand side, and we
find that $\tilde\der_i J_\pi=0$.  If it has one fewer crossing, then
taking $\B$-leading terms on both sides of the equation results in
Eq.~(\ref{eqn:fiJeq}) with the identification $f_i\cdot\pi=\bar
f_i^{-1}\cdot\pi$.  (The meaning of this strange change of notation
will be explained in the final paragraph.)

Next start from Eq.~(\ref{eqn:qkzev2}), valid for any link pattern
$\pi$, and send $\B$ to infinity. In order to compare degrees in $\B$
of the various terms in the sum over $\pi'$, $e_i\cdot\pi'=\pi$, we
need to compute their number of crossings. But since $e_i$ cannot
create crossings, the number of crossings of $\pi'$ is greater or
equal to that of $\pi$.  Thus, the $\B$-leading term of
Eq.~(\ref{eqn:qkzev2}) is exactly Eq.~(\ref{eqn:hottafinal}), in
which one must sum over $\pi'$ that are preimages of $\pi$ and have
the same number of crossings as $\pi$. This is equivalent to the
prescription given in the text after Eq.~(\ref{eqn:hottafinal}).

Finally, from the algebraic point of view, note that $\B\to\infty$
corresponds to the parameter $\beta$ of the Brauer algebra 
$\Br_N(\beta)$
being sent to the value $2$. This is a degenerate situation in which the
$R$-matrix given by Eq.~(\ref{eqn:Rmat}) loses its term proportional
to $f_i$, and so doing becomes the rational Temperley--Lieb $R$-matrix
(see \cite{DFZJ05b,RTVZJ} for a related discussion of the Temperley--Lieb
$q$KZ equation). This strongly suggests that a more interesting point
of view is to replace $f_i$ by $\bar f_i:=(1-\beta/2)f_i$ and only
then take the limit $\beta\to 2$. The resulting algebra, the
{\dfn degenerate Brauer ($\beta=2$) algebra} $\bar\Br_N(2)$, 
is given by generators $e_i$,
$\bar f_i$, $i=1,\ldots,N-1$ and relations
\begin{align}
e_i^2&=2 e_i& e_ie_{i\pm 1}e_i&=e_i\\
\notag \bar f_i^2&=0& \bar f_i\bar f_{i+1}\bar f_i&=\bar f_{i+1}\bar f_i\bar f_{i+1}\\
\notag \bar f_ie_i&=e_i\bar f_i=0& \bar f_i e_{i\pm 1} e_i &= \bar f_{i\pm 1} e_i&
\notag e_i e_{i\pm 1} \bar f_i &= e_i \bar f_{i\pm 1}\\
\notag e_i e_j&=e_j e_i& \bar f_i \bar f_j&=\bar f_j \bar f_i& e_i \bar f_j&=\bar f_je_i& |i-j|>1
\end{align}
Its action on linear combinations of link patterns is the same as
usual, with the additional rule that if a link pattern $\pi$ is such
that the arches coming out of $i$ and $i+1$ cross, then $\bar f_i\cdot\pi=0$. 
It now has a non-trivial $R$-matrix and this way, the various
equations satisfied by the $J_\pi$ are its $q$KZ equation.
(This also explains the notation $\bar f_i^{-1}\cdot\pi$
used before, since $\bar f_i$, contrary to $f_i$, is not an involution or even invertible:
by $\bar f_i^{-1}\cdot\pi$ we mean the unique preimage of $\pi$ when it exists).
Note in particular that the $\bar f_i$ generate a subalgebra called the nil-Hecke algebra,
which was discussed in a similar context in \cite{FK}.

\appendix
\section{The affine Weyl group \texorpdfstring{$\mathcal{\hat S}_N$}{SN}}\label{sec:affweyl}
The affine Weyl group of type A, denoted here $\mathcal{\hat S}_N$, 
is defined by generators
$f_i$, $i\in \integers/N\integers$, and relations
\begin{equation}
  f_i^2=1\qquad (f_if_{i+1})^3=1\qquad f_i f_j=f_j f_i\qquad j\ne i-1,i+1
\end{equation}
It is a semi-direct product $\mathcal{S}_N \ltimes \integers^{N-1}$, as
will be made explicit now.

First define an alternative description of $\mathcal{\hat S}_N$ which is
particularly convenient for our purposes.  Let $\star$ be the canonical
projection from $\integers$ to $\integers/N\integers$.  Consider the
group $F_N$ of invertible maps $\phi$ from $\integers$ to $\integers$
such that $\phi(i+N)=\phi(i)+N$ for all $i\in\integers$, endowed with
composition.  
(This group appeared also in \cite{ER}, where it is the
``group of juggling patterns with period $N$''.)
Then it is easy to show that there exists an injective
morphism $\iota$ from $\mathcal{\hat S}_N$ to $F_N$ such that
\begin{equation}\label{eqn:injec}
\iota(f_i): j\mapsto 
\begin{cases}
j&j^\star\ne i,i+1\\
j+1&j^\star=i\\
j-1&j^\star=i+1
\end{cases}
\end{equation}
Its image is precisely the maps $\phi\in F_N$ such that 
$\sum_{i=1}^N (\phi(i)-i)=0$ (``juggling patterns with $0$ balls'').
From now on we identify $\mathcal{\hat S}_N$ with its image in $F_N$.

Next define the projection $p$ from $\mathcal{\hat S}_N$ to $\mathcal{S}_N$
viewed as the group of permutations of $\integers/N\integers$.  Any
$\phi\in F_N$ has a unique factorization $\star \phi=p(\phi) \star$
with $p(\phi)\in \mathcal{S}_N$, and in particular by restriction we get
a map $p: \mathcal{\hat S}_N\to \mathcal{S}_N$.

The kernel of $p$ is made of maps $\phi$ such that $\phi(i)=i\mod N$
for all $i\in \integers$.  Thus it is isomorphic to $(\integers^N,+)$
via $\phi\mapsto ((\phi(1)-1)/N,\ldots,(\phi(N)-N)/N)$. Restricting to
$\mathcal{\hat S}_N$ we obtain a subgroup $\{ (k_1,\ldots,k_N):
\sum_{i=1}^N k_i=0 \}$ of $\integers^N$ isomorphic to
$\integers^{N-1}$.

It is now an easy exercise to conclude that
$\mathcal{\hat S}_N\cong \mathcal{S}_N \ltimes \integers^{N-1}$.  We may
choose as a particular subgroup isomorphic to $\mathcal{S}_N$ the one
generated by $f_1,\ldots,f_{N-1}$.

In this paper, particular subsets $\mathcal{\hat S}_{\pi,\pi'}$ of $\mathcal{\hat S}_N$ are defined, see Eq.~(\ref{eqn:defred});
here $\pi$ and $\pi'$ are two involutions of $\integers/N\integers$ without fixed points, on which elements $s$ of $\mathcal{\hat S}_N$
act by conjugation by $p(s)$. Using our alternative description of $\mathcal{\hat S}_N$ we can find a much simpler characterization
of $\mathcal{\hat S}_{\pi,\pi'}$:
\begin{Proposition}\label{prop:altgroupoid}
\begin{equation*}
  \mathcal{\hat S}_{\pi,\pi'}=\{ s\in \mathcal{\hat S}_N, s\cdot\pi=\pi'
  \ \big|\  \forall i,j\in\integers,
  \pi(i^\star)=j^\star\ \text{and}\ i<j \Rightarrow s(i)<s(j)\}
\end{equation*}
\end{Proposition}

\begin{proof}
  Let $s=f_{i_k}\cdots f_{i_1}$ be an element of $\mathcal{\hat S}_{\pi,\pi'}$ 
as in Eq.~(\ref{eqn:defred}).  We prove by
  induction on $k$ that $s$ satisfies the property of
  Prop.~\ref{prop:altgroupoid} (the induction is for all $\pi'$
  simultaneously).  It is trivial at $k=0$; to go from $k-1$ to $k$,
  write $s=f_{i_k} s'$ with $s'=f_{i_{k-1}}\cdots f_{i_1}$, and pick a
  pair of integers $i,j$ with $\pi(i^\star)=j^\star$. Due to the
  induction hypothesis we know that $s'(i)<s'(j)$ and want to apply
  $f_{i_k}$ to both sides of the inequality. Since the effect of
  $f_{i_{k}}$ is only to increase/decrease by $1$ (c.f.
  Eq.~(\ref{eqn:injec})), and it only affects $i_k$ and $i_k+1$, we
  have automatically $s(i)<s(j)$ unless $s'(i)^\star=i_k$ and
  $s'(j)=s'(i)+1$. But this contradicts the defining property in
  Eq.~(\ref{eqn:defred}) at $\ell=k$: indeed we would have
  $(f_{i_{k-1}}\cdots f_{i_1}\cdot\pi)(i_k)
  =(s'\cdot\pi)(i_k)=p(s')(\pi(p(s')^{-1}(i_k)))
  =p(s')(\pi(i^\star))=p(s')(j^\star)=i_{k+1}$.
  
  Conversely, assume $s$ satisfies the property of
  Prop.~\ref{prop:altgroupoid}, and write a decomposition
  $s=f_{i_k}\cdots f_{i_1}$ of {\em minimum length\/} $k$.  We claim
  this word satisfies Eq.~(\ref{eqn:defred}). To abbreviate let us
  denote $w_m=f_{i_m}\cdots f_{i_1}$, and assume there is a step
  $\ell$ such that $(w_{\ell-1}\cdot \pi)(i_\ell)= i_\ell+1$.  In
  other words there is a pair $i,j$ such that $j^\star=\pi(i^\star)$
  and $w_{\ell-1}(i)=i_\ell$, $w_{\ell-1}(j)=i_\ell+1$ -- hence also
  $w_\ell(i)=i_{\ell}+1$, $w_\ell(j)=i_\ell$.  Consider now $S=\{
  m=1,\ldots,k : (w_{m-1}(i)-w_{m-1}(j))(w_m(i)-w_m(j))<0\}$.  $S$ is
  non-empty since $\ell\in S$. Furthermore, the property of
  Prop.~\ref{prop:altgroupoid} implies that $w_0(i)-w_0(j)=i-j$ and
  $w_k(i)-w_k(j)=s(i)-s(j)$ have same sign; therefore $S$ has even
  cardinality.  We may then pick a pair of distinct elements, say
  $\ell,\ell'\in S$ and remove them from the word: it is simple to
  check that the new word
  $f_{i_k}\cdots\widehat{f_{i_\ell}}\cdots\widehat{f_{i_{\ell'}}}\cdots
  f_{i_1}$ still equals $s$, which contradicts the hypothesis of
  minimum length of the original word.
\end{proof}

We can now obtain the 
\begin{Corollary*} (Lemma \ref{lem:stab})
  $\mathcal{\hat S}_{\pi_0,\pi_0}$ is the subgroup of $\mathcal{\hat S}_N$
  generated by the $f_i f_{i+n}$, $i\in\integers/N\integers$.
\end{Corollary*}

\begin{proof}
That the $f_i f_{i+n}$ belong to $\mathcal{\hat S}_{\pi_0,\pi_0}$ is elementary. 

Conversely, consider $s\in\mathcal{\hat S}_{\pi_0,\pi_0}$. By successive
multiplications by $f_i f_{i+n}$ we want to reduce it to the identity.
Apply Prop.~\ref{prop:altgroupoid}:
\begin{equation}\label{eqn:stab}
  \mathcal{\hat S}_{\pi_0,\pi_0}
  =\{ s\in\mathcal{\hat S}_N, s\cdot\pi_0=\pi_0 
  : \forall i\in\integers\ s(i)<s(i+n)\}
\end{equation}

The first required property is that $p(s)$ commute with the involution
$\pi_0$. We know that the group of such permutations is isomorphic to
$\mathcal{S}_n \ltimes (\integers/2\integers)^n$ where $\mathcal{S}_n$
permutes the $n$ cycles and the $\integers/2\integers$ permute the two
elements of each cycle. In this formulation, $f_i f_{i+n}$ can be
viewed as the elementary transposition $(i,i+1)$ of $\mathcal{S}_n$.
Therefore, by successive multiplications by $f_i f_{i+n}$ one can
assume that $p(s)$ preserves each cycle of $\pi_0$, i.e., $s(i)=i\mod
n$ for all $i$.

Define the integers $\tilde k_i=(s(i)-i)/n$. 
According to Eq.~(\ref{eqn:stab}), $\tilde k_i<1+\tilde k_{i+n}<2+\tilde k_i$ 
and therefore $\tilde k_i=\tilde k_{i+n}$.
Similarly as before, $s\mapsto (\tilde k_1,\ldots,\tilde k_n)$ provides
an injective morphism from the $s\in\mathcal{\hat S}_{\pi_0,\pi_0}$ such that $p(s)$ preserves each cycle of $\pi_0$ to the $n$-uplets
$(\tilde k_1,\ldots,\tilde k_n)$ such that $\sum_{i=1}^n \tilde k_i=0$.

We finally multiply $s$ by elements of the form $T_i=U_i U_{i+1}^{-1}$, $i=1,\ldots,n$, where 
$U_i=f_{i+n}f_i\,f_{i+n-1}f_{i-1}\cdots f_{i+1}f_{i+n+1} \,f_i f_{i+n}$.
The $T_i$ also preserve
each cycle of $\pi_0$ and correspond to the values $\tilde k_i=\tilde k_{i+n}=+1$, $\tilde k_{i+1}=\tilde k_{i+n+1}=-1$ and the other
$\tilde k_j=0$. This clearly allows to reduce to $\tilde k_i=0$, i.e., $s=1$.
\end{proof}

{\em Remark:} In the alternative description of $\mathcal{\hat S}_N$ one
could introduce the extra map $r: i\mapsto i+1$ (``the standard
$1$-ball juggling pattern''). This would be the
proper abstract element corresponding to the operator $r$ on $V$
introduced in the text, such that $r f_i r^{-1}=f_{i+1}$.

\section{More on \texorpdfstring{$X_\rho$}{Xp}}
In \S\ref{sec:ei} the following varieties
are introduced
\[
X_\rho=E_\rho
\cap \left\{M : (M^2)_{i,i+N} = (M^2)_{i+1,i+1+N}\right\}
\]
In this appendix,
$\rho$ will always be a link pattern such that $\rho(i)\ne i+1$ and the chords
coming out of $i$ and $i+1$ {\em cross}. 

Our main 
goal is to describe explicitly a dense subset of orbits inside $X_\rho$.
This is slightly tricky because these orbits do not have the simple
structure that is found in generic orbits of $E_\pi$ 
(cf Prop.~3 of \cite{KZJ} and Lemma \ref{lem:regorb} below)
i.e., they do not possess representatives which are permutation matrices.
We define $\mathcal U_N=\{M\in\MMN: M_{ii}=1\}$.

\begin{Lemma}\label{lem:singorb}
Consider matrices $M\in E$ of the form
\[
\kbordermatrix{
&i&i+1&\rho(i)&\rho(i+1)&i+N&i+N+1&\rho(i)+N\\
i&0& 0 & t_i & 0  & \cdots \\
i+1& & 0 &  0  & t_{i+1}& 0 & \cdots \\
\rho(i)& &   &  0  & 0     & t_{\rho(i)} & x & \cdots \\
\rho(i+1)& &   &     & 0     & 0         & t_{\rho(i+1)} & y\\
}
\]
in the rows and columns for which we chose
$\mathrm{mod}\ N$ representatives of the form $i<i+1<\rho(i)<\rho(i+1)<i+N$,
where the parameters satisfy $t_i t_{\rho(i)}=t_{i+1}t_{\rho(i+1)}\ne0$;
and whose other non-zero entries are $M_{k,l}=t_k$ where
$\rho(k)=l$ and $k<l<k+N$. Then the union of their orbits
by conjugation by $\mathcal U_N$ is a dense subset of $X_\rho$.
\end{Lemma}
\begin{proof}
First consider the projection $(U,L)\mapsto U$ 
(we may assume by cyclic invariance $i=1$, which means that
on the picture of the lemma, the projection
corresponds to keeping the first 4 columns of the matrix). The resulting
matrix $U$ has the same non-zero entries as $\rho_<$ and so is in
$B_N\cdot \rho_<$. Furthermore
$(M^2)_{i,i+N} =t_i t_{\rho(i)}$ and $(M^2)_{i+1,i+1+N}=t_{i+1}t_{\rho(i+1)}$,
so these two quantities are equal. The matrices of the lemma
therefore belong to $F_\rho \cap \{ M: (M^2)_{i,i+N}=(M^2)_{i+1,i+1+N}\}
\subset X_\rho$. $X_\rho$ is
by definition stable by conjugation by $\mathcal U_N$, 
so their orbits sit inside $X_\rho$.

Next we compute the dimension of a single (generic) orbit.
For that we consider the infinitesimal stabilizer of $\mathcal U_N$ on
$M$ of the form of the lemma. The equation $M P = P M$ (where $P$
is strictly upper triangular) takes
the same form as in Thm.~4 of \cite{KZJ} 
(correcting a small error in the proof), namely
for each pair of chords of $\rho$ we have associated equations:
$\{j,\rho(j)\}$ and $\{k,\rho(k)\}$:
  \begin{enumerate}
  \item The chords $\{j,\rho(j)\}$ and $\{k,\rho(k)\}$ do not cross each
    other:
\linkpattern[shape=circle,alias=false,numbering={1/k,2/\rho(k),3/j,4/\rho(j)}]{1/2,3/4} 
in which case we can choose representatives $j<\rho(j)< k<\rho(k)<j+N$
and we find:
    \begin{align*}
      t_{\rho(k)} P_{jk} + t_{\rho(j)} P_{\rho(j)\rho(k)} &= 0 \\
      t_{\rho(j)} P_{kj} + t_{\rho(k)} P_{\rho(k)\rho(j)} &= 0\\
      t_{\rho(j)} P_{\rho(j)k} &= 0 \\
      t_{\rho(k)} P_{\rho(k)j} &= 0\\
      t_{\rho(k)} P_{\rho(j)k} &= 0 \\
      t_{\rho(j)} P_{\rho(k)j} &= 0
    \end{align*}
    (note that they form groups of two, related by a rotation of
    $180^\circ$ or equivalently exchange of $j$ and $k$).  
    For generic (i.e., non-zero) $t$'s there are four independent equations.
  \item The chords $\{ j,\rho(j)\}$ and $\{ k,\rho(k)\}$ cross each other:
\linkpattern[shape=circle,alias=false,numbering={1/\rho(j),2/\rho(k),3/j,4/k}]{1/3,2/4} 
    in which case we can choose representatives
    $j< k<\rho(j)<\rho(k)<j+N$ and we find
    \begin{align*}
      t_j P_{j,k} + t_k P_{\rho(j),\rho(k)} &= 0\\
      t_k P_{k,\rho(j)} + t_{\rho(j)} P_{\rho(k),j} &=0 \\
      t_{\rho(j)} P_{\rho(j)\rho(k)} + t_{\rho(k)} P_{j,k} &=0 \\
      t_{\rho(k)} P_{\rho(k)j} + t_j P_{k,\rho(j)} &=0
    \end{align*}
    (all these equations are obtained from each other by rotation of
    $90^\circ$, which is the symmetry of the diagram). 
    If $\{i,i+1\}\not\subset \{j,k,\rho(j),\rho(k)\}$, then generically, 
    $t_j t_{\rho(j)}\ne t_k t_{\rho(k)}$ and the linear system is non-degenerate,
    so that there are exactly four independent equations.

    However, if $\{i,i+1\}\subset\{j,k,\rho(j),\rho(k)\}$
    (note that by hypothesis the chords coming from $i,i+1$ are crossing),
    then we have $t_it_{\rho(i)}=t_{i+1}t_{\rho(i+1)}$ and the linear
    system becomes degenerate, so that we find only two independent equations.
\end{enumerate}

  The conclusion is that each pair of chords contributes $4$
  equations, except one of them that contributes $2$, hence a total
  of $4\times n(n-1)/2-2=2(n^2-n-1)$. The dimension of an orbit
  is the dimension of the group minus the dimension of the stabilizer,
  which is precisely this number of equations, that is $2(n^2-n-1)$.

Next we check that each orbit possesses a unique representative of the
form of the lemma. Write $M P=P M'$, $P\in \mathcal U_N$.
It is perhaps useful to write out explictly $M P-PM'$
in the rows and columns of interest:
\[
\kbordermatrix{\gdef\s{\scriptstyle}\gdef\sx#1#2{{\s #1\atop\s\qquad #2}}
&\rho(i)&\rho(i+1)&i+N&i+N+1&\rho(i)+N\\
i&\s t_i-t'_i &\sx{P_{\rho(i),\rho(i+1)}t_i}{-P_{i,i+1}t'_{i+1}}  & \cdots \\
i+1& \s 0  &\s t_{i+1}-t'_{i+1}&\sx{P_{\rho(i+1),i}t_{i+1}}{-P_{i+1,\rho(i)}t'_{\rho(i)}} & \cdots \\
\rho(i)&  \s 0  & \s 0     &\s t_{\rho(i)}-t'_{\rho(i)} &\sx{x-x'+
P_{i,i+1}t_{\rho(i)}}{-P_{\rho(i),\rho(i+1)}t_{\rho(i+1)}} & \cdots \\
\rho(i+1)&     & \s 0     & \s 0         &\s t_{\rho(i+1)}-t'_{\rho(i+1)} &\sx{y-y'+
P_{i+1,\rho(i)}t_{\rho(i+1)}}{-P_{\rho(i+1),i}t'_i}\\
}
\]
In fact the entry $(k,\rho(k))$ reads $t_k=t'_k$, for all $k$.
Now if we use once again the relation $t_i t_{\rho(i)}=t_{i+1}t_{\rho(i+1)}$
we find that the remaining four non-trivial equations simplify so that
$x=x'$, $y=y'$. So the representatives are unique.

Finally, the space of matrices of the form of the lemma is
$2n+2$ parameters minus $1$ equation, that is $2n+1$.
So the union of orbits has dimension $2(n^2-n-1)+2n+1=2n^2-1=\dim X_\rho$
and we conclude by irreducibility of $X_\rho$ that it is dense.
\end{proof}

From this we immediately deduce
\begin{Lemma}\label{lem:inclX}
\[
X_\rho \subseteq
 E_\pi \cap E_\rho \cap E_{f_i\cdot\rho}
\]
where $\pi$ and $f_i\cdot\rho$ are the link patterns
obtained from $\rho$ by ``uncrossing'' the chords coming from $i,i+1$
in the two possible ways.
\end{Lemma}
\begin{proof}
$X_\rho\subset E_\rho$ by definition.
The other two inclusions are obtained by checking them on the 
orbit representatives
of lemma \ref{lem:singorb}. From this point of view $E_\pi$ and $E_{f_i\cdot\rho}$
play strictly identical roles, so we do the reasoning for $E_\pi$ only.
Consider the one-parameter family of matrices $M_z$ which are equal to
$M$ in lemma \ref{lem:singorb} except:
\[
\kbordermatrix{
&i&i+1&\rho(i)&\rho(i+1)&i+N&i+N+1&\rho(i)+N\\
i&0& t_i z & t_i & 0  & \cdots \\
i+1& & 0 &  0  & t_{i+1}& 0 & \cdots \\
\rho(i)& &   &  0  & -t_{i+1}z     & t_{\rho(i)} & x & \cdots \\
\rho(i+1)& &   &     & 0     & 0         & t_{\rho(i+1)} & y\\
}
\]
Using once more $t_i t_{\rho(i)}=t_{i+1}t_{\rho(i+1)}$, we check that $M_z^2=0$.
For $z=0$ we recover the matrix of the lemma. Furthermore,
\begin{align*}
(M_z^2)_{i,i+N}&=t_i t_{\rho(i)}\\
(M_z^2)_{i+1,i+N+1}&=t_{i+1} t_{\rho(i+1)}\\
(M_z^2)_{\rho(i),\rho(i)+N}&=t_i t_{\rho(i)}-y z t_{i+1}\\
(M_z^2)_{\rho(i+1),\rho(i+1)+N}&=t_{i+1} t_{\rho(i+1)}-y z t_{i+1}
\end{align*}
so for generic $z, y, t_k$ we have the following coincidences
of $(M_z^2)_{k,k+N}$ (and no other):
$(M_z^2)_{i,i+N}=(M_z^2)_{i+1,i+1+N}$,
$(M_z^2)_{\rho(i),\rho(i)+N}=(M_z^2)_{\rho(i+1),\rho(i+1)+N}$
and $(M_z^2)_{k,k+N}=(M_z^2)_{\rho(k),\rho(k)+N}$ for $k\ne i,i+1,\rho(i),\rho(i+1)$. These are exactly the chords of $\pi$.
According to Thm.~\ref{thm:Ecomps}, this implies that $M_z\in E_\pi$.
Since $E_\pi$ is closed, $M_0\in E_\pi$.
\end{proof}
In fact, it is a consequence of the conjecture that equations of Thm.~\ref{thm:compeqns} define $E_\pi$,
that $X_\rho$ is equal to the intersection of any pairs of
the three components $E_\pi,E_\rho,E_{f_i\cdot\rho}$.

We now briefly discuss $X_{f_i\cdot\rho}$. 
Recall that Prop.~\ref{prop:Xrho} says that it has two components,
one being $X_\rho$ and the other one called $Y_\rho$.
\begin{Lemma}\label{lem:regorb}
The irreducible set $\mathcal U_N\cdot \{ \underline{f_i\cdot\rho}\, t: t\in T\ 
\text{and}\ t_i t_{\rho(i+1)}=t_{i+1}t_{\rho(i)}\}$ is dense in $Y_\rho$.
\end{Lemma}
\begin{proof}
First note that 
\[
\{ \underline{f_i\cdot\rho}\, t: t\in T\ 
\text{and}\ t_i t_{\rho(i+1)}=t_{i+1}t_{\rho(i)}\}\subseteq
F_{f_i\cdot\rho}
\cap \left\{M : (M^2)_{i,i+N} = (M^2)_{i+1,i+1+N}\right\}
=
X_{f_i\cdot\rho}\cap F_{f_i\cdot\rho}
\]
so that it is a subset of $Y_\rho$.
The rest of the proof is strictly identical to that of Prop.~3 of \cite{KZJ}
(or of Lemma \ref{lem:singorb} but without the added complication of
``irregular'' orbits)
and we shall only sketch it here.

We first compute the dimension of a single orbit via
that of the infinitesimal stabilizer of $\mathcal U_N$ on $\underline{f_i\cdot\rho}\,t$
and find $2n(n-1)$. We then check that each $\mathcal U_N$-orbit contains a unique
element of the form $\underline{f_i\cdot\rho}\,t$. Finally,
we compute
$\dim \{ \underline{f_i\cdot\rho}\,t, t\in T\ 
\text{and}\ t_i t_{\rho(i)}=t_{i+1}t_{\rho(i+1)}\}=2n-1$ and find
that the total dimension is $2n(n-1)+2n-1=2n^2-1=\dim Y_\rho$.
We conclude by irreducibility of $Y_\rho$.
\end{proof}


\begin{thebibliography}{10}

\newcommand\arxiv[1]{
\href{http://arxiv.org/abs/#1}{\tt arXiv:#1}}


\bibitem[BBM89]{BBM}  W. Borho, J.-L. Brylinski, R. MacPherson,
   Nilpotent orbits, primitive ideals, and characteristic classes,
   A geometric perspective in ring theory. Progress in Mathematics, 78.
   Birkh\"auser Boston, Inc., Boston, MA, 1989.

\bibitem[BS55]{BS} R. Bott, H. Samelson,
  The cohomology ring of $G/T$.
  Proc. Nat. Acad. Sci. U. S. A. 41 (1955), 490--493.

\bibitem[Br37]{Brauer} R. Brauer,
On algebras which are connected with the semisimple continuous groups,
Ann. Math. 38 (1937), 857--872.


\bibitem[Br97]{Br}  M. Brion,
  Equivariant cohomology and equivariant intersection theory,
  Notes by Alvaro Rittatore. NATO Adv. Sci. Inst. Ser. C Math. Phys. Sci., 514,
  Representation theories and algebraic geometry (Montreal, PQ, 1997), 1--37,
  \arxiv{math.AG/9802063}.

\bibitem[CG97]{CG} N. Chriss, V. Ginzburg,
   Representation theory and complex geometry,
   Birkh\"auser Boston, Inc., Boston, MA, 1997.
  
\bibitem[dGN05]{dGN} J. de Gier, B. Nienhuis,
  Brauer loops and the commuting variety,
  J. Stat. Mech. (2005) P01006,
  \arxiv{math.AG/0410392}.

\bibitem[DFZJ06]{DFZJ06} P. Di Francesco, P. Zinn-Justin,
  Inhomogeneous model of crossing loops 
  and multidegrees of some algebraic varieties,
  Commun. Math. Phys. 262 (2006), 459--487,
  \arxiv{math-ph/0412031}.

\bibitem[DFZJ05a]{DFZJ05}
  P.~Di Francesco and P.~Zinn-Justin,
  Around the Razumov--Stroganov conjecture: proof of a multi-parameter sum rule,
  E. J. Combi. 12 (1) (2005), R6,
  \arxiv{math-ph/0410061}.

\bibitem[DFZJ05b]{DFZJ05b} P. Di Francesco and P. Zinn-Justin, 
  Quantum Knizhnik--Zamolodchikov equation, 
  generalized Razumov--Stroganov sum rules and extended Joseph polynomials, 
  J. Phys. A 38 (2005) L815--L822, 
  \arxiv{math-ph/0508059}.

\bibitem[ER96]{ER} R. Ehrenborg and M. Readdy,
  Juggling and applications to $q$-analogues,
  Discrete Math. 157 (1996), no. 1-3, 107--125. 

\bibitem[Ei95]{Ei}
  D. Eisenbud, 
  Commutative algebra, with a view toward algebraic geometry, 
  Graduate Texts in Mathematics, vol. 150, Springer-Verlag, New York, 1995.

\bibitem[FK96]{FK}
S.~Fomin and A.N.~Kirillov,
The Yang--Baxter equation, symmetric functions and Schubert polynomials,
Discr. Math. 153 (1996) 123--143.

\bibitem[FR92]{FR}
I.~B.~Frenkel and N. Yu.~Reshetikhin,
Quantum affine algebras and holonomic difference equations,
Commun. Math. Phys. 146 (1992), 1--60.

\bibitem[Ho84]{Ho} R.~Hotta, 
  On Joseph's construction of Weyl group representations, 
  Tohoku Math. J. Vol. 36 (1984), 49--74.

\bibitem[Ji89]{Jimbo-review} M.~Jimbo,
Introduction to the Yang--Baxter equation,
 Int. J. Mod. Phys. A 4 (15) (1989), 3759--3777.

\bibitem[Jo84]{Jo} A.~Joseph, 
  {On the variety of a highest weight module}.
  J. Algebra 88 (1984), no.~1, 238--278.
  
\bibitem[Jo97]{Jo97} A. Joseph, 
  Orbital varieties, Goldie rank polynomials and unitary highest
  weight modules.  Algebraic and analytic methods in representation
  theory (S\o nderborg), Perspect.  Math., 17 (1994), 53--98.

\bibitem[KR88]{KR} V. G. Kac and A. K. Raina, 
  {Bombay lectures on highest weight representations of infinite
    dimensional Lie algebras,} (Lecture 9). Advanced Series in
  Mathematical Physics Vol. 2 (1988), World Scientific.

\bibitem[Kn05]{Kn}   A.~Knutson,
  Some schemes related to the commuting variety,
  J. Algebraic Geom.  14  (2005), 283--294, 
  \arxiv{math.AG/0306275}.

\bibitem[KM05]{KM} A. Knutson, E. Miller, 
  Gr\"obner geometry of Schubert polynomials,
  Annals of Mathematics 161 (2005) no. 3, 1245--1318,
  \arxiv{math.AG/0110058}.

\bibitem[KMY]{KMY} A. Knutson, E. Miller, A. Yong,
  Gr\"obner geometry of vertex decompositions and of flagged tableaux,
  Journal f\"ur die reine und angewandte Mathematik 630 (2009), 1--31,
  \arxiv{math.AG/0502144}

\bibitem[KS]{KempfCollapsing} A. Knutson, M. Shimozono,
  Kempf collapsing and quiver loci, 
  \arxiv{math.AG/0608327}
  
\bibitem[KZJ07]{KZJ} A. Knutson, P. Zinn-Justin,
  A scheme related to Brauer loops,
  Adv. in Math. 214 (2007), 40--77, 
  \arxiv{math.AG/0503224}.

\bibitem[La00]{lascoux} A.~Lascoux, 
  Transition on Grothendieck polynomials, Physics and combinatorics, 2000 
  (Nagoya.), 164--179, World Sci.~publishing, River Edge, NJ, 2001.

\bibitem[Mar02]{Mar02} J. Martin, 
  Combinatorial and geometric properties of graph varieties,
  Ph.D. thesis, University of California at San Diego, 2001.

\bibitem[MNR98]{MNR} M. J.~Martins, B.~Nienhuis and R.~Rietman, 
  An Intersecting Loop Model as a Solvable Super Spin Chain,
  Phys. Rev. Lett. 81 (1998) 504--507,
  \arxiv{cond-mat/9709051}.

\bibitem[MR97]{MR} M. J.~Martins and P. B.~Ramos, 
  The Algebraic Bethe Ansatz for rational braid-monoid lattice models,
  Nucl. Phys. B500 (1997) 579--620,
  \arxiv{hep-th/9703023}.

\bibitem[Me00]{M1}  A. Melnikov,
  $B$-Orbits in Solutions to the Equation $X^2 = 0$ in Triangular Matrices,
  Journal of Algebra 223, 101--108 (2000).

\bibitem[Me06]{M2} \bysame, 
  Description of $B$-orbit closure of order 2 in upper triangular matrices,
  Transformation Groups, Vol. 11 No. 2, 2006, 217--247,
  \arxiv{math.RT/0312290}.


\bibitem[MS04]{MS}   E.~Miller and B.~Sturmfels, 
  {Combinatorial commutative algebra},
  Graduate Texts in Mathematics, vol.~227, Springer--Verlag, New York, 2004.

\bibitem[Na99]{N} H. Nakajima, 
  Lectures on Hilbert schemes of points on surfaces,
  University Lecture Series, 18. American Mathematical Society, Providence, RI,
  1999.

\bibitem[RTVZJ12]{RTVZJ}  R.~Rim{\'a}nyi, V. Tarasov, A. Varchenko and P. Zinn-Justin,
Extended {J}oseph polynomials, quantized conformal blocks, and a {$q$}-{S}elberg type integral,
J. Geom. Phys. 62 (11) (2012), 2188--2207,
\arxiv{1110.2187}.

\bibitem[Sm86]{Smi}
F.~A.~Smirnov,
A general formula for soliton form factors in the quantum sine-Gordon model, 
J. Phys. A 19 (1986), L575--L578.

\bibitem[Sp76]{Sp} N. Spaltenstein,
  The fixed point set of a unipotent transformation on the flag manifold,
  Nederl. Akad. Wetensch. Proc. Ser. A Indag. Math. 38 no. 5 (1976), 452--456.

\bibitem[St97]{Stanley} R. Stanley,
  Enumerative Combinatorics 2, 2nd Edition.
  Cambridge University Press, 1997. \\
  Catalan interpretations in 
  \url{http://www-math.mit.edu/~rstan/ec/catalan.pdf}

\bibitem[Ro09]{R} B. Rothbach, 
  Borel orbits of $X^2 = 0$ in $\lie{gl}_n$, PhD thesis 2009. \\
  \url{http://search.proquest.com//docview/304845738}

\bibitem[Ro10]{rothbach} B.~Rothbach, 
  Equidimensionality of the Brauer loop scheme, 
  Electronic Journal of Combinatorics,
  Vol. 17 (2010), R75.
  \url{http://www.combinatorics.org/ojs/index.php/eljc/article/view/v17i1r75}

\end{thebibliography}
\end{document}